\documentclass[a4paper,11pt]{amsart}
\usepackage{amssymb}
\usepackage{amscd}
\usepackage{comment}
\usepackage{amsmath,amsthm}
\usepackage[colorlinks=true]{hyperref}
\usepackage{enumerate}
\usepackage{booktabs,multirow}
\usepackage{tikz}
\usepackage{rotating}
\usetikzlibrary{patterns}
\usetikzlibrary{decorations.pathreplacing}
\usetikzlibrary{calc,through}
\usetikzlibrary{cd}

\allowdisplaybreaks[1]
\numberwithin{figure}{section}
\numberwithin{equation}{section}

\title[Bijections on Dyck tilings]
{Bijections on Dyck tilings: DTS/DTR bijections, Dyck tableaux and tree-like tableaux}
\author[K.~Shigechi]{Keiichi~Shigechi}
\email{k1.shigechi AT gmail.com}
\date{\today}

\newcommand\tikzpic[2]{
\raisebox{#1\totalheight}{
\begin{tikzpicture}
#2
\end{tikzpicture}
}}
\newtheorem{theorem}[figure]{Theorem}
\newtheorem{example}[figure]{Example}

\newtheorem{defn}[figure]{Definition}
\newtheorem{prop}[figure]{Proposition}
\newtheorem{cor}[figure]{Corollary}

\newtheorem{remark}[figure]{Remark}
\begin{document}

\begin{abstract}
Dyck tilings are certain tilings in the region surrounded by two Dyck paths.
We study bijections and combinatorial objects bijective to Dyck tilings,
which include Dyck tiling strip (DTS) and Dyck tiling ribbon (DTR) bijections, 
increasing and decreasing trees, Hermite histories,
Dyck tableaux and tree-like tableaux.
Dyck tableaux and tree-like tableaux are originally defined for a zigzag path, or equivalently
a permutation. 
We generalize them to the case of general Dyck paths.
We show that the most properties of Dyck tableaux can be generalized to the generic case,
and show some enumerative results on generalized tree-like tableaux.
We also show connections among DTS and DTR bijections, Hermite histories, 
involutions on increasing and decreasing trees and the reflection of Dyck tilings. 
\end{abstract}

\maketitle

\section{Introduction}
Cover-inclusive Dyck tilings were introduced by Zinn-Justin and the author \cite{SZJ12}
in the study of Kazhdan--Lusztig polynomials for the Grassmannian permutations.
Independently, they appeared in the study of the so-called double-dimer model 
by Kenyon and Wilson \cite{KW11} (see also \cite{K12} for the proof of conjectures 
by Kenyon and Wilson). 
In both cases, we consider the partition function of Dyck tilings above a 
Dyck path. 
The difference of the two models is the weight given to a Dyck tiling.
Since then, Dyck tilings have arisen in different contexts: 
the double-dimer model and related physical models \cite{KW11,KW15,Pon18}, 
fully packed loop systems \cite{FN12}, representation of the symmetric group \cite{F16},
the pure partition function of multiple Schramm--Loewner evolutions \cite{KKP17,PelWu19}, 
and characterization of a basis of the intersection cohomology 
of Grassmannian Schubert varieties \cite{Pa19}. 

There are several combinatorial objects which are bijective to cover-inclusive 
Dyck tilings.
In \cite{KMPW12}, it was shown that increasing trees, perfect matchings and Hermite histories
are bijective to Dyck tilings through the Dyck tiling strip (DTS) bijection and 
the Dyck tiling ribbon (DTR) bijection.
Further, they also show that the weight of a Dyck tiling is compatible with statistics
of a permutation.
The DTR bijection for a Dyck tiling whose lower path is a zigzag path is equivalent 
to two other combinatorial objects: Dyck tableaux \cite{ABDH11} and tree-like tableaux \cite{ABN11}.
Dyck tableaux are regarded as a variant of another combinatorial object, permutation
tableaux \cite{B07,CN09,N11,P06,Vie07,W05}. 
They appear in the combinatorial description of the physical model PASEP (Partially 
Asymmetric Exclusion Process) \cite{CW072,CW071,CW073}.
Tree-like tableaux were introduced in \cite{ABN11} and bijective to permutation tableaux \cite{P06,SW07}
and alternative tableaux \cite{N11,Vie07}.
One of the advantages of Dyck tableaux and tree-like tableaux studied in \cite{ABDH11} and \cite{ABN11} 
is that they have nice recursive constructions and one can easily show bijections to other 
combinatorial objects mentioned before.
One of the purposes of this paper is to introduce and study generalized Dyck tableaux and 
generalized tree-like tableaux for general Dyck paths.
 
There are several generalizations of Dyck tilings: ballot tilings \cite{S142} and generalized Dyck 
tilings studied in \cite{JVK16}.
Ballot tilings \cite{S142} can be regarded as a type $B$ analogue of Dyck tilings since they naturally 
arise in the study of Kazhdan--Lusztig polynomials for an Hermitian symmetric pair \cite{Boe88,S141}.
The generalized Dyck tilings include $k$-Dyck tilings and symmetric Dyck tilings \cite{JVK16}.
There, the analogue of Bruhat order plays a role.
We remark that ballot tilings are a subset of symmetric Dyck tilings with certain conditions.
The result of this paper can be a starting point to generalize the notions of combinatorial
objects bijective to generalized cover-inclusive Dyck tilings.

We first revisit the relations among Dyck tilings, increasing trees, Hermite histories,
and DTS and DTR bijections.
We consider two types of labeling of a tree: a natural label and a weakly increasing label.
The former is a labeling of a tree in $[1,n]$ such that labels are strictly increasing 
from the root to a leaf in the tree.
The latter is a labeling such that labels are weakly increasing from the root to a leaf.
We utilize weakly increasing labels defined by Lascoux and Sch\"utzenberger in the 
study of Kazhdan--Lusztig polynomials for Grassmannian permutations \cite{LS81}.
There exists a simple bijection between these two labels.
We introduce a cover-inclusive Dyck tiling $D_{0}$ associated with a weakly increasing label, 
which is shown to be bijective to the Dyck tiling $D_{1}$ constructed from a natural label
by DTS bijection.
Further, we show that a Dyck tiling $D_{0}$ is compatible with the Hermite history of $D_{1}$.

Dyck tableaux \cite{ABDH11} and tree-like tableaux \cite{ABN11} are defined for a special class of 
Dyck paths called zigzag paths.
One of the main purposes of this paper is to generalize the notions of Dyck tableaux and tree-like 
tableaux to all Dyck paths.
A Dyck tableau of size $n$ is a skew Ferrers $\mu/\nu$ diagram with $n$ dots, where 
$\mu$ is a staircase partition.
Note that the boundary of a staircase partition $\mu$ is a zigzag path.
By relaxing the condition on $\mu$ such that the boundary of $\mu$ is a Dyck path,
we can construct a generalized Dyck tableau in a similar way as \cite{ABDH11}.
A tree-like tableau of size $n$ is a Ferrers diagram with $n$ dots (or points), 
which we call off-diagonal dots. 
By introducing a new class of dots, which we call diagonal dots, 
we generalize the notion of a tree-like tableau to that of a Dyck path.
We naturally have a bijection between generalized tree-like tableaux and Dyck 
tilings for general Dyck paths.  
For both generalized Dyck tableaux and generalized tree-like tableaux, 
we have recursive constructions.
We generalize most of the results in \cite{ABDH11} to the case of general Dyck paths with slight 
modifications.
It includes the insertion procedures for Dyck tableaux and weighted Dyck words, the bijection 
between Dyck tableaux and Dyck tilings, and the relations between generalized patterns for a 
natural label and a Dyck tableau.
We also show some enumerative results in case of generalized tree-like tableaux together 
with the insertion procedure.
The bijection between Dyck tableaux and tree-like tableaux is also given.

We have an operation called reflection on Dyck tilings, which reflects Dyck tilings 
along a vertical line.
We also have a natural operation called bar involution which maps a natural label to 
a decreasing natural label.
We study several relations among the DTS bijection, the reflection, the bar operation, and Hermite histories.
One of the relations is that a Dyck tiling constructed through an Hermite history is equal to 
a Dyck tiling by a variant of DTS bijection.

We introduce two operations on a natural label, $\ast$-operation and $\times$-operation.
These two operations are related by the bar operation.
The $\times$-operation on a natural label $L$ is characterized by the reflection of a Dyck 
tiling $D$ along a vertical line, where the Dyck tiling $D$ is constructed from the natural 
label $L$ by the DTR bijection.
Further, we study the action of $\ast$-operation on Dyck tilings above a certain special class of 
Dyck paths.
This special class of Dyck paths is a generalization of zigzag paths.
There, we show that the action of $\ast$-operation on Dyck tilings is controlled by the positions 
of dots in generalized Dyck tableaux.

When the lower path of a Dyck tiling is a zigzag path, we have two extreme Dyck tableaux, one of 
which is the one with all dots in the highest position, and the other is with all dots in 
the lowest position.
We give an algorithm to obtain extreme Dyck tableaux from a Dyck tableau by 
introducing the notions of skeleton and resolution.
A skeleton is a graph consisting of arches, which contains the information of decreasing sequences in 
a permutation, and a resolution is an operation to transform a skeleton to another skeleton 
with less intersections of arches. 
By successive applications of resolutions, we give an algorithm to obtain extreme Dyck tableaux 
from a generic Dyck tableau.

Given a decreasing natural label on a tree, we can construct a Dyck tiling through the Hermite 
history associated to the label.
We give an insertion algorithm to capture the top path of the Dyck tiling constructed 
by the Hermite history by introducing a new concept, Dyck bi-words.

This paper is divided into six sections.
In Section 2, we introduce the basic notions of Dyck tilings, Hermite histories, and 
DTS/DTR bijections.
In Section 3, we introduce a tree and its two labels: natural labels and weak increasing 
labels. Then, we study the relation between Dyck tilings $D_{1}$ and another 
Dyck tilings $D_{0}$ associated with the weakly increasing labels. 
Here, trees introduced by Lascoux and Sch\"utzenberger play a central role.
We also show that Dyck tilings $D_{0}$ are compatible with the Hermite history of $D_{1}$.
Then, we introduce an involution on Dyck tilings and study its relations to Dyck 
tilings $D_{0}$ and $D_{1}$.
Section 4 presents a generalization of Dyck tableaux for general Dyck paths.
We introduce weighted Dyck words generalized to general Dyck paths.
Then, we construct recursive algorithms for both Dyck tableaux and weighted 
Dyck words. 
We study three properties of generalized Dyck tableaux: generalized patterns, 
their shapes, and (LR/RL)-(minima/maxima).
In Section 5, we study tree-like tableaux for general Dyck paths.
We define the insertion procedure for the tableaux and observe the recursive
structure.
We present some enumerative results with respect to generalized tree-like tableaux.
We show that there exists a bijection between generalized Dyck tableaux and 
generalized tree-like tableaux.
Section 6 is devoted to the analysis of bijections characterizing Dyck tilings.
We first study the DTS bijection and some involutions on natural labels and 
on Dyck tilings.
Secondly, we move to the DTR bijection and involutions on Dyck tableaux.
Thirdly, we present two extreme Dyck tilings associated to a Dyck tableau of the 
same boundary paths.
Finally, we point out that there exists an insertion algorithm to identify 
the top path of the Dyck tiling constructed from an Hermite history.

\section{Dyck tilings}
\subsection{Dyck tilings}
We recall the definitions of Dyck tilings following \cite{KW11,SZJ12}.

A {\it Dyck path of length $2n$} is a lattice path from the origin $(0,0)$ 
to $(2n,0)$ with up steps $(1,1)$ and down steps $(1,-1)$, which does not 
go below the horizontal line $y=0$.
We write simply ``U" (resp. ``D") for an up (resp. down) step.
A sequence of ``U" and ``D" corresponding to a Dyck path is called a Dyck word.
We call the Dyck path $U\cdots UD\cdots D$ (resp. $UDUD\cdots UD$) 
the highest (resp. lowest) path.
We also denote by $|\lambda|$ the length of $\lambda$ ($|\lambda|=2n$).
When $\lambda$ can be expressed of a concatenation of two Dyck paths 
$\lambda_1$ and $\lambda_2$, 
we denote $\lambda=\lambda_1\circ\lambda_2$.
Here, concatenation means that we connect two Dyck paths one after another.
For example, when $\lambda_{1}=UDUD$ and $\lambda_{2}=UUDD$, a concatenation 
$\lambda_1\circ\lambda_2=UDUDUUDD$.

We introduce two special classes of Dyck paths: 
\begin{eqnarray*}
\mathrm{zigzag}_{n}&:=&UDUD\ldots UD=(UD)^{n}, \\ 
\wedge_{m}&:=&U\ldots UD\ldots D=U^{m}D^{m}.
\end{eqnarray*}
Given a sequence of positive integers $\mathbf{m}:=(m_1,\ldots,m_n)$,
we define a Dyck path 
$\wedge_{\mathbf{m}}:=\wedge_{m_1}\circ\ldots\circ\wedge_{m_n}$.
We call this class of Dyck paths generalized zigzag paths.

For later purpose, we also define a word $\vee_{n}$ by 
\begin{eqnarray*}
\vee_{n}=D\ldots DU\ldots U:=D^{n}U^{n}.
\end{eqnarray*}

A Dyck path $\lambda$ of length $2n$ can be identified with the Young diagram
which is determined by the path $\lambda$, the line $y=x$ and the line $y=-x+2n$.
Let $\lambda$ and $\mu$ be two Dyck paths.
If the skew shape $\lambda/\mu$ exists, we call $\lambda$ and $\mu$ 
the lower path and the upper (or top) path.

{\it A Dyck tile} is a ribbon (a connected skew shape which does not contain a $2\times2$ box)
such that the centers of the boxes form a Dyck path.
A single box is a Dyck tile of length $0$.
The size of a Dyck tile associated with a Dyck path of length $2n$ is $n$.
Let $\lambda$ and $\mu$ be two Dyck paths.
{\it A Dyck tiling} is a tiling of a skew Young diagram $\lambda/\mu$ by Dyck tiles.
A Dyck tiling $D$ is called {\it cover-inclusive} (resp. {\it cover-exclusive}) 
if we translate a Dyck tile of $D$ downward by $(0,-2)$ 
(resp. upward by $(0,2)$), then it is strictly below (resp. above) $\lambda$ (resp. $\mu$) 
or contained in another Dyck tile.

For a Dyck tiling $D$ of shape $\lambda/\mu$, we define $\mathrm{tiles}(D)$ to 
be the number of Dyck tiles in $D$, and $\mathrm{area}(D)$ to be the number 
of boxes in the skew shape $\lambda/\mu$.
We define 
\begin{eqnarray*}
\mathrm{art}(D):=(\mathrm{area}(D)+\mathrm{tiles}(D))/2.
\end{eqnarray*}

\subsection{Planted plane trees}
\label{sec:ppt}
Given a Dyck path $\lambda$, or equivalently a Dyck word
consisting of $U$'s and $D$'s, we define a planted plane tree 
$\mathrm{Tree}(\lambda)$
associated with $\lambda$
as follows:
\begin{enumerate}
\item $\mathrm{Tree}(\emptyset)$ is an empty tree.
\item Suppose that $\lambda$ is a concatenation of two Dyck words $\lambda_{1}$ 
and $\lambda_{2}$, {\it i.e.} $\lambda=\lambda_{1}\circ\lambda_{2}$.
Then, the tree $\mathrm{Tree}(\lambda)$ is obtained by attaching the two 
trees $\mathrm{Tree}(\lambda_{1})$ and $\mathrm{Tree}(\lambda_{2})$
at their roots. 
\item When $\lambda=U\lambda'D$ with a Dyck word $\lambda'$, the tree $\mathrm{Tree}(\lambda)$ is obtained 
by attaching an edge just above the tree $\mathrm{Tree}(\lambda')$.
\end{enumerate}
A top node is called the root of the tree, and a node which does not have edges below it is called 
as a leaf of the tree.
When an edge $e$ has several edges just below it, we call them children of $e$.

A planted plane tree $\mathrm{Tree}(\lambda)$ has a natural poset structure.
We have an order-preserving bijection $L: \mathrm{Tree}(\lambda)\rightarrow [n]$ 
where $n$ is the number of edges in $\mathrm{Tree}(\lambda)$ and 
$[n]:=\{1,2,\ldots,n\}$.
Here, order-preserving means that an integer on an edge $e_{1}$ of the tree is bigger 
than that on an edge $e_{2}$, where $e_{2}$ is just above $e_{1}$.
We call a planed plane tree with a natural labeling an increasing planted plane tree, 
or simply a natural label.

In the study of the Kazhdan--Lusztig polynomials for the Grassmannian permutations, 
Lascoux and Sch\"utzenberger introduced a weakly increasing planted plane tree \cite{LS81}.
Let $\lambda$ be a Dyck word and $\lambda_{0}$ is a word consisting of $U$ and $D$ 
such that the path $\lambda_{0}$ is above $\lambda$ and denote it by $\lambda_{0}\ge\lambda$. 
By definition, $\lambda_{0}$ is not necessarily a Dyck word but 
the first word of $\lambda_{0}$ is $U$.
For a word $\lambda$, we denote by $||\lambda||$ the length of the word 
and by $||\lambda||_{\alpha}$, $\alpha=U$ or $D$, the number of $\alpha$ 
in the word $\lambda$.
Suppose that $\lambda_{0}=\lambda_{0}'vw\lambda_{0}''$ 
and $\lambda=\lambda'UD\lambda''$ where $v,w\in\{U,D\}$ and 
$||\lambda_{0}'||=||\lambda'||$. 
Then, a {\it capacity} of the partial word corresponding to $UD$ in $\lambda$
is defined by 
\begin{eqnarray*}
\mathrm{cap}(UD):=||\lambda_{0}'v||_{U}-||\lambda'U||_{U}.
\end{eqnarray*}
The condition $\lambda_{0}\ge\lambda$ ensures that all capacities of $\lambda$
with respect to $\lambda_{0}$ is non-negative.

When the last word of $\lambda_{0}$ is $D$, we can have a path $\lambda_{0}$ 
satisfying $\lambda_{0}\ge\lambda$. 
We can similarly define the capacities by reflecting the paths by a vertical line.
More precisely, let $\lambda^{\mathrm{ref}}$ and $\lambda_{0}^{\mathrm{ref}}$ be 
the words obtained from $\lambda$ and $\lambda_{0}$ 
by exchanging $U$ and $D$ and reading from left to right.
The first letter of $\lambda_{0}^{\mathrm{ref}}$ is $U$.
We have capacities of $\lambda^{\mathrm{ref}}$ with respect to $\lambda_{0}^{\mathrm{ref}}$ 
as above. 

We put capacities of $\lambda$ with respect to $\lambda_{0}$ on the leaves of 
the tree $\mathrm{Tree}(\lambda)$. 
We denote by $\mathrm{Tree}(\lambda/\lambda_{0})$ the tree $\mathrm{Tree}(\lambda)$
with capacities.
The weakly increasing planted plane trees $\mathrm{Tree}(\lambda/\lambda_{0})$ 
satisfies 
\begin{enumerate}
\item Non-negative integers on edges are weakly increasing from the root to a leaf.
\item An integer on an edge connected to a leaf is less than or equal to 
its capacity.
\end{enumerate}

We denote by $T$ the planted plane tree $\mathrm{Tree}(\lambda)$. 
When an edge $e'$ is strictly right (resp. left) to the edge $e$, we denote this relation 
by $e\rightarrow e'$ (resp. $e'\leftarrow e$). Here, ``strictly" means that the edge $e'$ is positioned 
right to the edge $e$ and not positions in-between the edge $e$ and the root in $T$.
Similarly, when an edge $e'$ is weakly right (resp. left) to the edge $e$, 
we denote this relation by $e\Rightarrow e'$ (resp. $e'\Leftarrow e$).
Here, ``weakly right" means that $e'$ is strictly right to $e$ or positions in-between 
the edge $e$ and the root.
By definition, when $e\rightarrow e'$, $e$ satisfies also $e\Rightarrow e'$.
However, the reverse is not true.
When $e\Rightarrow e'$ but not $e\rightarrow e'$, we denote this relation 
by $e\uparrow e'$.
In other words, $e\uparrow e'$ means that $e'$ is positioned in-between 
the edge $e$ and the root in $T$.

\subsection{A Dyck tiling and an Hermite history}
\label{sec:DTHh}
In this subsection, we study the relation between a Dyck tiling and an Hermite history.
An Hermite history encodes the number of Dyck tiles and it is bijective to a perfect 
matching \cite{KMPW12}.
However, we consider another type of Hermite histories, which encodes the statistics $\mathrm{art}$ of Dyck 
tiles rather than the number of Dyck tiles.

Let $\lambda$ be a Dyck path of length $2n$ and $D$ is a cover-inclusive 
Dyck tiling over $\lambda$.
An Hermite history of $D$ is a collection of non-negative integers of length $n$.
These non-negative integers are associated with the down steps of $\lambda$.

Recall that a Dyck tile is a ribbon. 
We put a line in a Dyck tile from the left-most south-west edge to the right-most north-east 
edge.
We draw lines on all the Dyck tiles forming the Dyck tiling $D$.
Since $\lambda$ is a lowest path of $D$, each line starts from a down step $d$
of $\lambda$, or from a box which is placed upward by $(0,2m)$ for some $m$ 
from a down step $d$ of $\lambda$. 
In both cases, we say that a line is associated with the down step $d$.
We call a line associated with a down step a {\it trajectory}.

Let $D_{l}$ be the set of Dyck tiles which a trajectory $l$ passes through.
We define the weight $\mathrm{wt}$ of a trajectory $l$ on $D$ by 
\begin{eqnarray*}
\mathrm{wt}(l):=\sum_{t\in D_l}\mathrm{art}(t).
\end{eqnarray*}

Since a Dyck word of length $2n$ is a balanced word, we have $n$ down steps in it.
The collection of non-negative integers $\mathbf{H}:=(H_1,\ldots,H_n)$
is defined by 
\begin{eqnarray*}
H_{i}:=\mathrm{wt}(i),
\end{eqnarray*}
where the trajectory $i$ is associated with the $i$-th down step of $\lambda$.
If there is no trajectory associated with the $i$-th down step, we define 
$H_{i}:=0$.

See Figure \ref{fig:Hh} for an example of a Dyck tiling and its Hermite 
history.
\begin{figure}[ht]
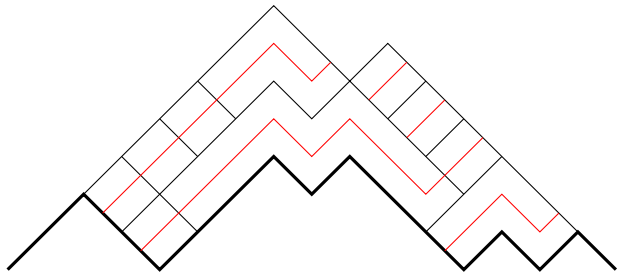

\tikzpic{-0.5}{[x=0.5cm,y=0.5cm]
\draw[very thick](0,0)--(2,2)--(4,0)--(7,3)--(8,2)--(9,3)--(12,0)--(13,1)--(14,0)--(15,1)--(16,0);
\draw(2,2)--(7,7)--(9,5)--(10,6)--(15,1)(3,1)--(7,5)--(8,4)--(9,5)--(12,2);
\draw(3,3)--(5,1)(4,4)--(5,3)(5,5)--(6,4)(11,5)--(10,4)(12,4)--(11,3)(13,3)--(11,1);
\draw[red](2.5,1.5)--(7,6)--(8,5)--(8.5,5.5)
          (3.5,0.5)--(7,4)--(8,3)--(9,4)--(11,2)--(12.5,3.5)
          (9.5,4.5)--(10.5,5.5)(10.5,3.5)--(11.5,4.5)
          (11.5,0.5)--(13,2)--(14,1)--(14.5,1.5);
}

\caption{A cover-inclusive Dyck tiling with the Hermite 
history $\mathbf{H}=(5,6,0,1,1,2,0,0)$}
\label{fig:Hh}
\end{figure}

\subsection{Bijections: Dyck tiling strip and Dyck tiling ribbon}
\label{sec:DTSDTR}
Let $\lambda$ be a Dyck path of length $2n$.
Since $\lambda$ is a lattice path from the origin $(0,0)$ to $(2n,0)$, 
we have a unique intersection between 
$\lambda$ and the line $x=m$ for $0\le m\le 2n$.
We divide a Dyck tiling $D$ over $\lambda$ into two pieces by the line $x=m$.
Then, we move the right piece to the right by $(2,0)$.
We reconnect two pieces by paths $UD$.
We call this operation the {\it spread} of $D$ at $x=m$.
When a single box is divided into two pieces by $x=m$, we obtain 
a Dyck tile of length $1$.
Similarly, if a Dyck tile of length $r$ is divided into two pieces by $x=m$,
we obtain a Dyck tile of length $r+1$ (See Figure \ref{fig:spread} for examples).
\begin{figure}[ht]
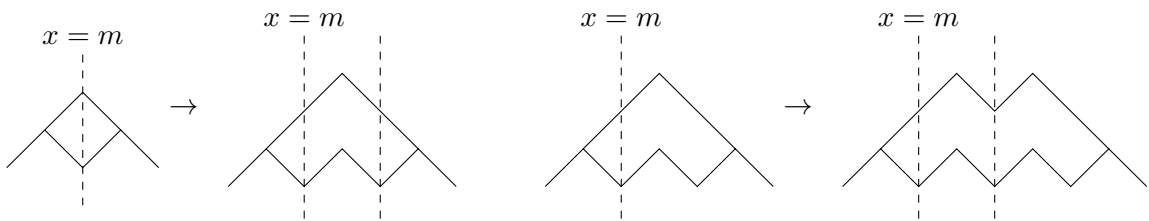

\tikzpic{-0.5}{[x=0.5cm,y=0.5cm]
\draw(-1,-1)--(0,0)--(1,1)--(2,0)--(1,-1)--(0,0)(2,0)--(3,-1);
\draw[dashed](1,2)--(1,-2);
\node[anchor=south] at(1,2){$x=m$};
}$\rightarrow$
\tikzpic{-0.5}{[x=0.5cm,y=0.5cm]
\draw(-1,-1)--(0,0)--(2,2)--(4,0)--(3,-1)--(2,0)--(1,-1)--(0,0)(4,0)--(5,-1);
\draw[dashed](1,3)--(1,-2)(3,3)--(3,-2);
\node[anchor=south] at(1,3){$x=m$};
}\qquad
\tikzpic{-0.5}{[x=0.5cm,y=0.5cm]
\draw(-1,-1)--(0,0)--(2,2)--(4,0)--(3,-1)--(2,0)--(1,-1)--(0,0)(4,0)--(5,-1);
\draw[dashed](1,3)--(1,-2);
\node[anchor=south] at(1,3){$x=m$};
}$\rightarrow$
\tikzpic{-0.5}{[x=0.5cm,y=0.5cm]
\draw(-1,-1)--(0,0)--(2,2)--(3,1)--(4,2)--(6,0)--(5,-1)--(4,0)--(3,-1)--(2,0)--(1,-1)--(0,0)
     (6,0)--(7,-1);
\draw[dashed](1,3)--(1,-2)(3,3)--(3,-2);
\node[anchor=south] at(1,3){$x=m$};
}
\caption{Examples of the operation spread}
\label{fig:spread}
\end{figure}

Given a natural label $T_1$, we denote by $e(i)$ the edge labeled by $i$, 
and by $n(e)$ the label of the edge $e$.
Let $n+1$ be the number of edges in $T_1$.
we define a sequence $\mathbf{h}:=(h_1,\ldots,h_{n+1})$ of non-negative integers by 
\begin{eqnarray*}
h_i:=2\cdot\#\{e' | e'\leftarrow e(i), n(e')<i \}+\#\{e' | e(i)\uparrow e'\}.
\end{eqnarray*}
By definition, we have $h_{i}\in[0,2(i-1)]$.
We call $\mathbf{h}$ an {\it insertion history}.
Note that a sequence $\mathbf{h}$ is bijective to a natural label $T_{1}$.

We will recursively construct two bijections between Dyck tilings over $\lambda$ and 
non-negative integer sequences $\mathbf{h}$.
When we have a Dyck tiling $D$ of length $2n$ associated with a sequence $(h_1,\ldots,h_n)$,
we perform the spread of $D$ at $x=h_{n+1}$.
We denote by $\widetilde{D}$ the new Dyck tiling obtained from $D$ by the spread, and 
simply write the spread of $D$ at $x=m$ as $\mathrm{sp}_{m}(D):=\widetilde{D}$.

We define two processes on Dyck tilings: the strip-growth and the ribbon-growth.
The former is used to define the Dyck tiling strip (DTS) bijection and 
the latter is for the Dyck tiling ribbon (DTR) bijection. 
We define the DTS and DTR bijections following \cite{KMPW12}.

\paragraph{\bf The strip-growth}
Given a Dyck tiling $\widetilde{D}$, we add a vertical strip to $\widetilde{D}$ 
in the north-east direction right to the line $x=h_{n+1}$. 
In other words, we attach a single box for each up step in $\widetilde{D}$ such 
that the up step is right to the line $x=h_{n+1}$.
We denote by $\mathrm{SG}(\widetilde{D})$ the new Dyck tiling obtained from 
$\widetilde{D}$ by the strip-growth.

\paragraph{\bf The ribbon-growth}
To define the ribbon-growth, we first introduce the notion of the special column
of a Dyck tiling following \cite{KMPW12}.
Let $\mu$ be the top path of a Dyck tiling $\widetilde{D}$.
The special column of a Dyck tiling is the right-most column $s$ satisfying 
\begin{enumerate}
\item The top path $\mu$ contains an up step which ends in column $s$.
\item The intersection of $\mu$ with the line $x=s$ is not the top corner 
of a Dyck tile of $\widetilde{D}$ consisting of a single box.
\end{enumerate}
A column satisfying above two conditions are said to be eligible.
Thus, the special column is the right-most eligible column.
Given a Dyck tiling $\widetilde{D}$, we add a ribbon consisting of single boxes 
in the north-east direction right to the line $x=h_{n+1}$ up to the line $x=s$.
We denote by $\mathrm{RG}(\widetilde{D})$ the new Dyck tiling obtained from 
$\widetilde{D}$ by the ribbon-growth.

Let $T_{1}'$ be a natural label obtained from $T_{1}$ by deleting the edge with 
the label $n+1$.
We define the DTS and DTR bijections recursively by
\begin{eqnarray*}
DTS(T_1)&:=&\mathrm{SG}(\widetilde{D})=\mathrm{SG}(\mathrm{sp}_{h_{n+1}}(DTS(T_{1}'))), \\
DTR(T_1)&:=&\mathrm{RG}(\widetilde{D})=\mathrm{RG}(\mathrm{sp}_{h_{n+1}}(DTR(T_{1}'))).
\end{eqnarray*}

Figure \ref{fig:inctreeDTSDTR} gives an example of a labeled tree, the DTS and the DTR bijections.
\begin{figure}[ht]
\tikzpic{-0.5}{[scale=0.8]
\coordinate
       child{coordinate (c5)}
       child{coordinate (c2)
            child{coordinate (c6)}
            child{coordinate (c3)
                 child[missing]
                 child{coordinate (c4)}
                 }
            }
       child{coordinate (c1)};
\node[anchor=south west] at ($(0,0)!.5!(c1)$){1};
\node[anchor=east]at($(0,0)!.6!(c2)$){2};
\node[anchor=south east]at($(0,0)!.6!(c5)$){5};
\node[anchor=south east]at($(c2)!.6!(c6)$){6};
\node[anchor=south west]at($(c2)!.6!(c3)$){3};
\node[anchor=south west]at($(c3)!.6!(c4)$){4};
\node at(c3){$-$};
}\quad
\tikzpic{-0.5}{[x=0.4cm,y=0.4cm]
\draw[very thick](0,0)--(1,1)--(2,0)--(4,2)--(5,1)--(7,3)--(10,0)--(11,1)--(12,0);
\draw(1,1)--(6,6)--(11,1)(2,2)--(3,1)(4,4)--(6,2)(5,5)--(7,3)--(8,4)(5,3)--(7,5);
\draw(8,2)--(9,3)(9,1)--(10,2);
}\quad
\tikzpic{-0.5}{[x=0.4cm,y=0.4cm]
\draw[very thick](0,0)--(1,1)--(2,0)--(4,2)--(5,1)--(7,3)--(10,0)--(11,1)--(12,0);
\draw(1,1)--(4,4)--(6,2)(5,3)--(6,4)--(7,3)(2,2)--(3,1);
\draw(9,1)--(10,2)--(11,1);
}
\caption{A natural label (the left picture), DTS bijection (the middle picture) and 
DTR bijection (the right picture)}
\label{fig:inctreeDTSDTR}
\end{figure}

We introduce three types of variants of the DTS bijection:
the left DTS (lDTS) bijection, the reverse order DTS (rDTS) bijection,
and the reverse order left DTS (rlDTS) bijection.
The DTS bijection is a map from an increasing tree to a Dyck tiling, and 
its addition of a vertical strip is right to the insertion point.
The lDTS bijection is the DTS bijection such that the addition of a vertical 
strip is left to the insertion point. Thus, lDTS is a map from an increasing 
tree to a Dyck tiling.
The rDTS bijection is the DTS bijection such that the bijection is a map 
from a decreasing tree to a Dyck tiling, insertion order is according to 
the decreasing order of labels, and the addition of a vertical strip is 
right to the insertion point.
The rlDTS bijection is a map from a deceasing tree to a Dyck tiling as 
the rDTS bijection, but the addition of a vertical strip is left to the 
insertion point.

See Figure \ref{fig:vaDTS} for an example of these DTS bijections.
\begin{figure}[ht]
\begin{eqnarray*}
\tikzpic{-0.5}{[scale=0.6]
\coordinate
	child{coordinate(c3)}
	child{coordinate(c1)
		child{coordinate(c4)}
		child{coordinate(c2)}};
\draw($(0,0)!0.5!(c3)$)node[anchor=south east]{$3$}($(0,0)!0.5!(c1)$)node[anchor=south west]{$1$};
\draw($(c1)!0.5!(c4)$)node[anchor=south east]{$4$}($(c1)!0.5!(c2)$)node[anchor=south west]{$2$};
}\qquad
\tikzpic{-0.5}{[x=0.4cm,y=0.4cm]
\draw[very thick](0,0)--(1,1)--(2,0)--(4,2)--(5,1)--(6,2)--(8,0);
\draw(1,1)--(4,4)--(6,2)(2,2)--(3,1);
}\qquad
\tikzpic{-0.5}{[x=0.4cm,y=0.4cm]
\draw[very thick](0,0)--(1,1)--(2,0)--(4,2)--(5,1)--(6,2)--(8,0);
\draw(1,1)--(2,2)--(3,1);
} \\
\tikzpic{-0.5}{[scale=0.6]
\coordinate
	child{coordinate(c3)}
	child{coordinate(c4)
		child{coordinate(c2)}
		child{coordinate(c1)}};
\draw($(0,0)!0.5!(c3)$)node[anchor=south east]{$3$}($(0,0)!0.5!(c4)$)node[anchor=south west]{$4$};
\draw($(c4)!0.5!(c2)$)node[anchor=south east]{$2$}($(c4)!0.5!(c1)$)node[anchor=south west]{$1$};
}\qquad
\tikzpic{-0.5}{[x=0.4cm,y=0.4cm]
\draw[very thick](0,0)--(1,1)--(2,0)--(4,2)--(5,1)--(6,2)--(8,0);
\draw(1,1)--(2,2)--(3,1);
}\qquad
\tikzpic{-0.5}{[x=0.4cm,y=0.4cm]
\draw[very thick](0,0)--(1,1)--(2,0)--(4,2)--(5,1)--(6,2)--(8,0);
\draw(1,1)--(2,2)--(3,1)(4,2)--(5,3)--(6,2)(2,2)--(3,3)--(4,2);
}
\end{eqnarray*}
\caption{An increasing tree (left picture in the first row), its DTS bijection 
(middle picture in the first row) and its lDTS bijection (right picture in the first row).
The second row is an example of a decreasing tree (left picture), the rDTS bijection (middle picture),
and rlDTS bijection (right picture).}
\label{fig:vaDTS}
\end{figure}

Similarly, we define the reverse order left DTR (rlDTR for short) as the DTR bijection from a 
decreasing tree to a Dyck tiling such that the ribbon growth is to the left of the 
insertion point for the spread.

\section{Bijections associated with an increasing tree}
In this section, we consider bijections among increasing trees,
two weakly increasing trees, two Dyck tilings associated with the weakly increasing trees,  
and two Hermite histories.
Let $\lambda$ be a Dyck path of length $2n$.

\subsection{Bijection between an increasing tree and a weakly increasing tree}
\label{sec:bijwit}
We denote by $T$ the planted plane tree $\mathrm{Tree}(\lambda)$ and 
$T_{1}$ be a natural label of $T$.
Given an edge $e$ of $T$, we denote by $T_{1}(e)$ the label of the edge $e$.

We construct a weakly increasing tree $T_{2}$ as follows.
A label $T_{2}(e)$ of an edge $e$ is equal to 
\begin{eqnarray}
\label{T2}
T_{2}(e):=\#\{e'| T_{1}(e)>T_{1}(e'), e\rightarrow e'\}.
\end{eqnarray}
By construction, the maximum value of $T_{2}(e)$ is nothing but the 
number of $e'$'s satisfying $e\rightarrow e'$ in $T$.
When an edge $e_{1}$ is the parent of an edge $e_{2}$ ($e_{1}$ is just above $e_{2}$), 
we have $T_{2}(e_{1})\le T_{2}(e_{2})$.

\begin{figure}[ht]
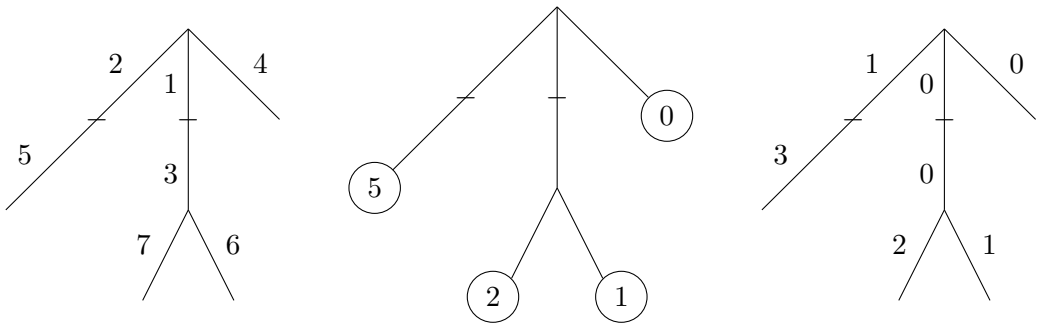

\tikzpic{-0.5}{[scale=0.8]
\coordinate 
 	child{coordinate (c2)
	     child{coordinate (c5)}
             child[missing]
             child[missing]
	     }
	child{coordinate (c1)
             child{coordinate (c3)
                  child{coordinate (c7)}
                  child{coordinate (c6)}
		  }
             }
	child{coordinate (c4)};
\draw(c2)node{$-$}(c1)node{$-$};
\node[anchor=south east] at ($(0,0)!.6!(c2)$){2};
\node[anchor=east] at ($(0,0)!.6!(c1)$){1};
\node[anchor=south west] at ($(0,0)!.6!(c4)$){4};
\node[anchor=south east] at ($(c2)!.6!(c5)$){5};
\node[anchor=east] at ($(c1)!.6!(c3)$){3};
\node[anchor=south east] at ($(c3)!.6!(c7)$){7};
\node[anchor=south west] at ($(c3)!.6!(c6)$){6};
}
\quad
\tikzpic{-0.5}{[scale=0.8]
\coordinate 
 	child{coordinate (c2)
	     child{coordinate (c5)}
             child[missing]
             child[missing]
	     }
	child{coordinate (c1)
             child{coordinate (c3)
                  child{coordinate (c7)}
                  child{coordinate (c6)}
		  }
             }
	child{coordinate (c4)};
\draw(c2)node{$-$}(c1)node{$-$};
\node[circle,draw,fill=white] at (c5){5};
\node[circle,draw,fill=white,anchor=north east] at (c7){2};
\node[circle,draw,fill=white,anchor=north west] at (c6){1};
\node[circle,draw,fill=white,anchor=north west] at (c4){0};
}
\quad
\tikzpic{-0.5}{[scale=0.8]
\coordinate 
 	child{coordinate (c2)
	     child{coordinate (c5)}
             child[missing]
             child[missing]
	     }
	child{coordinate (c1)
             child{coordinate (c3)
                  child{coordinate (c7)}
                  child{coordinate (c6)}
		  }
             }
	child{coordinate (c4)};
\draw(c2)node{$-$}(c1)node{$-$};
\node[anchor=south east] at ($(0,0)!.6!(c2)$){1};
\node[anchor=east] at ($(0,0)!.6!(c1)$){0};
\node[anchor=south west] at ($(0,0)!.6!(c4)$){0};
\node[anchor=south east] at ($(c2)!.6!(c5)$){3};
\node[anchor=east] at ($(c1)!.6!(c3)$){0};
\node[anchor=south east] at ($(c3)!.6!(c7)$){2};
\node[anchor=south west] at ($(c3)!.6!(c6)$){1};
}

\caption{A natural label of a tree associated with a Dyck path $\lambda=UUDDUUUDUDDDUD$
(the left picture), a tree with capacities (the middle picture), and 
the weakly increasing tree (the right picture).
}
\label{fig:inKLwin}
\end{figure}

Let $\lambda_{0}=D\ldots D$ be a Dyck path. 
The right-most steps of $\lambda_{0}$ and $\lambda$ are overlapped. 
Thus, we have $\lambda_{0}\ge\lambda$.
As in Section \ref{sec:ppt}, we consider a weakly increasing tree $\mathrm{Tree}(\lambda/\lambda_{0})$.
Suppose that $e$ is an edge connected to a leaf. Then, by the choice of $\lambda_{0}$, 
the capacity of the edge $e$ is equal to the number of $e'$ satisfying $e\rightarrow e'$ in $T$.

From above considerations, we have a bijection between increasing trees $T_{1}$ and 
weakly increasing trees $T_{2}$ (or equivalently weakly increasing trees 
of $\mathrm{Tree}(\lambda/\lambda_{0})$).
See Figure \ref{fig:inKLwin} for an example of a natural label $L$ and the weakly increasing 
tree associated with $L$.

The two paths $\lambda$ and $\lambda_{0}$, and boxes on the vertical line $x=0$ 
form a region $R$ in the Cartesian coordinate.
We construct a bijection between weakly increasing trees $T_{2}$ and conver-inclusive 
Dyck tilings in the region $R$.
We generalize the identification studied in \cite{SZJ12}, which connects a cover-inclusive 
Dyck tiling and a weakly increasing tree of Lascoux and Sch\"utzenberger.
An edge $e$ in the tree $\mathrm{Tree}(\lambda)$ corresponds to a pair of $U$ and $D$ steps 
in the lowest path $\lambda$. Especially, an edge connected to a leaf corresponds to 
a pair of $U$ and $D$ steps next to each other in $\lambda$. 
In a weakly increasing tree $T_{2}$, when an edge $e$ has an integer $T_{2}(e)$, 
we have $T_{2}(e)$ non-trivial Dyck tiles (not a single box) over the pair of $U$ and $D$.
Recall that a label $T_2(e)$ is less than or equal to a capacity. 
In the region $R$, we can put non-trivial Dyck tiles over the pair of $U$ and $D$ steps
up to its capacity by the choice of $\lambda_{0}$.
Since the labels of edges from the root to a leaf are weakly increasing in $T_{2}$, 
we obtain an cover-inclusive Dyck tiling in the region $R$. 

We call left-most boxes in the region $R$ as anchor boxes and enumerate them from 
top to bottom by $0$ to $n-1$, where $n$ is the number of edges of $\mathrm{Tree}(\lambda)$. 
See Figure \ref{fig:DTRegion} for an example of a cover-inclusive Dyck tiling 
in the region $R$ above a Dyck path $\lambda$.

Let $D$ be a Dyck tiling in the region $R$ corresponding to a weakly increasing tree $T_{2}$.
We consider an Hermite history, which consists of $n$ trajectories.
Here, a trajectory is a line on Dyck tiles starting from an up step of $\lambda$ and ending 
at an anchor box.
More precise definition of a trajectory is as follows.
For a Dyck tile, we call the rightmost southeast edge {\it entry} and 
the leftmost northwest edge {\it exit}.
A trajectory on a Dyck tile connects the entry and the exit by a line.
We concatenate trajectories of $D$ if and only if the entry of a Dyck tile is attached 
to the exit of another Dyck tile.
Some of trajectories start from the up steps of $\lambda$ and other trajectories start from 
Dyck tiles which is not attached to the lowest path $\lambda$.
We translate a trajectory downward and if the rightmost entry of the trajectory is on an 
up step $U$ of $\lambda$, we say that this trajectory is associated with the up step $U$.
The Hermite history is a collection of trajectories associated with the up steps of $\lambda$.

\begin{figure}[ht]
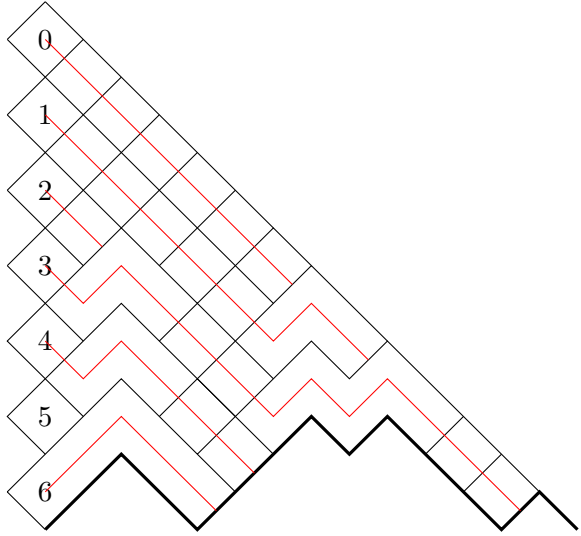

\tikzpic{-0.5}{[x=0.5cm,y=0.5cm]
\draw[very thick](0,0)--(2,2)--(4,0)--(7,3)--(8,2)--(9,3)--(12,0)--(13,1)--(14,0);
\draw(-1,1)--(0,0)(-1,1)--(2,4)--(4,2)--(7,5)--(8,4)--(9,5)--(13,1);
\draw(0,2)--(-1,3)--(0,4)--(2,6)--(4,4)--(7,7)--(9,5);
\draw(0,4)--(-1,5)--(0,6)--(2,8)--(4,6)--(6,8)--(7,7);
\draw(3,7)--(5,9)(2,8)--(4,10)(-1,7)--(3,11)(-1,9)--(2,12)(-1,11)--(1,13)(-1,13)--(0,14);
\draw(0,14)--(6,8)(-1,13)--(5,7)(-1,11)--(2,8)(-1,9)--(1,7)(-1,7)--(1,5)(0,4)--(1,3);
\draw(4,2)--(5,1)(4,4)--(6,2)(11,3)--(10,2)(12,2)--(11,1);
\draw(3,3)--(4,4)--(5,3)(5,5)--(6,4);
\draw(3,5)--(4,6)--(5,5)(5,7)--(6,6);
\foreach \x in {6,5,...,0} {\node at (0,13-2*\x) {\x};}
\draw[red](0,1)--(2,3)--(4.5,0.5)
          (0,5)--(1,4)--(2,5)--(5.5,1.5)
          (0,7)--(1,6)--(2,7)--(6,3)--(7,4)--(8,3)--(9,4)--(12.5,0.5)
          (0,9)--(1.5,7.5)
          (0,11)--(6,5)--(7,6)--(8.5,4.5)
          (0,13)--(6.5,6.5);
}
\caption{The cover-inclusive Dyck tiling in the region $R$ associated with the natural label
in Figure \ref{fig:inKLwin}. The red lines are the trajectories of the Hermite history.}
\label{fig:DTRegion}
\end{figure}

Suppose that a trajectory in an Hermite history connects the $i$-th up step from right 
and an anchor box with a label $a_{i}\in\{0,1,\ldots,n-1\}$. 
We define a sequence of integers $\mathbf{b}:=(b_{1},\ldots,b_{n})$ by 
\begin{eqnarray*}
b_{i}:=a_{i}-\#\{a_{j}: j<i, a_{j}<a_{i}\}.
\end{eqnarray*}
By definition, we have $b_{i}\ge0$.

Below, we construct a cover-inclusive Dyck tiling over $\lambda$ from 
the sequence of integers $\mathbf{b}$.
We identify $\mathbf{b}$ as a sequence of integers associated with 
an Hermite history starting from the up steps of $\lambda$.
Suppose that the $i$-th up step from right has a trajectory which 
passes through Dyck tiles $d_{j}, 1\le j\le m$.
Then, we have 
\begin{eqnarray*}
b_{i}=\sum_{j=1}^{m}\mathrm{art}(d_{j}).
\end{eqnarray*}

\begin{example}
For the cover-inclusive Dyck tiling in Figure \ref{fig:DTRegion}, we have 
\begin{eqnarray*}
\mathbf{a}&=&(3,1,0,4,6,2,5), \\
\mathbf{b}&=&(3,1,0,1,2,0,0).
\end{eqnarray*}
By use of the correspondence between a Dyck tiling and an Hermite history, 
we have the following cover-inclusive Dyck tiling:
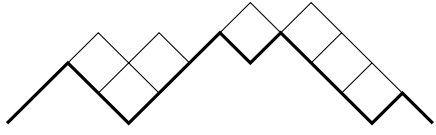
\begin{center}
\tikzpic{-0.5}{[x=0.4cm,y=0.4cm]
\draw[very thick](0,0)--(2,2)--(4,0)--(7,3)--(8,2)--(9,3)--(12,0)--(13,1)--(14,0);
\draw(2,2)--(3,3)--(5,1)(3,1)--(5,3)--(6,2)(7,3)--(8,4)--(9,3)--(10,4)--(13,1)
     (10,2)--(11,3)(11,1)--(12,2);
}
\end{center}
\end{example}

\begin{prop}
The above construction is well-defined. 
In other words, $\mathbf{b}$ defines an Hermite history on a cover-inclusive 
Dyck tiling above $\lambda$.
\end{prop}
\begin{proof}
Let $U_{i}$ and $U_{i+1}$ be the $i$-th and $i+1$-th up steps from the right end in $\lambda$.
We have two cases: a) $U_{i}$ and $U_{i+1}$ are next to each other, and 
b) there exist down steps between $U_{i}$ and $U_{i+1}$.

First, we consider the case a).
In a cover-inclusive Dyck tiling $D_{R}$ in the region $R$, we have 
\begin{enumerate}
\item 
\label{DR1}
trajectories of the Hermite history of $D_{R}$ are non-intersecting, 
\item 
\label{DR2}
the labels of anchor boxes increase one-by-one from top to bottom, 
\item 
\label{DR3}
the trajectory for $U_{i}$ starts above the trajectory for $U_{i+1}$.
\end{enumerate}
Here, the third property comes from the following fact: 
if the starting point of the trajectory of $U_{i+1}$ is lifted upward by 
a trajectory of $U_{j}$ with $j<i$, the starting point of the 
trajectory of $U_{i}$ should also be lifted upward.
From properties from (\ref{DR1}) to (\ref{DR3}) of $D_{R}$ in $R$, 
we have $a_{i}<a_{i+1}$, which implies $b_{i}\le b_{i+1}$.
This condition is admissible as a condition for the Hermite history.

We consider the case b).
Let $n_{0}$ be the label of the anchor box which is in the 
north-west direction from the step $U_{i}$.
Let $N_{1}$ be the maximal integer such that there is no down 
step between $U_{i+1}$ and $U_{i+N_{1}}$, and 
$M_{1}$ be the number of down steps between $U_{i}$ and $U_{i+1}$.
The partial path around $U_{i}$ and $U_{i+1}$ is depicted as follows:
\begin{eqnarray}
\label{eqn:HhDT}
\tikzpic{-0.5}{[x=0.4cm,y=0.4cm]
\draw(-0.2,0.2)--(1,-1)--(2.2,0.2)(3.4,1.4)--(4.6,2.6)--(5.8,1.4)(7,0.2)--(8.2,-1);
\draw(8.2,-1)--(9.4,0.2);
\draw(2,0)node{$-$}(3.6,1.6)node{$-$}(0,0)node{$-$}(5.6,1.6)node{$-$}(7.2,0)node{$-$}
     (9.2,0)node{$-$};
\draw[dashed](2.2,0.2)--(3.4,1.4)(5.8,1.4)--(7,0.2)(-1,0.8)--(-0.2,0.2)(9.2,0)--(10,0.8);
\draw(4.1,2.3)node[anchor=south east]{$U_{i+1}$};
\draw(1.5,-0.5)node[anchor=north west]{$U_{i+N_{1}}$};
\draw(8.7,-0.5)node[anchor=north west]{$U_{i}$};
\draw[decoration={brace,mirror,raise=5pt},decorate](8.1,-0.9)--(4.6,2.6);
\draw(7.5,2)node{$M_{1}$};
}
\end{eqnarray}

Suppose that the north-west box of the edge $U_{i}$ is contained by a Dyck tile, 
which is right to $U_{i}$ and of length larger than zero.
Then, the trajectory of this tile is associated to an up step $U_{h}$ for $h<i$.

If the trajectory associated to $U_{h}$ does not contain the north-west box of the edge $U_{i+1}$,
the trajectory associated to $U_{h}$ is positioned between the trajectories 
of $U_{i}$ and $U_{i+1}$.
We have $a_{i}<a_{i+1}$, or equivalently, $b_{i}\le b_{i+1}$, which implies 
that this configuration is admissible as an Hermite history.

If the trajectory associated to $U_{h}$ contains the north-west box of the edge $U_{i+1}$ as 
a Dyck tile of length larger than zero, the starting points of the trajectories associated to
$U_{i}$ and $U_{i+1}$ are moved upward by $(0,2)$.
The local configuration around $U_{i}$ and $U_{i+1}$ looks the same as Eqn. (\ref{eqn:HhDT})
except that we have the trajectory associated to $U_{h}$ below the partial path.

From above observations, one can assume that the north-west box of $U_{i}$ is contained 
in the trajectory associated to $U_{i}$ and local configuration is as Eqn. (\ref{eqn:HhDT})
without loss of generality. 
We have two cases: b1) $N_{1}\le M_{1}$, and b2) $N_{1}>M_{1}$.

\paragraph{\bf Case b1).} 
Since we have a cover-inclusive Dyck tiling in $R$, one may have 
a Dyck tile of length $l$ at the $\wedge$-corner whose left edge is $U_{i+1}$.
We have three cases for the length $l$: i) $0\le l\le N_{1}$, ii) $N_{1}<l\le M_{1}$, 
and iii) $M_{1}<l$.

\paragraph{ Case i).}
When $l=0$, it is obvious that the trajectory associated $U_{i}$ is above the trajectory 
associated to $U_{i+1}$. 
We have $b_{i}\le b_{i+1}$, which implies that the Hermite history is admissible.

We consider $l>0$. 
The starting points of trajectories starting from the steps 
from $U_{i+1}$ to $U_{i+l}$ is above the trajectory starting from 
$U_{i}$.
By the same argument as the case a), we have 
\begin{eqnarray*}
a_{i+1}<a_{i+2}<\cdots<a_{i+l}<a_{i}<a_{i+l+1}<\cdots<a_{i+N_{1}},
\end{eqnarray*} 
which implies 
\begin{eqnarray*}
b_{i+1}\le b_{i+2}\le\cdots\le b_{i+l}\le b_{i}\le b_{i+l+1}.
\end{eqnarray*}
We also have $a_{i}\ge n_{0}+l$ since there exist at least $l$ trajectories 
above the one of $U_{i}$.
If the step $U_{i}$ is connected to the anchor box with $n_{0}$ in $R$, we have 
$n_{0}$ up steps right to $U_{i}$.
From these observations, we have $b_{i}\ge l$ with $0\le l\le N_{1}$.
To regard $\mathbf{b}$ with an Hermite history of a Dyck tiling $D_{1}$ over $\lambda$, 
we put a trajectory starting from the step $U_{i}$ whose art weight is $b_{i}$.
When $b_{i}\le M_{1}$, one can put $b_{i}$ boxes at the step $U_{i}$ in $D_{1}$.
When $b_{i}>M_{1}$, we put several Dyck tiles at the step $U_{i}$ in $D_{1}$ such that 
the art weight on the trajectory is $b_{i}$, and we have a 
Dyck tile of length larger than zero over the $\wedge$-corner whose left edge is the up step 
$U_{i+1}$ in $D_{1}$.  
We have $b_{i}=a_{i}-n_{0}$ and $b_{i+l+1}=a_{i+l+1}-n_{0}-l-1$.
The condition $a_{i}<a_{i+l+1}$ implies $b_{i}\le b_{i+l+1}+l\le b_{i+l+1}+M_{1}$.
Thus, this configuration is admissible as an Hermite history.

\paragraph{Case ii).}
The trajectory starting from $U_{i+1}$ is above the trajectory starting from $U_{i}$, 
which implies $a_{i+1}<a_{i}$ and $a_{i}\ge n_{0}+l$.
By a similar argument to Case i), we have $b_{i+1}\le b_{i}$, and $b_{i}\ge l$.
When $b_{i}\le M_{1}$, we put $b_{i}$ boxes at the step $U_{i}$ in $D_{1}$.
When $b_{i}\ge M_{1}$, we put boxes such that the art weight on the trajectory 
is $b_{i}$, and may have a Dyck tile of length $l$ over the $\wedge$-corner 
whose left edge is the up step $U_{i+1}$ in $D_{1}$.
This configuration is admissible as an Hermite history.

\paragraph{Case iii).}
Since $l>M_{1}$, we have an up step $U_{j}$ such that the trajectory 
associated to it starts from the box just below the Dyck tile of size $l$.
By the same argument as case i), we have 
$b_{i}\le b_{j}+M'$ where $M'$ is the number of down steps between 
$U_{i}$ and $U_{j}$.
Then, this configuration is admissible as an Hermite history.

\paragraph{\bf Case b2).}
As in the case b1), we may have a Dyck tile of length $l$ at 
the $\wedge$-corner whose left edge is $U_{i+1}$ in $R$.
Here, we have two cases: i) $0\le l\le M_{1}$, and ii) $M_{1}<l$.

\paragraph{Case i).} 
By the same argument as Case b1), we have $a_{i}>a_{i+1}$ which implies 
$b_{i}\ge b{i+1}$.
When $b_{i}\le M_{1}$, we put $b_{i}$ boxes at the step $U_{i}$
in $D_{1}$. This is admissible as an Hermite history.
The remaining case is $b_{i}>M_{1}$. 
In this case, $a_{i+l}$ also satisfies $a_{i+l}>a_{i}$, which implies $b_{i+l}\ge b_{i}-l\ge b_{i}-M_{1}$. 
In $D_{1}$, if we put Dyck tiles at the up step $U_{i}$, we have a non-trivial Dyck tile at the 
$\wedge$-corner whose left edge is $U_{i+1}$.
Further, we have at least $n_{0}+l-1$ trajectories above the trajectory of $U_{i}$ in $R$ since 
$b_{i}>M_{1}$.
The condition $b_{i+l}\ge b_{i}-M_{1}$ insures that we have Dyck tiles starting from $U_{i+l}$ in $D_{1}$
and this configuration is admissible as an Hermite history.

\paragraph{Case ii).}
One can apply the same argument as Case b1-iii).
\end{proof}

The following theorem gives the relation between a natural label $T_{1}$ and $\mathbf{b}$.
\begin{theorem}
\label{thrm:HhDTS}
Let $D_{1}$ be a Dyck tiling constructed from $\mathbf{b}$ by the Hermite history
and $D_{2}$ be a Dyck tiling obtained by the DTS bijection from a natural label $T_{1}$.
Then, we have $D_{1}=D_{2}$.
\end{theorem}
\begin{proof}
We prove the theorem by induction.

Let $n$ be the number of edges in $\mathrm{Tree}(\lambda)$, and 
$T_{1}^{'}$ be the natural label obtained from $T_{1}$ by deleting the edge with 
the label $n$.
Let $D_{R}$ be the Dyck tiling in the region $R$ associated with a natural label 
$T_{1}$.
We consider the following operation on $D_{R}$ and obtain a new Dyck tiling 
$D_{R}^{'}$ of size $n-1$.
By induction assumption, we denote by $D_{2}^{'}$ the Dyck tiling associated with 
$T'_{1}$ by the DTS bijection.
The Dyck tiling constructed from $D_{R}'$ by the Hermite history is equal to $D_{2}'$.

We will define an operation called {\it truncation} of $D_{R}$ as follows.
First, we find the south-east box $b$ on the trajectory connected to the anchor box
with the label $0$.
Let $p$ be the position of $b$ from left.
We delete the trajectory connected to the anchor box labeled by zero, and delete the boxes 
in the column at the position $p+1$. 
Then, we reconnect the two regions by moving the right region left by $(-2,0)$.
\begin{figure}[ht]
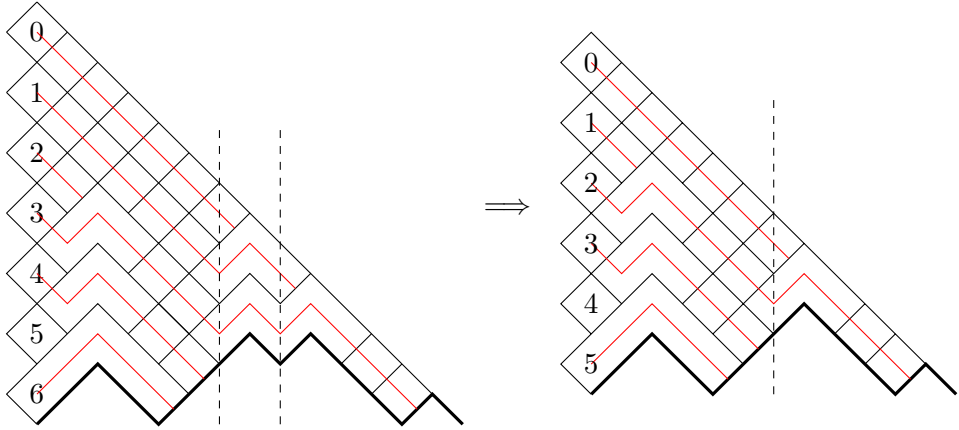

\tikzpic{-0.5}{[x=0.4cm,y=0.4cm]
\draw[very thick](0,0)--(2,2)--(4,0)--(7,3)--(8,2)--(9,3)--(12,0)--(13,1)--(14,0);
\draw(-1,1)--(0,0)(-1,1)--(2,4)--(4,2)--(7,5)--(8,4)--(9,5)--(13,1);
\draw(0,2)--(-1,3)--(0,4)--(2,6)--(4,4)--(7,7)--(9,5);
\draw(0,4)--(-1,5)--(0,6)--(2,8)--(4,6)--(6,8)--(7,7);
\draw(3,7)--(5,9)(2,8)--(4,10)(-1,7)--(3,11)(-1,9)--(2,12)(-1,11)--(1,13)(-1,13)--(0,14);
\draw(0,14)--(6,8)(-1,13)--(5,7)(-1,11)--(2,8)(-1,9)--(1,7)(-1,7)--(1,5)(0,4)--(1,3);
\draw(4,2)--(5,1)(4,4)--(6,2)(11,3)--(10,2)(12,2)--(11,1);
\draw(3,3)--(4,4)--(5,3)(5,5)--(6,4);
\draw(3,5)--(4,6)--(5,5)(5,7)--(6,6);
\foreach \x in {6,5,...,0} {\node at (0,13-2*\x) {\x};}
\draw[red](0,1)--(2,3)--(4.5,0.5)
          (0,5)--(1,4)--(2,5)--(5.5,1.5)
          (0,7)--(1,6)--(2,7)--(6,3)--(7,4)--(8,3)--(9,4)--(12.5,0.5)
          (0,9)--(1.5,7.5)
          (0,11)--(6,5)--(7,6)--(8.5,4.5)
          (0,13)--(6.5,6.5);
\draw[dashed](6,0)--(6,10)(8,0)--(8,10);
}
$\Longrightarrow$
\tikzpic{-0.5}{[x=0.4cm,y=0.4cm]
\draw[very thick](0,0)--(2,2)--(4,0)--(7,3)--(10,0)--(11,1)--(12,0);
\draw(0,0)--(-1,1)--(2,4)--(5,1)(0,2)--(-1,3)--(2,6)--(6,2)(0,4)--(-1,5)--(2,8)--(6,4)
     (0,6)--(-1,7)--(2,10)(0,8)--(-1,9)--(1,11)(0,10)--(-1,11)--(0,12)--(11,1)
     (1,3)--(0,4)(1,5)--(0,6)(1,7)--(0,8)(2,8)--(0,10)(2,8)--(3,9)
     (3,7)--(4,8)(3,5)--(5,7)(3,3)--(6,6)(4,2)--(7,5)(8,2)--(9,3)(9,1)--(10,2);
\draw[red](0,1)--(2,3)--(4.5,0.5)(0,5)--(1,4)--(2,5)--(5.5,1.5)
          (0,7)--(1,6)--(2,7)--(6,3)--(7,4)--(10.5,0.5)(0,9)--(1.5,7.5)(0,11)--(6.5,4.5); 
\foreach \x in {5,...,0} {\node at (0,11-2*\x) {\x};}
\draw[dashed](6,0)--(6,10);
}
\caption{An example of truncation of the Dyck tiling at $x=7$.
We delete the column between dashed line, and reconnect the two regions
at $x=6$.}
\label{fig:DRred}
\end{figure}
We call this procedure the truncation of a Dyck tiling, and 
the new Dyck tiling of size $n-1$ is called $D_{R}^{'}$.
See Figure \ref{fig:DRred} for an example of the truncation of a Dyck tiling in $R$.
We delete the column next to the box $b$, which implies we have 
a $\wedge$-corner at the position $x=p+1$ after deleting the top trajectory connected 
to the anchor box labeled by zero.
We have a pair of an up step $U$ and a down step $D$ which is deleted in the truncation, 
and we denote the up step by $U_{p+1}$.
A trajectory starting from an up step left to $U_{p+1}$ in $D_{R}$
is not changed by the truncation.
This means that  the connectivity of a trajectory left to $U_{p+1}$ in $D_{1}$ is not 
changed by the truncation.
Note that the label of the anchor box on the trajectory decreases by one.

Secondly, the top trajectory is connected to the anchor box with the label $0$ in $D_{R}$. 
This means that $a_{i}=b_{i}=0$ for $i$ such that the top trajectory in $D_{R}$ is associated 
with the $i$-th up step $U_{i}$ from left. 
In terms of the Dyck tiling $D_{1}$ above $\lambda$, there is no box attached to 
the up step $U_{i}$ in $D_{1}$.

Thirdly, all the trajectories associated with the up steps right to $U_{i}$ in $D_{R}$ 
have a $\wedge$-shape at $x=p+1$.
Sine we delete the $\wedge$-corner from $D_{R}$ by truncation, 
$a_{j}$ for $D_{R}$ becomes $a_{j}-1$ in $D_{R}^{'}$ for $i<j$.
From the second observation and the fact that we have $a_{i}=b_{i}=0$, 
$b_{j}$ for $D_{R}$ becomes $b_{j}-1$ in $D_{R}^{'}$ for $i<j$.

When we perform the insertion procedure of the DTS bijection on a Dyck tiling $D_{2}'$,
we insert a $\wedge$ (an adjacent pair of $U$ and $D$) into somewhere of $D_{2}'$.
This added $U$ corresponds to $U_{i}$ in $D_{R}$. 
The configuration left to $U_{i}$ in $D_{2}$ is the same as that of $D_{2}'$.
There is no Dyck tile which is attached to the added edge $U_{i}$.
By the strip-growth of the DTS bijection, we add single boxes to up steps  
to obtain the Dyck tiling $D_{2}$.
This addition is equivalent to the increment of $a_{j}$ for all the trajectories
in $D'_{R}$ with $i<j$ by one when we construct $D_{2}$ from $D_{2}^{'}$.

By summarizing the above arguments, the truncation of $D_{R}$ corresponds to 
the inverse of the insertion procedure of the DTS bijection at the position $x=p$.
Thus, $D_1$ is obtained from $D_{2}'$ by the insertion of the DTS bijection,
which implies $D_{1}=D_{2}$.
\end{proof}

Recall that $D_{R}$ is a Dyck tiling in the region $R$ corresponding to a weakly increasing 
tree $T_{2}$.
We define a statistics on $D_{R}$ by
\begin{eqnarray*}
\mathrm{art}_{-}(D_{R})=(\mathrm{area}(D_{R})-\mathrm{tiles}(D_{R}))/2.
\end{eqnarray*}
where the statistics $\mathrm{area}$ is the number of boxes in $R$ and 
the statistics $\mathrm{tiles}$ is the number of Dyck tiles in $R$.
Let $D_{2}$ be a Dyck tiling via the DTS bijection from a natural label $T_{1}$.

\begin{prop}
We have 
\begin{eqnarray*}
\mathrm{art}(D_{2})=\mathrm{art}_{-}(D_R)
\end{eqnarray*}
\end{prop}
\begin{proof}
When $\mathbf{b}=0$, we have an empty Dyck tiling over the path 
$\lambda$ and $\mathrm{art}(D_{2})=0$.
In this case, the Dyck tiling in $R$ contains only single boxes, 
and $\mathrm{art}_{-}(D_{R})=0$.
When we increase the length of a single box by one in $D_{R}$ 
(the single box is above a $\wedge$-corner in $\lambda$),
we increase the number of single boxes in $D(\lambda)$ or 
the length of a non-trivial Dyck tile in $D(\lambda)$ by one.
Since $\mathbf{b}$ defines an Hermite history, 
the statistics $\mathrm{art}_{-}$ on $D_{R}$ gives 
$\mathrm{art}_{-}(D_{R})=\sum_{1\le k\le n}b_{k}$, 
which is nothing but $\mathrm{art}(D_{2})$ from Theorem \ref{thrm:HhDTS}.
\end{proof}

In the proof of Theorem \ref{thrm:HhDTS}, we introduce the truncation 
of the Dyck tiling $D_{R}$ in the region $R$.
By taking the inverse of truncation, we obtain the insertion procedure
for $D_{R}$.

Recall that an insertion history $\mathbf{h}$ (defined in Section \ref{sec:DTSDTR})
is bijective to a natural label $T_{1}$ of the tree $\mathrm{Tree}(\lambda)$.

The insertion procedure for the Dyck tiling in $R$ consists of 
two steps: {\it column addition} and {\it  trajectory addition}.
Since $h_{1}=0$ for any insertion history, we put a single box such that 
its west vertex of the box is at $(0,0)$ and the center of the box is placed 
at $(1,0)$ in the Cartesian coordinate.
Suppose we have a Dyck tiling $D_R$ in $R$ for $\mathbf{h}=(h_1,\ldots,h_{n})$. 
A column addition for $h_{n+1}$ is as follows:
\begin{enumerate}
\item 
\label{colinsDR1}
We split the Dyck tiling $D_{R}$ at $x=h_{n+1}$ and translate the right 
  pieces right by $(2,0)$.
\item 
\label{colinsDR2}
We connect the vertices on $x=h_{n+1}$ and $x=h_{n+1}+2$ by the path $UD$.
\end{enumerate}
We denote by $\widetilde{D}_{R}$ the diagram obtained by column insertion.
In the operation (\ref{colinsDR2}), we add the path $UD$ in the top path.
Note that there is no Dyck tile attached to the added up step $U$ which 
is on the top path of $\widetilde{D}_{R}$. 

The trajectory addition on $\widetilde{D}_{R}$ is to add $h_{n+1}+1$ single
boxes in the $(-1,1)$ direction from the up step which is on the 
top path of $\widetilde{D}_{R}$ and added by the column addition.
Then, we obtain a Dyck tiling of size $n+1$ in the region $R$.
\begin{figure}[ht]
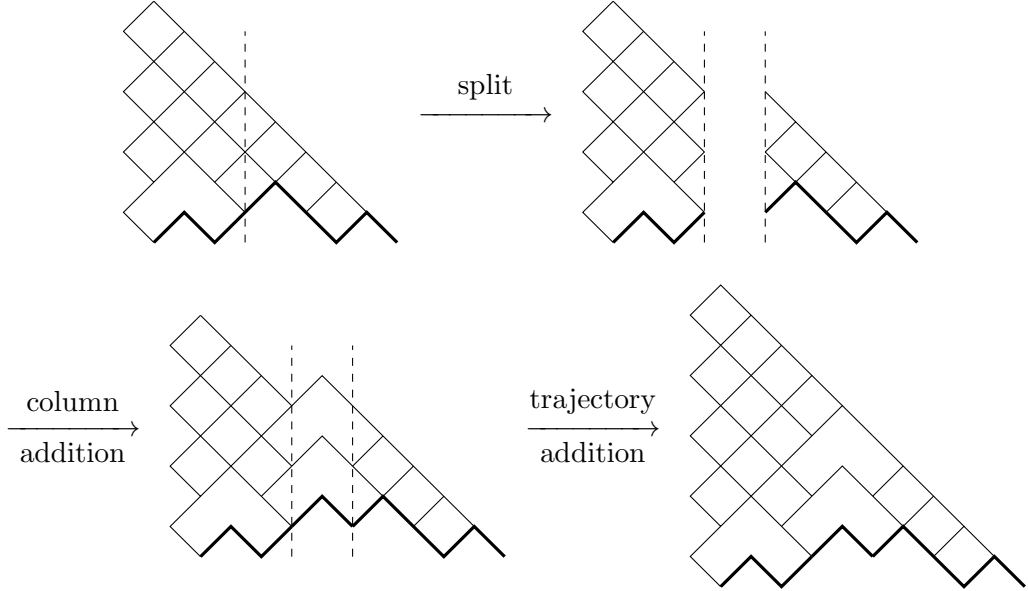

\tikzpic{-0.5}{[x=0.4cm,y=0.4cm]
\draw[very thick](0,0)--(1,1)--(2,0)--(4,2)--(6,0)--(7,1)--(8,0);
\draw(0,0)--(-1,1)--(3,5)(0,2)--(-1,3)--(2,6)(0,4)--(-1,5)--(1,7)(0,6)--(-1,7)
     (2,2)--(4,4)(4,2)--(5,3)(5,1)--(6,2)(-1,7)--(0,8);
\draw(0,8)--(7,1)(0,6)--(4,2)(0,4)--(3,1);
\draw[dashed](3,0)--(3,7);
}
$\stackrel{\mbox{split}}{\xrightarrow{\qquad\qquad}}$
\tikzpic{-0.5}{[x=0.4cm,y=0.4cm]
\draw[very thick](0,0)--(1,1)--(2,0)--(3,1)(5,1)--(6,2)--(8,0)--(9,1)--(10,0);
\draw(0,0)--(-1,1)--(3,5)(0,2)--(-1,3)--(2,6)(0,4)--(-1,5)--(1,7)(0,6)--(-1,7)--(0,8)
     (2,2)--(3,3);
\draw(3,1)--(0,4)(3,3)--(0,6)(3,5)--(0,8);
\draw(5,3)--(6,4)(6,2)--(7,3)(7,1)--(8,2)(5,5)--(9,1)(5,3)--(6,2);
\draw[dashed](3,0)--(3,7)(5,0)--(5,7);
} \\[5mm]
$\overset{\mbox{column}}{\underset{\mbox{addition}}{\xrightarrow{\qquad\qquad}}}$
\tikzpic{-0.5}{[x=0.4cm,y=0.4cm]
\draw[very thick](0,0)--(1,1)--(2,0)--(3,1)(5,1)--(6,2)--(8,0)--(9,1)--(10,0);
\draw(0,0)--(-1,1)--(3,5)(0,2)--(-1,3)--(2,6)(0,4)--(-1,5)--(1,7)(0,6)--(-1,7)--(0,8)
     (2,2)--(3,3);
\draw(3,1)--(0,4)(3,3)--(0,6)(3,5)--(0,8);
\draw(5,3)--(6,4)(6,2)--(7,3)(7,1)--(8,2)(5,5)--(9,1)(5,3)--(6,2);
\draw[very thick](3,1)--(4,2)--(5,1);
\draw(3,3)--(4,4)--(5,3)(3,5)--(4,6)--(5,5);
\draw[dashed](3,0)--(3,7)(5,0)--(5,7);
}
$\overset{\mbox{trajectory}}{\underset{\mbox{addition}}{\xrightarrow{\qquad\qquad}}}$
\tikzpic{-0.5}{[x=0.4cm,y=0.4cm]
\draw[very thick](0,0)--(1,1)--(2,0)--(3,1)(5,1)--(6,2)--(8,0)--(9,1)--(10,0);
\draw(0,0)--(-1,1)--(3,5)(0,2)--(-1,3)--(2,6)(0,4)--(-1,5)--(1,7)(0,6)--(-1,7)--(0,8)
     (2,2)--(3,3);
\draw(3,1)--(0,4)(3,3)--(0,6)(3,5)--(0,8);
\draw(5,3)--(6,4)(6,2)--(7,3)(7,1)--(8,2)(5,5)--(9,1)(5,3)--(6,2);
\draw[very thick](3,1)--(4,2)--(5,1);
\draw(3,3)--(4,4)--(5,3)(3,5)--(4,6)--(5,5);
\draw(0,8)--(-1,9)--(0,10)--(4,6)(0,8)--(1,9)(1,7)--(2,8)(2,6)--(3,7);
}
\caption{An example of insertion procedure for a Dyck tiling of size $4$ in $R$.
The insertion point is $x=3$.}
\label{fig:exinsDR}
\end{figure}
See Figure \ref{fig:exinsDR} for an example of insertion procedure for a Dyck 
tiling in $R$.

\subsection{Involutions on increasing trees}
\label{sec:alpha}
Given a path $\lambda$, we denote by $\overline{\lambda}$ a path reflected 
by a vertical line.
In other words, $\lambda:=\lambda_{1}\ldots \lambda_{n}$ for $\lambda_{i}\in\{U,D\}$ 
gives $\overline{\lambda}:=\overline{\lambda_{n}}\ldots\overline{\lambda_{1}}$ 
where $\overline{U}=D$ and $\overline{D}=U$.
The tree $\mathrm{Tree}(\overline{\lambda})$ can be obtained from $\mathrm{Tree}(\lambda)$
by reflecting along a vertical line.

Let $T_{1}$ be a natural label on a tree $\mathrm{Tree}(\lambda)$, 
$T_{1}(e)$ is a label on an edge $e$, and $n$ be the number of edges in $T_{1}$.
We define an {\it bar operation} on a label $T_{1}(e)\in[1,n]$: 
\begin{eqnarray*}
\overline{T_{1}(e)}:=n+1-T_{1}(e).
\end{eqnarray*}
Then, we denote by $\overline{T_{1}}$ a decreasing tree from the root to the leaves, 
which is obtained from $T_{1}$ by acting the bar operation on all edges of $T_{1}$.

From $\overline{T_{1}}$, we construct an increasing tree $S_{1}$ as follows.
Let $f_{i}$ for $i\in[1,n]$ be the edge of the tree $\overline{T_{1}}$ with 
the label $i$.
Take an edge $f_{m_0}$ of $\overline{T_{1}}$. 
Suppose that the edge $f_{m_1}$ is a child of $f_{m_0}$ such $m_1$ is the maximum integer 
among the integers satisfying $m_1<m_0$. 
We denote this relation by $f_{m_{0}}\searrow f_{m_1}$.
Then, we have a unique chain of edges starting from $f_{n}$:
\begin{eqnarray*}
f_{n_0}\searrow f_{n_1}\searrow \ldots \searrow f_{n_{p}},
\end{eqnarray*}
where $n_0:=n$ and $n_0>n_{1}>\ldots>n_p$ with maximal $p$.
Note that the edge $f_{n_{0}}$ is connected to the root and 
the edge $f_{n_p}$ just above a leaf. 
We change the label $n_{i}$ of the edges $f_{n_{i}}$ by $n_{i+1}$ for $0\le i\le p-1$
and the label $n_{p}$ by $n_{0}$.
After the above operation, the integer $n_{0}$ is on the edge connected to a leaf.
Then, we construct a chain starting from $f_{n-1}$, and 
shift the labels of the chain in the similar way as above.
Note that, after the second operation, the edge with the label $n-1$ is connected to 
a leaf or connected to the edge with the label $n$.
We continue this procedure until we have a chain starting from 
and ending with $f_{1}$.
During the above successive operations, we may have a child $f_{m}$ of an edge $f_{n_p}$
with $m>n_{p}$. 
By construction of the relation by $\searrow$, this child 
edge $f_{m}$ does not effect anything on the algorithm at all.
We just keep the label of the edge as it is.
We call this procedure on $\overline{T_{1}}$ the {\it cyclic operation}.

We denote a map defined above by $\alpha: T_{1}\mapsto S_{1}$, that is, 
$\alpha$ is a composition of the bar operation and the cyclic operation.
By construction, it is obvious that $S_{1}$ is again an increasing tree 
of the shape $\mathrm{Tree}(\lambda)$.
See Figure \ref{fig:exalpha} for an example.
\begin{figure}[ht]
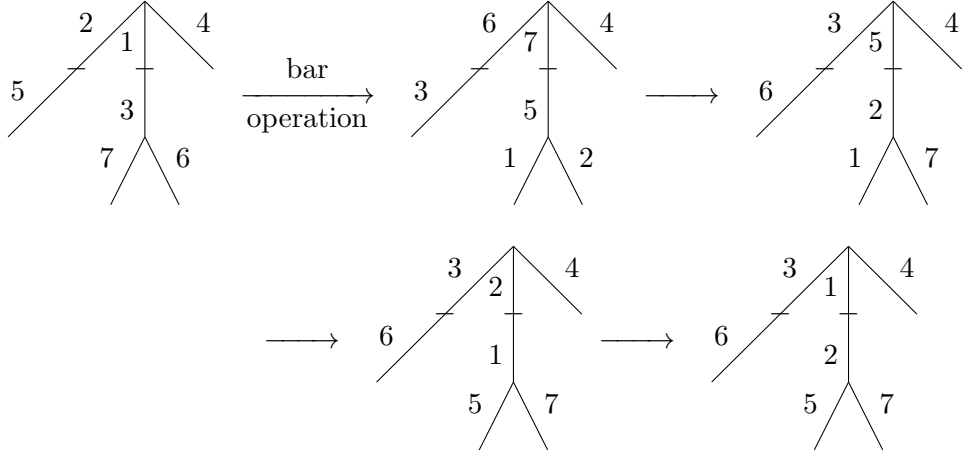

\tikzpic{-0.5}{[scale=0.6]
\coordinate 
 	child{coordinate (c2)
	     child{coordinate (c5)}
             child[missing]
             child[missing]
	     }
	child{coordinate (c1)
             child{coordinate (c3)
                  child{coordinate (c7)}
                  child{coordinate (c6)}
		  }
             }
	child{coordinate (c4)};
\draw(c2)node{$-$}(c1)node{$-$};
\node[anchor=south east] at ($(0,0)!.6!(c2)$){2};
\node[anchor=east] at ($(0,0)!.6!(c1)$){1};
\node[anchor=south west] at ($(0,0)!.6!(c4)$){4};
\node[anchor=south east] at ($(c2)!.6!(c5)$){5};
\node[anchor=east] at ($(c1)!.6!(c3)$){3};
\node[anchor=south east] at ($(c3)!.6!(c7)$){7};
\node[anchor=south west] at ($(c3)!.6!(c6)$){6};
}
$\overset{\mbox{bar}}{\underset{\mbox{operation}}{\xrightarrow{\qquad\qquad}}}$
\tikzpic{-0.5}{[scale=0.6]
\coordinate 
 	child{coordinate (c2)
	     child{coordinate (c5)}
             child[missing]
             child[missing]
	     }
	child{coordinate (c1)
             child{coordinate (c3)
                  child{coordinate (c7)}
                  child{coordinate (c6)}
		  }
             }
	child{coordinate (c4)};
\draw(c2)node{$-$}(c1)node{$-$};
\node[anchor=south east] at ($(0,0)!.6!(c2)$){6};
\node[anchor=east] at ($(0,0)!.6!(c1)$){7};
\node[anchor=south west] at ($(0,0)!.6!(c4)$){4};
\node[anchor=south east] at ($(c2)!.6!(c5)$){3};
\node[anchor=east] at ($(c1)!.6!(c3)$){5};
\node[anchor=south east] at ($(c3)!.6!(c7)$){1};
\node[anchor=south west] at ($(c3)!.6!(c6)$){2};
}
$\xrightarrow{\qquad}$
\tikzpic{-0.5}{[scale=0.6]
\coordinate 
 	child{coordinate (c2)
	     child{coordinate (c5)}
             child[missing]
             child[missing]
	     }
	child{coordinate (c1)
             child{coordinate (c3)
                  child{coordinate (c7)}
                  child{coordinate (c6)}
		  }
             }
	child{coordinate (c4)};
\draw(c2)node{$-$}(c1)node{$-$};
\node[anchor=south east] at ($(0,0)!.6!(c2)$){3};
\node[anchor=east] at ($(0,0)!.6!(c1)$){5};
\node[anchor=south west] at ($(0,0)!.6!(c4)$){4};
\node[anchor=south east] at ($(c2)!.6!(c5)$){6};
\node[anchor=east] at ($(c1)!.6!(c3)$){2};
\node[anchor=south east] at ($(c3)!.6!(c7)$){1};
\node[anchor=south west] at ($(c3)!.6!(c6)$){7};
} \\[5mm]
\hspace{3cm}$\xrightarrow{\qquad}$
\tikzpic{-0.5}{[scale=0.6]
\coordinate 
 	child{coordinate (c2)
	     child{coordinate (c5)}
             child[missing]
             child[missing]
	     }
	child{coordinate (c1)
             child{coordinate (c3)
                  child{coordinate (c7)}
                  child{coordinate (c6)}
		  }
             }
	child{coordinate (c4)};
\draw(c2)node{$-$}(c1)node{$-$};
\node[anchor=south east] at ($(0,0)!.6!(c2)$){3};
\node[anchor=east] at ($(0,0)!.6!(c1)$){2};
\node[anchor=south west] at ($(0,0)!.6!(c4)$){4};
\node[anchor=south east] at ($(c2)!.6!(c5)$){6};
\node[anchor=east] at ($(c1)!.6!(c3)$){1};
\node[anchor=south east] at ($(c3)!.6!(c7)$){5};
\node[anchor=south west] at ($(c3)!.6!(c6)$){7};
}$\xrightarrow{\qquad}$
\tikzpic{-0.5}{[scale=0.6]
\coordinate 
 	child{coordinate (c2)
	     child{coordinate (c5)}
             child[missing]
             child[missing]
	     }
	child{coordinate (c1)
             child{coordinate (c3)
                  child{coordinate (c7)}
                  child{coordinate (c6)}
		  }
             }
	child{coordinate (c4)};
\draw(c2)node{$-$}(c1)node{$-$};
\node[anchor=south east] at ($(0,0)!.6!(c2)$){3};
\node[anchor=east] at ($(0,0)!.6!(c1)$){1};
\node[anchor=south west] at ($(0,0)!.6!(c4)$){4};
\node[anchor=south east] at ($(c2)!.6!(c5)$){6};
\node[anchor=east] at ($(c1)!.6!(c3)$){2};
\node[anchor=south east] at ($(c3)!.6!(c7)$){5};
\node[anchor=south west] at ($(c3)!.6!(c6)$){7};
}
\caption{An example of the action of $\alpha$ on a natural label.}
\label{fig:exalpha}
\end{figure}

\begin{prop}
The map $\alpha$ is an involution on natural labels of $\mathrm{Tree}(\lambda)$.
\end{prop}
\begin{proof}

Suppose that $\mathrm{Tree}(\lambda)$ can be decomposed into a concatenation 
of Dyck paths $\lambda_1,\ldots,\lambda_p$ which cannot be decomposed into 
a concatenation of Dyck paths of smaller length.
The tree $\mathrm{Tree}(\lambda)$ can be written as
$\mathrm{Tree}(\lambda)=\mathrm{Tree}(\lambda_1)\circ\cdots\circ\mathrm{Tree}(\lambda_p)$.

The bar operation and the cyclic operation are commutative with a concatenation of 
natural labels.
More precisely, suppose that a natural label $T_{1}$ is written as 
$T_{1,1}\circ\cdots\circ T_{1,p}$ where the shape of $T_{1,i}$ is $\mathrm{Tree}(\lambda_{i})$.
Then, since $\alpha$ is a successive actions of the bar and cyclic operations, we have 
\begin{eqnarray*}
\alpha(T_{1})&=&\alpha(T_{1,1}\circ\cdots\circ T_{1,p}), \\
&=&\alpha(T_{1,1})\circ\cdots\circ\alpha(T_{1,p}).
\end{eqnarray*}
To show that $\alpha$ is an involution, we need to show that $\alpha$ is an involution
on $T_{1,i}$ for $1\le i\le p$.
Since $T_{1,i}$ cannot be decomposed into a concatenation of natural labels of smaller 
size, $T_{1,i}$ has a unique edge $e_{0}$ which is connected to the root.
The label of $e_{0}$ in $T_{1,i}$ is the smallest label in $T_{1,i}$.
After the action of $\alpha$ on $T_{1,i}$, it is obvious that we have again that 
$e_{0}$ is the smallest label in $\alpha(T_{1,i})$.
Let $S$ be the set of labels in $T_{1,i}$ and $\widetilde{S}$ be the set of labels 
in $\alpha(T_{1,i})$.
Let $T_{1,i}^{\times}$ be the natural label obtained from $T_{1,i}$ by deleting the 
edge $e_{0}$.
Then, the labels of $T_{1,i}^{\times}$ are in $S\setminus\{\min(S)\}$
and the labels of $\alpha(T_{1,i})^{\times}$ are in $\widetilde{S}\setminus\{\min(\widetilde{S})\}$.
Note that $\min(\widetilde{\widetilde{S}})=\min(S)$ and we also have $\min(S)+\max(\widetilde{S})=n+1$. 
This implies that $\alpha$ is an involution if it is an involution on $T_{1,i}^{\times}$.
By continuing the decomposition of $\mathrm{Tree}(\lambda)$ into trees of smaller size 
and the deleting of the unique edge which is connected to the root, 
it is enough to check the action of $\alpha$ on a tree of smaller size.

From the above observations, it is enough to show that $\alpha$ is an involution
on the following two extreme trees $\mathrm{Tree}(\lambda)$: 
1) $\lambda=\wedge_{m}$, and 2) $\lambda=\mathrm{zigzag}_{m}$.

For case 1), the labels are $1,\ldots,m$ and increasing from top to bottom on the tree.
The application of the map $\alpha$ keeps the labels the same order as before.
Thus, we have $\alpha^{2}$ acts as the identity in this case.

For case 2), suppose that we have the labels $(n_1,n_2,\ldots,n_m)$ from 
left to right on the tree.
The action of $\alpha$'s changes $n_{i}$ to $\overline{n_{i}}$, then to 
$\overline{\overline{n_{i}}}$. 
Since we have $\overline{\overline{m}}=m$ from the definition of the bar 
operation, we have $\alpha^{2}$ acts as the identity.
\end{proof}

Given a Dyck path $\lambda$, we define $N_{\mathrm{max}}$ as the number of single boxes
in the region specified by $\lambda$ and the top path $U\ldots UD\ldots D$.
\begin{prop}
Let $S_{1}:=\alpha(T_{1})$. 
We have 
\begin{eqnarray}
\label{eqn:tsmax}
\mathrm{art}(\mathrm{DTS}(T_{1}))+\mathrm{art}(\mathrm{DTS}(S_{1}))=N_{\mathrm{max}}.
\end{eqnarray}
\end{prop}
\begin{proof}

Let $T_{2}$ be a weakly increasing tree obtained from $T_1$ by Eqn. (\ref{T2}), 
and $T_{3}$ be a weakly decreasing tree obtained from $\overline{T_{1}}$ by
\begin{eqnarray*}
T_{3}(e):=\#\{e'| e\rightarrow e', \overline{T_{1}}(e)>\overline{T_{1}}(e')\},
\end{eqnarray*}
where $T(e)$ is the label of the edge $e$ in a labeled tree $T$.
We denote by $|T|$ the sum of the labels in a labeled tree $T$.

We first claim that $|T_{2}|+|T_{3}|=N_{\mathrm{max}}$.
When we perform the spread of a Dyck tiling, the number of boxes is not 
changed. Then, when we perform the strip-growth, we add some boxes.
The number of added boxes is equal to the number of edges which are right to 
the added $UD$ path and whose labels are smaller than the one of the edges
corresponding to the added $UD$ path.
Recall that $\mathrm{Tree}(\lambda/\lambda_{0})$ with $\lambda_{0}:=D^{2n}$ 
defines the capacities on the leaves of the tree $\mathrm{Tree}(\lambda)$. 
The above consideration leads to the following statement: $N_{\max}$ 
is equal to the sum of labels which is maximal with respect to the 
capacities of $\mathrm{Tree}(\lambda/\lambda_{0})$. 
Since the bar operation reverses the order of the labels on the edges, {\it i.e.},
the increasing (resp. decreasing) tree is mapped to a decreasing (resp. an increasing)
tree, $T_{2}(e)+T_{3}(e)$ is equal to the maximal label of the edge $e$ 
with respect to the capacity.
Summing $T_{2}(e)+T_{3}(e)$ on all over the edges, we obtain that 
$|T_{2}|+|T_{3}|$ is equal to the sum of maximal labels and 
it is equal to $N_{\max}$. 

Next, we claim that $|T_{3}|$ is invariant under the cyclic operation 
on $\overline{T_{1}}$.
Let $f_{i}$ be the edge of $\overline{T_{1}}$ with the label $i$.
Let $n_{0}>n_{1}>\cdots>n_{p}$ be the integers satisfying 
$f_{n_0}\searrow f_{n_1}\searrow\ldots\searrow f_{n_p}$.
Suppose we have $M_{i}$ edges which are right to $f_{n_{i}}$ and are 
below $f_{n_{i-1}}$ for $1\le i\le p$.
From the definition of the cyclic operation, the labels of these 
$M_{i}$ edges are smaller than $n_{i}$.
For $1\le i\le p$, we move these labels $n_{1},\ldots,n_{p}$ upward by one edge during the 
cyclic operation. 
Let $f^{'}_{i}$ be the edge with the label $i$ in the new tree $T_{4}$ after
the move of these labels on $\overline{T_{1}}$.
Let $T_{5}$ be a label constructed from $T_{4}$ by Eqn. (\ref{T2}).
Note that $T_{5}$ may be neither a weakly increasing nor weakly decreasing tree.
Then, we have $T_{5}(f'_{n_i})=T_{3}(f_{n_i})-M_{i}$.
The edge $f_{n_i}$ in $T_{3}$ is a child of the edge $f'_{n_i}$ in $T_{5}$.
For $n_{0}$, we have $T_{5}(f'_{n_0})=T_{3}(f_{n_0})+\sum_{1\le j\le p}M_{j}$.
Summing up all the contributions, we have 
$|T_{3}|=|T_{5}|$. 
We can apply a similar argument on the cyclic operation on $T_{4}$, and easily 
show that $|T_{5}|$ is invariant. 
Thus, $|T_{3}|$ is invariant under the cyclic operations.

A label of $T_{2}$ counts the number of non-trivial Dyck tiles which are spread 
by length one plus the number of single boxes added in the strip-growth, 
which means that the sum of the labels in $T_{2}$ is equal to the 
statistics $\mathrm{art}$ of $\mathrm{DTS}(T_{1})$.
Thus, $|T_{2}|$ and $|T_{3}|$ are equal to $\mathrm{art}(\mathrm{DTS}(T_1))$ 
and $\mathrm{art}(\mathrm{DTS}(S_1))$ respectively.
Therefore, we obtain Eqn. (\ref{eqn:tsmax}).
\end{proof}

We have a natural involution called {\it reflection} between $T_{1}$ and a natural label on 
$\mathrm{Tree}(\overline{\lambda})$, where $\overline{\lambda}$ is the path 
obtained from $\lambda$ by the mirror image.
We denote by $\mathrm{ref}(T_{1})$ the mirror image (or equivalently reflection) 
of $T_{1}$.
A Dyck tiling $D$ also has a natural involution by a reflection.
We denote by $\mathrm{ref}(D)$ the reflection of $D$ along a vertical line.

Let $T_{1}$ be a natural label of $\mathrm{Tree}(\lambda)$.
We denote $X_{n}(\cdots (X_{2}(X_{1}(T))\cdots)=T_{X}$ by 
\begin{eqnarray*}
T\xrightarrow{X_{1}}\xrightarrow{X_{2}}\cdots\xrightarrow{X_{n}} T_{X}
\end{eqnarray*}
where $X_{i}$, $1\le i\le n$, 
is either $\alpha$, $\mathrm{DTS}^{\pm1}$ or $\mathrm{ref}$. 
\begin{theorem}
Let $T_{2}$ and $T_{3}$ be natural labels such that 
\begin{gather}
T_{1}\xrightarrow{\alpha}T_{2}, \\
\label{eqn:adrda}
T_{1}\xrightarrow{\alpha}\xrightarrow{\mathrm{DTS}}
\xrightarrow{\mathrm{ref}}\xrightarrow{\mathrm{DTS}^{-1}}
\xrightarrow{\alpha}T_{3}.
\end{gather}
Then we have 
\begin{eqnarray}
\label{eqn:alphaiff}
T_{3}=\mathrm{ref}(T_{1})
\Leftrightarrow
T_{2}\xrightarrow{\mathrm{ref}}\xrightarrow{\mathrm{DTS}}\xrightarrow{\mathrm{ref}}
\xrightarrow{\mathrm{DTS}^{-1}}T_{2}.
\end{eqnarray}
\end{theorem}
\begin{proof}
Suppose $T_{3}=\mathrm{ref}(T_{1})$.
Since $T_{2}=\alpha(T_{1})$ and $\alpha\circ\mathrm{ref}=\mathrm{ref}\circ\alpha$, 
we have 
\begin{eqnarray}
\label{eq:refal}
T_{2}=\mathrm{ref}\circ\alpha(T_{3}).
\end{eqnarray}
Let $T_{4}$ be a natural label such that 
\begin{eqnarray*}
T_{2}\xrightarrow{\mathrm{ref}}\xrightarrow{\mathrm{DTS}}\xrightarrow{\mathrm{ref}}T_{4}.
\end{eqnarray*}
The above equation can be written in terms of $T_{3}$ by Eqn.(\ref{eq:refal}), namely
we have 
\begin{eqnarray*}
T_{3}\xrightarrow{\alpha}\xrightarrow{\mathrm{DTS}}\xrightarrow{\mathrm{ref}}T_{4}.
\end{eqnarray*}
From the inverse of Eqn.(\ref{eqn:adrda}), we have 
\begin{eqnarray*}
T_{3}\xrightarrow{\mathrm{\alpha}}\xrightarrow{\mathrm{DTS}}\xrightarrow{\mathrm{ref}}
\mathrm{DTS}\circ\alpha(T_{1})=\mathrm{DTS}(T_{2}).
\end{eqnarray*}
Thus, we obtain $T_{4}=\mathrm{DTS}(T_{2})$, which implies the $\Rightarrow$ 
part of Eqn.(\ref{eqn:alphaiff}).

Next, we prove the $\Leftarrow$ part of Eqn.(\ref{eqn:alphaiff}).
Let $T_{5}$ be a natural label such that $T_{5}=\mathrm{ref}(T_{1})$.
Since $T_{2}=\alpha(T_{1})$, we have $T_{2}=\mathrm{ref}\circ\alpha(T_{5})$.
The right hand side of Eqn.(\ref{eqn:alphaiff}) can be written in terms of 
$T_{5}$:
\begin{eqnarray*}
T_{5}\xrightarrow{\alpha}\xrightarrow{\mathrm{DTS}}\xrightarrow{\mathrm{ref}}
\xrightarrow{\mathrm{DTS}^{-1}}\xrightarrow{\alpha}\mathrm{ref}(T_{5}).
\end{eqnarray*}
By taking the inverse and putting $T_{1}=\mathrm{ref}(T_{5})$, we obtain 
$T_{3}=\mathrm{ref}(T_{1})$.
\end{proof}

In Section \ref{sec:bijwit}, we give a bijection between a natural label $T_{1}$ 
and a Dyck tiling associated with a weakly increasing tree $T_{2}$.
We also have a bijection between $S_{1}:=\alpha(T_{1})$ and a Dyck tiling in 
the region $R'$ surrounded by the lowest path $\lambda$, the path $U^{2n}$ and $x=2n$.
We construct a weakly increasing tree $S_{2}$ from $S_{1}$ as follows.
A label $S_{2}(e)$ of an edge $e$ is equal to 
\begin{eqnarray*}
S_{2}(e):=\#\{e'| S_{1}(e)>S_{1}(e'), e'\leftarrow e\}.
\end{eqnarray*}
Then, as in Section \ref{sec:bijwit}, we have a cover-inclusive Dyck 
tiling in the region $R'$.
By summarizing above considerations, we have the following theorem.
\begin{theorem}
Given a Dyck path $\lambda$, there exists a bijection between a Dyck tiling $D'_{1}$ 
in the region $R$ and a Dyck tiling $D'_{2}$ in the region $R'$, 
where $D'_{1}$ is associated with $T_{1}$ and $D'_{2}$ is associated 
with $S_{1}=\alpha(T_{1})$.
Furthermore, let $D_{1}$ (resp. $D_{2}$) be the Dyck tiling above $\lambda$ 
constructed from $D'_{1}$ (resp. $D'_{2}$) via the Hermite history.
Then, we have $\mathrm{art}(D_{1})=\mathrm{art}(D_{2})$.
\end{theorem}

\begin{example}
We consider the following natural label $T_{1}$, the Dyck tiling $D'_{1}$ in $R$, and 
the Dyck tiling $D_{1}$ above $\lambda$.
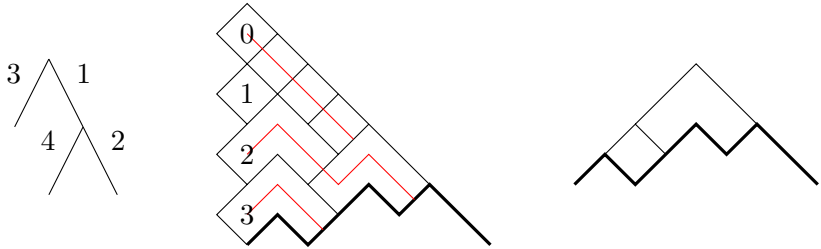
\begin{eqnarray*}
\tikzpic{-0.5}{[scale=0.6]
\coordinate
	child{coordinate(c3)}
	child{coordinate(c1)
		child{coordinate(c4)}
		child{coordinate(c2)}
	};
\draw($(0,0)!0.5!(c3)$)node[anchor=south east]{$3$};
\draw($(0,0)!0.5!(c1)$)node[anchor=south west]{$1$};
\draw($(c1)!0.5!(c4)$)node[anchor=south east]{$4$};
\draw($(c1)!0.5!(c2)$)node[anchor=south west]{$2$};
}\qquad
\tikzpic{-0.5}{[x=0.4cm,y=0.4cm]
\draw[very thick](0,0)--(1,1)--(2,0)--(4,2)--(5,1)--(6,2)--(8,0);
\draw(0,0)--(-1,1)--(1,3)--(3,1)(2,2)--(4,4)(0,2)--(-1,3)--(1,5)--(3,3)
     (0,4)--(-1,5)--(1,7)(0,6)--(-1,7)--(0,8)--(6,2)
     (0,6)--(1,5)--(2,6)(2,4)--(3,5);
\draw[red](0,1)--(1,2)--(2.5,0.5)(0,3)--(1,4)--(3,2)--(4,3)--(5.5,1.5)
          (0,7)--(3.5,3.5);
\draw(0,1)node{$3$}(0,3)node{$2$}(0,5)node{$1$}(0,7)node{$0$};
}\qquad
\tikzpic{-0.5}{[x=0.4cm,y=0.4cm]
\draw[very thick](0,0)--(1,1)--(2,0)--(4,2)--(5,1)--(6,2)--(8,0);
\draw(1,1)--(2,2)--(3,1)(2,2)--(4,4)--(6,2);
}
\end{eqnarray*}
Then, the natural label $S_{1}=\alpha(T_{1})$, the Dyck tiling $D'_{2}$ in $R'$,
and the Dyck tiling $D_{2}$ above $\lambda$ are depicted as below.
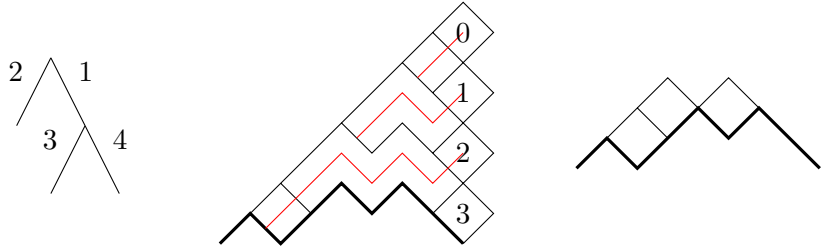
\begin{eqnarray*}
\tikzpic{-0.5}{[scale=0.6]
\coordinate
	child{coordinate(c2)}
	child{coordinate(c1)
		child{coordinate(c3)}
		child{coordinate(c4)}
	};
\draw($(0,0)!0.5!(c2)$)node[anchor=south east]{$2$};
\draw($(0,0)!0.5!(c1)$)node[anchor=south west]{$1$};
\draw($(c1)!0.5!(c3)$)node[anchor=south east]{$3$};
\draw($(c1)!0.5!(c4)$)node[anchor=south west]{$4$};
}\qquad
\tikzpic{-0.5}{[x=0.4cm,y=0.4cm]
\draw[very thick](0,0)--(1,1)--(2,0)--(4,2)--(5,1)--(6,2)--(8,0);
\draw(1,1)--(8,8)--(9,7)--(7,5)(7,7)--(9,5)--(8,4)(2,2)--(3,1)
     (4,4)--(5,3)--(6,4)--(9,1)--(8,0)(7,1)--(9,3)--(6,6)(7,3)--(8,4);
\draw[red](1.5,0.5)--(4,3)--(5,2)--(6,3)--(7,2)--(8,3)
          (4.5,3.5)--(6,5)--(7,4)--(8,5)(6.5,5.5)--(8,7);
\draw(8,1)node{$3$}(8,3)node{$2$}(8,5)node{$1$}(8,7)node{$0$};
}\qquad
\tikzpic{-0.5}{[x=0.4cm,y=0.4cm]
\draw[very thick](0,0)--(1,1)--(2,0)--(4,2)--(5,1)--(6,2)--(8,0);
\draw(1,1)--(3,3)--(4,2)(2,2)--(3,1)(4,2)--(5,3)--(6,2);
}
\end{eqnarray*}
We have $\mathrm{art}(D_{1})=\mathrm{art}(D_{2})=3$.
\end{example}

\section{Dyck tableaux for general Dyck tilings}
\label{sec:DTab}
\subsection{Dyck tableaux}
Let $\lambda$ be a Dyck path (not necessarily a zigzag path).
Due to the construction of the tree $\mathrm{Tree}(\lambda)$ 
from the Dyck path $\lambda$ as in Section \ref{sec:ppt}, 
an edge of $\mathrm{Tree}(\lambda)$ consists of a pair 
of an up step $U_1$ and a down step $D_1$ in $\lambda$.
There exists a unique box which is in the south-east direction from $U_1$ 
and in the south-west direction from $D_1$ under the path $\lambda$.
When $\lambda$ is of length $2n$, we have $n$ such boxes corresponding 
to $n$ edges of $\mathrm{Tree}(\lambda)$.
We call these unique boxes {\it anchor boxes} of $\lambda$.

Let $\lambda_{i}$ for $1\le i\le m$ be Dyck paths such that 
they cannot be written as a concatenation of Dyck paths.
When $\lambda$ is written as a concatenation of Dyck paths, {\it i.e.,} 
$\lambda=\lambda_1\circ \ldots \circ\lambda_{m}$, 
we define a path $\underline{\lambda}$ by
\begin{eqnarray*}
\underline{\lambda}
:=\vee_{|\lambda_1|/2}\circ\vee_{|\lambda_2|/2}\circ\ldots\circ\vee_{|\lambda_m|/2}.
\end{eqnarray*}
We call the region surrounded by $\lambda$ and $\underline{\lambda}$ 
as the {\it frozen region} associated with $\lambda$.
It is obvious that the frozen region associated with $\lambda$ is written as 
a concatenation of the frozen regions associated with $\lambda_{i}$ for $1\le i\le m$.
Note that we have $|\lambda_i|/2$ anchor boxes in the frozen region 
surrounded by $\lambda_{i}$ and $\underline{\lambda_i}$.
We call these anchor boxes as anchor boxes in the zeroth floor.
If we translate an anchor box in the zeroth floor upward by $(0,2m)$,
the new box is called an {\it anchor box in the $m$-th floor}. 

Figure \ref{fig:frab} is an example of a frozen region and anchor boxes.

\begin{figure}[ht]
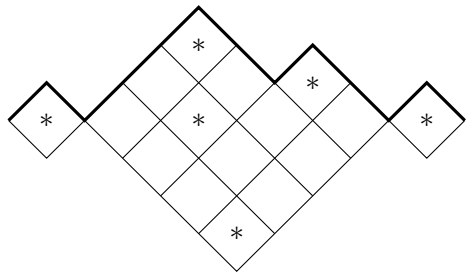

\tikzpic{-0.5}{[x=0.5cm,y=0.5cm]
\draw[very thick](0,0)--(1,1)--(2,0)--(5,3)--(7,1)--(8,2)--(10,0)--(11,1)--(12,0);
\draw(0,0)--(1,-1)--(2,0)--(6,-4)--(10,0)--(11,-1)--(12,0);
\draw(3,-1)--(6,2)(4,-2)--(7,1)(5,-3)--(9,1);
\draw(3,1)--(7,-3)(4,2)--(8,-2)(7,1)--(9,-1);
\draw(1,0)node{$\ast$}
(5,2)node{$\ast$}(5,0)node{$\ast$}
(8,1)node{$\ast$}
(6,-3)node{$\ast$}(11,0)node{$\ast$};
}
\caption{The frozen region and anchor boxes associated with a Dyck path $\lambda=UDUUUDDUDDUD$.
The lowest path is $\underline{\lambda}=DUDDDDUUUUDU$.
The boxes with $\ast$ are anchor boxes in the $0$-th floor.}
\label{fig:frab}
\end{figure}

Let $a$ be an anchor box in the zeroth floor. As mentioned before, this box is characterized 
by a pair of an up step $u$ and a down step $d$.
We take a partial path from $u$ to $d$ in $\lambda$, and obtain a partial frozen region.
This partial frozen region is said to be associated with the anchor box $a$.

An anchor box $a_1$ is said to be just below an another anchor box $a_2$ if and only if 
the edge of $\mathrm{Tree}(\lambda)$ corresponding to $a_1$ is the parent of the 
edge corresponding to $a_{2}$.
Note that the parent edge of an edge is unique if it exists.

We introduce four classes of boxes which are used to construct a Dyck tableau: 
\begin{enumerate}
\item An empty box.
\item A box with a label $i\in[1,n]$.
\item A parallel box. A line passes through from its north-west edge to its south-east edge
or from its south-west edge to its north-east edge.
\item A turn box. 
A $\vee$-turn (resp. $\wedge$-turn) box is a box with a line passing 
through from the north-west (resp. south-west) edge to 
the north-east (resp. south-east) edge. 
\end{enumerate}
Figure \ref{fig:fourboxes} shows the four classes of boxes.

\begin{figure}[ht]
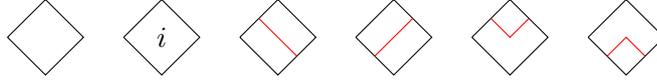

\tikzpic{-0.5}{[x=0.5cm,y=0.5cm]
\draw(0,0)--(1,1)--(2,0)--(1,-1)--(0,0);
}
\tikzpic{-0.5}{[x=0.5cm,y=0.5cm]
\draw(0,0)--(1,1)--(2,0)--(1,-1)--(0,0);
\node at (1,0){$i$};
}
\tikzpic{-0.5}{[x=0.5cm,y=0.5cm]
\draw(0,0)--(1,1)--(2,0)--(1,-1)--(0,0);
\draw[red](0.5,0.5)--(1.5,-0.5);
}
\tikzpic{-0.5}{[x=0.5cm,y=0.5cm]
\draw(0,0)--(1,1)--(2,0)--(1,-1)--(0,0);
\draw[red](0.5,-0.5)--(1.5,0.5);
}
\tikzpic{-0.5}{[x=0.5cm,y=0.5cm]
\draw(0,0)--(1,1)--(2,0)--(1,-1)--(0,0);
\draw[red](0.5,0.5)--(1,0)--(1.5,0.5);
}
\tikzpic{-0.5}{[x=0.5cm,y=0.5cm]
\draw(0,0)--(1,1)--(2,0)--(1,-1)--(0,0);
\draw[red](0.5,-0.5)--(1,0)--(1.5,-0.5);
}
\caption{Four classes of boxes. An empty box (the first picture), 
a box with the label $i$ (the second picture), parallel boxes (the third and the fourth picture)
and turn boxes (the fifth and the sixth picture).}
\label{fig:fourboxes}
\end{figure}

We put integers in $[1,n]$ in the region $R_0$ defined by $\wedge_{n}$ and $\underline{\lambda}$,
and obtain a generalized Dyck tableau. 
The algorithm to produce a Dyck tableau is as follows.

By the correspondence between an edge of $T_1$ and an anchor box in the $0$-th floor,
we put the integer $1$ in the corresponding anchor box in the frozen region.
We will put the integers $i\in[2,n]$ in the region $R_{0}$ recursively starting from $i=2$ 
and obtain a Dyck tableau of size $n$ by the following rules: 
\begin{enumerate}
\item Find an anchor box $B$ in the zeroth floor 
corresponding to the edge of $T_1$ with the integer $i$.
\item If the anchor boxes up to the $p-1$-th floor are occupied by $\vee$-turn boxes, 
we put the integer $i$ on the anchor box in the $p$-th floor. 
\item 
If anchor boxes just below the anchor box $B$ are occupied by a labeled box or a turn box up 
to the $p-1$-th floor, we put the integer $i$ on the anchor box in the $p$-th floor.

\item If the edge with the label $i-1$ is strictly right to the edge with the label
$i$, we connect by a line the anchor boxes with labels $i-1$ and $i$ in the following way:
\begin{enumerate}
\item
\label{DTab1}
 The line starts from the north-east edge of the anchor box labeled by $i$ and 
ends at the north-west edge of the anchor box labeled by $i-1$.
The line consists of north-east steps and south-east steps.
\item The line does not pass through the occupied anchor boxes with labels smaller than $i$.
\item The line can pass through the unoccupied anchor box (which does not have a label yet) 
only in the $p$-th floor as a $\vee$-turn box if the anchor boxes up to $p-1$-th floor 
are occupied by labeled boxes or $\vee$-turn boxes.  
\item When an anchor box in the $p$-th floor is labeled by the integer $1\le k\le i-1$,
we translate the partial frozen region associated with this anchor box 
upward by $(0,2p)$. Then, we redefine the frozen region 
as a union of the translated frozen region and the original frozen region.
\item 
\label{DTab4}
The line can pass through a box (not an anchor box) in the redefined frozen region 
from the south-west edge to the north-east edge or from the north-west edge 
to the south-east edge.
\item The line is the lowest path satisfying from (\ref{DTab1}) to (\ref{DTab4}). 
\end{enumerate}
\item Increase $i$ by one and apply (1) to (4) to the new $i$.
\end{enumerate}
We denote by $\mathrm{DTab}(T_{1})$ the diagram obtained from $T_1$ by the above 
procedure and call it a (generalized) {\it Dyck tableau}.

An anchor box is either a box with the label $i$ or a $\vee$-turn box.
A box (which is not an anchor box) in the frozen region is either
an empty box or a parallel box.
A box (which is not an anchor box) above the path $\lambda$ is either
an empty box, a parallel box, or a turn box.

\begin{remark}
When $\lambda$ is a zigzag path, the frozen region associated with $\lambda$ 
consists of single boxes which are anchor boxes in the zeroth floor. 
There are no empty boxes in the region $R_{0}$.
Further, boxes below a labeled box are $\vee$-turn boxes or the lower boundary path $\lambda$.
\end{remark}

Figure \ref{fig:DTabfromTree} is an example of the Dyck tableau associated with 
the natural label in Figure \ref{fig:inctreeDTSDTR}.
\begin{figure}[ht]
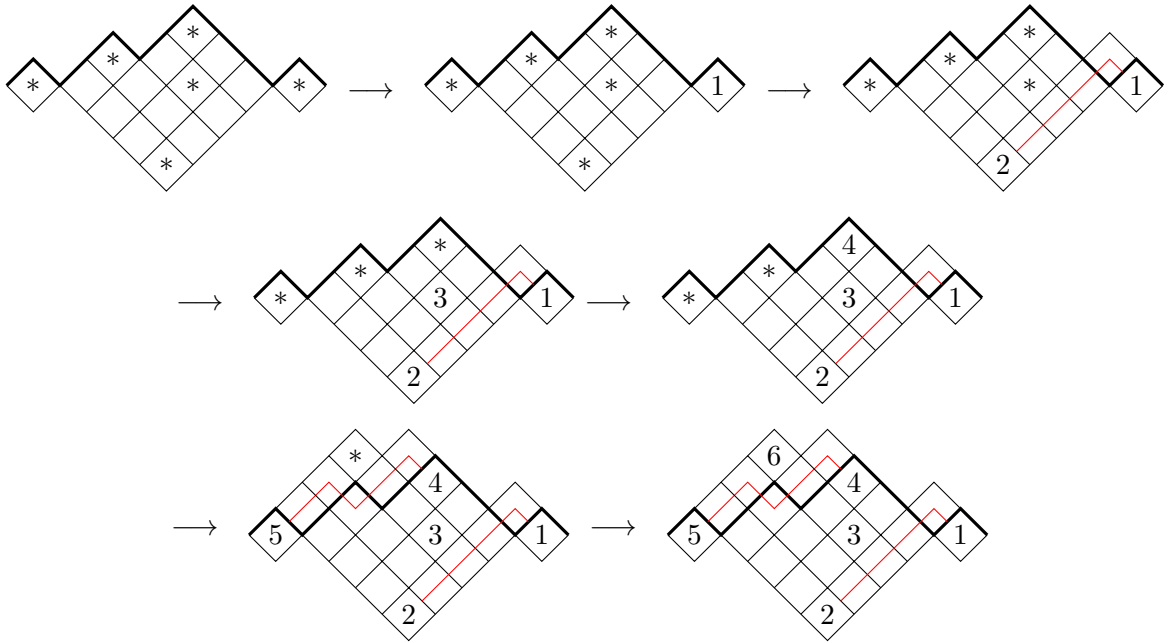

\tikzpic{-0.5}{[x=0.35cm,y=0.35cm]
\draw[very thick](0,0)--(1,1)--(2,0)--(4,2)--(5,1)--(7,3)--(10,0)--(11,1)--(12,0);
\draw(0,0)--(1,-1)--(2,0)--(6,-4)--(10,0)--(11,-1)--(12,0);
\draw(3,1)--(7,-3)(5,1)--(8,-2)(6,2)--(9,-1)(3,-1)--(5,1)(4,-2)--(8,2)(5,-3)--(9,1);
\draw(1,0)node{$\ast$}(4,1)node{$\ast$}(6,-3)node{$\ast$}
(7,2)node{$\ast$}(7,0)node{$\ast$}(11,0)node{$\ast$};
}
$\longrightarrow$
\tikzpic{-0.5}{[x=0.35cm,y=0.35cm]
\draw[very thick](0,0)--(1,1)--(2,0)--(4,2)--(5,1)--(7,3)--(10,0)--(11,1)--(12,0);
\draw(0,0)--(1,-1)--(2,0)--(6,-4)--(10,0)--(11,-1)--(12,0);
\draw(3,1)--(7,-3)(5,1)--(8,-2)(6,2)--(9,-1)(3,-1)--(5,1)(4,-2)--(8,2)(5,-3)--(9,1);
\draw(1,0)node{$\ast$}(4,1)node{$\ast$}(6,-3)node{$\ast$}
(7,2)node{$\ast$}(7,0)node{$\ast$}(11,0)node{$1$};
}
$\longrightarrow$
\tikzpic{-0.5}{[x=0.35cm,y=0.35cm]
\draw[very thick](0,0)--(1,1)--(2,0)--(4,2)--(5,1)--(7,3)--(10,0)--(11,1)--(12,0);
\draw(0,0)--(1,-1)--(2,0)--(6,-4)--(10,0)--(11,-1)--(12,0);
\draw(3,1)--(7,-3)(5,1)--(8,-2)(6,2)--(9,-1)(3,-1)--(5,1)(4,-2)--(8,2)(5,-3)--(9,1);
\draw(1,0)node{$\ast$}(4,1)node{$\ast$}(6,-3)node{$2$}
(7,2)node{$\ast$}(7,0)node{$\ast$}(11,0)node{$1$};
\draw[red](6.5,-2.5)--(10,1)--(10.5,0.5);
\draw(9,1)--(10,2)--(11,1);	
} \\[3mm]
$\longrightarrow$
\tikzpic{-0.5}{[x=0.35cm,y=0.35cm]
\draw[very thick](0,0)--(1,1)--(2,0)--(4,2)--(5,1)--(7,3)--(10,0)--(11,1)--(12,0);
\draw(0,0)--(1,-1)--(2,0)--(6,-4)--(10,0)--(11,-1)--(12,0);
\draw(3,1)--(7,-3)(5,1)--(8,-2)(6,2)--(9,-1)(3,-1)--(5,1)(4,-2)--(8,2)(5,-3)--(9,1);
\draw(1,0)node{$\ast$}(4,1)node{$\ast$}(6,-3)node{$2$}
(7,2)node{$\ast$}(7,0)node{$3$}(11,0)node{$1$};
\draw[red](6.5,-2.5)--(10,1)--(10.5,0.5);
\draw(9,1)--(10,2)--(11,1);	
}$\longrightarrow$
\tikzpic{-0.5}{[x=0.35cm,y=0.35cm]
\draw[very thick](0,0)--(1,1)--(2,0)--(4,2)--(5,1)--(7,3)--(10,0)--(11,1)--(12,0);
\draw(0,0)--(1,-1)--(2,0)--(6,-4)--(10,0)--(11,-1)--(12,0);
\draw(3,1)--(7,-3)(5,1)--(8,-2)(6,2)--(9,-1)(3,-1)--(5,1)(4,-2)--(8,2)(5,-3)--(9,1);
\draw(1,0)node{$\ast$}(4,1)node{$\ast$}(6,-3)node{$2$}
(7,2)node{$4$}(7,0)node{$3$}(11,0)node{$1$};
\draw[red](6.5,-2.5)--(10,1)--(10.5,0.5);
\draw(9,1)--(10,2)--(11,1);	
}\\[3mm]
$\longrightarrow$
\tikzpic{-0.5}{[x=0.35cm,y=0.35cm]
\draw[very thick](0,0)--(1,1)--(2,0)--(4,2)--(5,1)--(7,3)--(10,0)--(11,1)--(12,0);
\draw(0,0)--(1,-1)--(2,0)--(6,-4)--(10,0)--(11,-1)--(12,0);
\draw(3,1)--(7,-3)(5,1)--(8,-2)(6,2)--(9,-1)(3,-1)--(5,1)(4,-2)--(8,2)(5,-3)--(9,1);
\draw(1,0)node{$5$}(4,3)node{$\ast$}(6,-3)node{$2$}
(7,2)node{$4$}(7,0)node{$3$}(11,0)node{$1$};
\draw[red](6.5,-2.5)--(10,1)--(10.5,0.5)(1.5,0.5)--(3,2)--(4,1)--(6,3)--(6.5,2.5);
\draw(9,1)--(10,2)--(11,1)(1,1)--(3,3)--(4,2)--(6,4)--(7,3)(2,2)--(3,1)(5,3)--(6,2);
\draw(3,3)--(4,4)--(5,3);	
}
$\longrightarrow$
\tikzpic{-0.5}{[x=0.35cm,y=0.35cm]
\draw[very thick](0,0)--(1,1)--(2,0)--(4,2)--(5,1)--(7,3)--(10,0)--(11,1)--(12,0);
\draw(0,0)--(1,-1)--(2,0)--(6,-4)--(10,0)--(11,-1)--(12,0);
\draw(3,1)--(7,-3)(5,1)--(8,-2)(6,2)--(9,-1)(3,-1)--(5,1)(4,-2)--(8,2)(5,-3)--(9,1);
\draw(1,0)node{$5$}(4,3)node{$6$}(6,-3)node{$2$}
(7,2)node{$4$}(7,0)node{$3$}(11,0)node{$1$};
\draw[red](6.5,-2.5)--(10,1)--(10.5,0.5)(1.5,0.5)--(3,2)--(4,1)--(6,3)--(6.5,2.5);
\draw(9,1)--(10,2)--(11,1)(1,1)--(3,3)--(4,2)--(6,4)--(7,3)(2,2)--(3,1)(5,3)--(6,2);
\draw(3,3)--(4,4)--(5,3);	
}

\caption{A generation of a Dyck tableau from a natural label.}
\label{fig:DTabfromTree}
\end{figure}

\subsection{Weighted Dyck word for a general Dyck path}
In this subsection, we give a word representation for Dyck tableaux.

Let $\mathrm{Tree}(\lambda)$ be a tree for a Dyck path $\lambda$.

\begin{defn}[Position tree]
\label{def:PosTree}
A tree $\mathrm{PosTree}(\lambda)$ is a tree  such that its shape is $\mathrm{Tree}(\lambda)$ and 
its edge $e$ has the label $\mathrm{label}(e)$ given by    
\begin{eqnarray*}
\mathrm{label}(e):=2\cdot \#\{e'| e'\leftarrow e\}+\#\{e'| e\uparrow e' \}+\#\{e'|e'\uparrow e\}+1.
\end{eqnarray*}
We call $\mathrm{PosTree}(\lambda)$ the position tree of a Dyck path $\lambda$.
\end{defn}

\begin{defn}[Weighted word]
\label{defn:wdw}
Given a Dyck path $\lambda$, a {\it weighted word} for $\lambda$
is a word $w$ consisting of letters $\{\blacklozenge,U,D\}\cup\mathbb{N}$
such that 
\begin{enumerate}
\item 
\label{defn:wdw1}
the word $w$ is in the set of the language defined by
\begin{eqnarray*}
(\blacklozenge((U+D)^{\ast}\mathbb{N}^{\ast}(U+D)^{\ast})^{\ast})^{\ast}\blacklozenge,
\end{eqnarray*}
with the condition: 
the number of $w(i)\in\{U,D\}$ is twice of the number of $w(j)\in\mathbb{N}$ 
between two adjacent $\blacklozenge$'s;
\item 
\label{defn:wdw2}
we enumerate all $U$ and $D$ steps of $w$ by $1,2,\ldots$ from left to right.
We have $N$ non-negative integers after the $m$-th step, 
if and only if the position tree $\mathrm{PosTree}(\lambda)$ has $N$ edges with the label $m$;
\item 
\label{defn:wdw3}
the sub-word consisting of $U$ and $D$ is a Dyck word above $\lambda$;
\item
\label{defn:wdw4}
for each $i$,
\begin{multline*}
w(i), w(i+1),\ldots, w(i+m-1)\in\mathbb{N} 
\ \&\  w(i-1),w(i+m)\notin\mathbb{N}
 \\\Rightarrow 
0\le w(i)\le w(i+1)\le\cdots\le w(i+m-1)\le\mathrm{ch}(i,w).
\end{multline*}
Here, the {\it column height} $\mathrm{ch}(i,w)$ is defined by 
\begin{eqnarray*}
\mathrm{ch}(i,w)&=&\left\lceil\frac{1}{2}(|\{j<i|w(j)=U\}|-|\{j<i|w(j)=D\}|)\right\rceil \\
&&-|\{j|s<j<i, w(j)\in\{U,D\} \}|+|\{j|s<j<i, w(j)\in\mathbb{N}\}|+1
\end{eqnarray*}
where $s$ is the position of the rightmost $\blacklozenge$ left to $w(i)$.
\end{enumerate}
\end{defn}

We will construct a generalized Dyck tableau for a Dyck path $\lambda$ 
from a weighted word $w$ as follows.
However, we remark that not all weighted words produce Dyck tableaux. 
See Definition \ref{def:WDW} for the definition of weighted Dyck words.
By definition, the set of weighted Dyck words is bijective to the 
set of generalized Dyck tableaux for a general Dyck path.

When $w(2)=\blacklozenge$, which implies $w=\blacklozenge\blacklozenge$, we define 
$\lambda$ is an empty.
Below, we assume that $w(2)\neq\blacklozenge$.
From the condition (\ref{defn:wdw2}) in Definition \ref{defn:wdw},
we show that one can reconstruct the path $\lambda$ from the weighted word $w$.
When $w(i)\in\mathbb{N}$, let $X$ be a step $U$ or $D$ which is rightmost 
and left to $w(i)$, and $r$ be the position of $X$ in the sub-word 
of $w$ consisting of only $U$'s and $D$'s. 
We define the {\it position} of $w(i)\in\mathbb{N}$ as $r$.
Then, we get a sequence of integers $\mathbf{r}:=(r_1,r_2,\ldots,r_{n})$ 
where $r_{j}$ is the position of $j$-th letter in $\mathbb{N}$ in $w$ 
and $n$ is the number of letters in $\mathbb{N}$.

We give an algorithm to produce a tree from $\mathbf{r}$:
\begin{enumerate}
\item The tree for $\mathbf{r}=\emptyset$ is the empty tree.

\item 
Find the smallest $1\le k\le n$ and an integer $p$ associated to $k$ such that 
\begin{enumerate}
\item $p$ is maximal satisfying $r_{p}\le 2 r_{k}$,
\item $r_{k}=p$.
\end{enumerate}
If $p=n$, go to (\ref{alg:wdw4}). Otherwise, go to (\ref{alg:wdw3}).

\item 
\label{alg:wdw3}
We define two sequences $\mathbf{r}_{1}$ and 
$\mathbf{r}_{2}$ from $\mathbf{r}$:
\begin{eqnarray*}
\mathbf{r}_{1}&:=&(r_{1},\ldots,r_{p}), \\
\mathbf{r}_{2}&=&(r_{p+1}-2p,\ldots,r_{n}-2p).
\end{eqnarray*}
We attach two trees associated with $\mathbf{r}_{1}$ and $\mathbf{r}_{2}$ at their roots.
\item
\label{alg:wdw4}
We have the integer $n$ in $\mathbf{r}$, and suppose that $r_{p}=n$ for some $p\ge1$.
Let $\mathbf{r}'$ be a sequence of integers defined by 
\begin{eqnarray*}
\mathbf{r}':=(r_{1}-1,\ldots,r_{k-1}-1,r_{k+1}-1,\ldots,r_{n}-1).
\end{eqnarray*}
Then, the tree for $\mathbf{r}$ is obtained by putting an edge above the 
root of the tree for $\mathbf{r}'$.
\end{enumerate}

\begin{prop}
The above algorithm to produce a tree from $\mathbf{r}$ is well-defined.
In other words, one can find the integer $p$ such that $r_{k}=p$ in the step (\ref{alg:wdw3}).
\end{prop}
\begin{proof}
We consider the case where $\mathbf{r}$ does not have an integer $k$ 
satisfying $p\le n-1$ and conditions (2a) and (2b).
Note that the weight of an edge $e'$ is two if $e'$ is strictly left to an edge $e$ and 
the weight is one if $e'$ is above or below $e$ in the position tree.
Thus, such $k$ exists if and only if a tree for $\lambda$ can be obtained by attaching two trees 
at their roots.
In this case, a tree for $\mathbf{r}$ can not be decomposed into a concatenation of 
trees of smaller size.
This implies that there exists a unique edge $e$ connected to the root. 
Since all the other edges are below $e$, the label of $e$ in the position tree 
is equal to $n$. 
\end{proof}

\begin{remark}
Note that the path $\lambda$ is written as a concatenation of 
$q$ Dyck paths (which are not decomposed into a concatenation of Dyck paths of smaller length)
when $w$ has $q+1$ $\blacklozenge$'s. 
\end{remark}

Once we have a tree from $\mathbf{r}$, one can easily obtain a Dyck path $\lambda$.
Since the sub-word of $w$ consisting of $U$'s and $D$'s are a Dyck path $\mu$ above $\lambda$,
the top path of the generalized Dyck tableau is given by $\mu$.
We also have the frozen region associated with $\lambda$.
When $w(i)\in\mathbb{N}$, the position of $w(i)$ indicates the position of an anchor box
in the frozen region, and $w(i),\ldots,w(i+m)\in\mathbb{N}$ indicates that the dot 
corresponding the $w(j)$ for $i\le j\le i+m$ is in the $w(j)$-th floor.

For example, the weighted Dyck word $\blacklozenge UU00UU\blacklozenge DU1D0D0DD\blacklozenge$
corresponds to the generalized Dyck tableau for $\lambda=UUDDUUDUDD$:
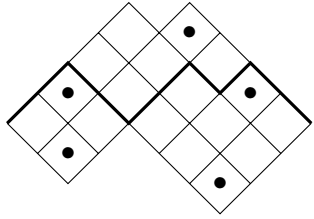
\begin{eqnarray*}
\tikzpic{-0.5}{[x=0.4cm,y=0.4cm]
\draw[very thick](0,0)--(2,2)--(4,0)--(6,2)--(7,1)--(8,2)--(10,0);
\draw(2,2)--(4,4)--(5,3)--(6,4)--(8,2)
     (0,0)--(2,-2)--(4,0)--(7,-3)--(10,0)(1,-1)--(5,3)--(6,2)(7,1)--(9,-1)
     (1,1)--(3,-1)(3,3)--(8,-2)(5,-1)--(7,1)(6,-2)--(9,1)(6,2)--(7,3);
\draw (2,-1)node{$\bullet$}(2,1)node{$\bullet$}(6,3)node{$\bullet$}(7,-2)node{$\bullet$}(8,1)node{$\bullet$};
}
\end{eqnarray*}
with $\mathbf{r}=(2,2,6,7,8)$.

\begin{defn}[Weighted Dyck word]
\label{def:WDW}
A weighted word $w$ is said to be a weighted Dyck word 
if there exists a Dyck tableau corresponding to the weighted 
word $w$.
\end{defn}

\begin{remark}
We consider the following weighted words:
\begin{eqnarray*}
\blacklozenge U0U\blacklozenge UU\alpha U1D1DD\blacklozenge D0D\blacklozenge
\end{eqnarray*}
where $\alpha=0$ or $1$.
The weighted word for $\alpha=0$ is not a weighted Dyck word.
\end{remark}

\subsection{Insertion procedure for Dyck tableaux}
The insertion procedure is the process to insert a labeled
box into a Dyck tableau. 
Since this procedure gives a recursive structure, we are able 
to construct a generation tree for Dyck tableaux.
The insertion procedure can be divided into two steps: 
addition of a labeled box and ribbon addition.

Let $\mathbf{h}:=(h_1,\ldots,h_n)$ be an insertion history such that  
$h_{i}\in[0,2(i-1)]$.

Let $\lambda$ be a Dyck path of length $2n$ and $T_1$ be a natural 
label of the tree $\mathrm{Tree}(\lambda)$.
Recall that a Dyck tableau $\mathrm{DTab}(T_{1})$ is placed in the Cartesian
coordinate system such that the Dyck path $\lambda$ starts from the origin 
of the coordinate system.
Then, we insert a labeled box at the line $x=h_{n+1}+1$ with $h_{n+1}\in[0,2n]$. 
Since $h_1=0$, we put the box with labeled by $1$ when $n=1$.

The insertion procedure for addition of a labeled box is as follows:
\begin{enumerate}
\item We divide a Dyck tableau into two pieces along the vertical line $x=h_{n+1}$
and translate the right piece right by $(2,0)$. 
\item We connect the top paths of the two pieces by the Dyck path $UD$.
Then, we put the label $n+1$ in the top box on the line $x=h_{n+1}+1$.
Similarly, the Dyck path $\lambda$ is cut into two pieces and connect them 
by the path $UD$. In this way, we obtain a new path $\lambda_{\mathrm{new}}$ 
of length $2(n+1)$.
\item Suppose that a labeled box $B$ (not necessarily in the $0$-th floor) 
in $\mathrm{DTab}(T_{1})$ corresponds to a pair of an up step $s_u$ of $\lambda$ 
and a down step $s_d$ of $\lambda$. 
If the line $x=h_{n+1}$ is placed between the up step $s_u$ and the down 
step $s_d$, we move the label by $(1,-1)$. 
If both steps $s_u$ and $s_d$ are right to the line $x=h_{n+1}$, then 
we move the label by $(2,0)$.
Otherwise, we do not move the label.
\item We change the bottom path from $\underline{\lambda}$ 
to $\underline{\lambda_{\mathrm{new}}}$ by (3).
Once the top and the bottom paths and boxes with labels 
are fixed, we put single boxes in the remaining region.
\end{enumerate}
We denote by $\mathrm{DTab}'(T_{1};h_{n+1})$ the new tableau obtained by 
adding the labeled box.

The insertion procedure for addition of a ribbon is as follows. 
If the box labeled by $n$ is right to the box labeled by $n+1$ in 
the tableau $\mathrm{DTab}'(T_1;h_{n+1})$, we put a ribbon 
(a skew Young tableau which is connected and does not contain 
a $2$-by-$2$ box) from the north-east edge of the box labeled by $n+1$ to 
the north-west edge of the box labeled by $n$.
Otherwise, we do not add a ribbon.
The new tableau is a Dyck tableau of size $n+1$.

Figure \ref{fig:exip} is an example of the insertion procedure of Dyck tableaux.

\begin{figure}[ht]
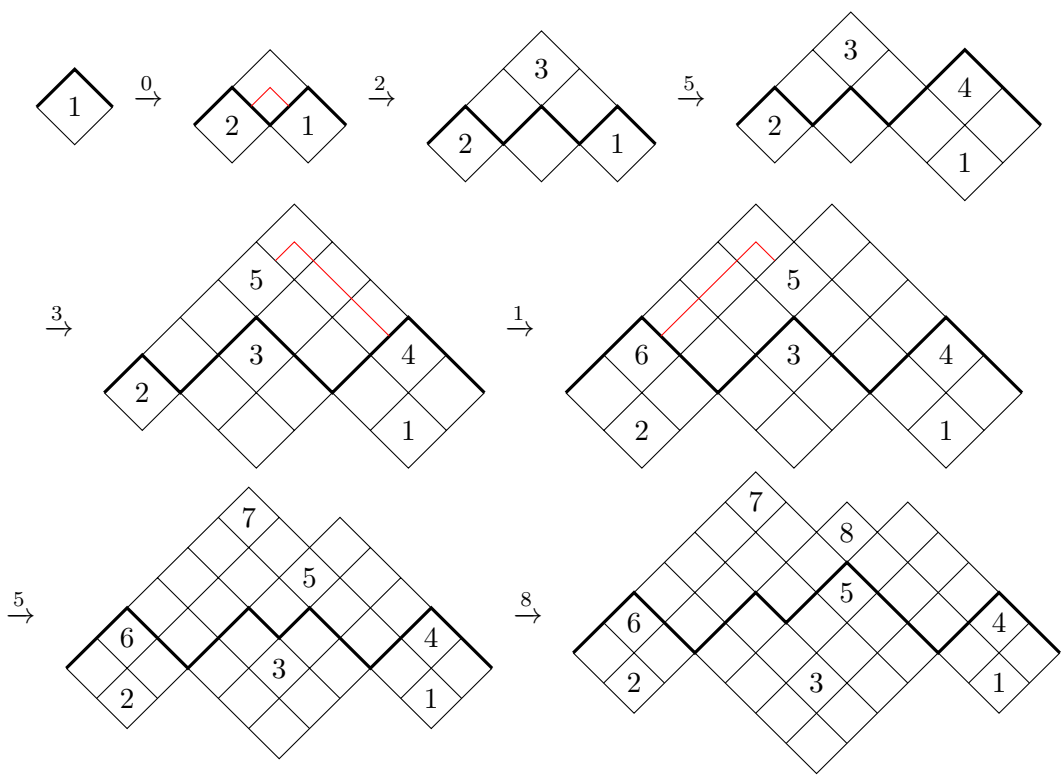

\begin{center}
\tikzpic{-0.5}{[x=0.5cm,y=0.5cm]
\draw(2,0)--(1,-1)--(0,0);
\draw[very thick](0,0)--(1,1)--(2,0);
\node at(1,0){$1$};
}
$\xrightarrow{0}$
\tikzpic{-0.5}{[x=0.5cm,y=0.5cm]
\draw(1,1)--(2,2)--(3,1)(0,0)--(1,-1)--(2,0)--(3,-1)--(4,0);
\draw[very thick](0,0)--(1,1)--(2,0)--(3,1)--(4,0);
\node at(1,0){$2$};\node at(3,0){$1$};
\draw[red](1.5,0.5)--(2,1)--(2.5,0.5);
}
$\xrightarrow{2}$
\tikzpic{-0.5}{[x=0.5cm,y=0.5cm]
\draw(1,1)--(3,3)--(5,1)(2,2)--(3,1)--(4,2)(0,0)--(1,-1)--(2,0)--(3,-1)--(4,0)--(5,-1)--(6,0);
\draw[very thick](0,0)--(1,1)--(2,0)--(3,1)--(4,0)--(5,1)--(6,0);
\node at(1,0){$2$};\node at(3,2){$3$};\node at(5,0){$1$};
}
$\xrightarrow{5}$
\tikzpic{-0.5}{[x=0.5cm,y=0.5cm]
\draw(1,1)--(3,3)--(5,1)(2,2)--(3,1)--(4,2);
\draw(0,0)--(1,-1)--(2,0)--(3,-1)--(4,0)--(6,-2)--(8,0)(5,1)--(7,-1)(5,-1)--(7,1);
\draw[very thick](0,0)--(1,1)--(2,0)--(3,1)--(4,0)--(6,2)--(8,0);
\node at(1,0){$2$};\node at(3,2){$3$};\node at(6,-1){$1$};\node at(6,1){$4$};
} \\
$\xrightarrow{3}$
\tikzpic{-0.5}{[x=0.5cm,y=0.5cm]
\draw[very thick](0,0)--(1,1)--(2,0)--(4,2)--(6,0)--(8,2)--(10,0);
\draw(1,1)--(5,5)--(8,2)(2,2)--(3,1)(3,3)--(4,2)--(6,4)(4,4)--(7,1)(5,1)--(7,3);
\draw(0,0)--(1,-1)--(2,0)--(4,-2)--(6,0)--(8,-2)--(10,0)
     (3,1)--(5,-1)(3,-1)--(5,1)(7,1)--(9,-1)(7,-1)--(9,1);
\node at (1,0){$2$};\node at(4,1){$3$};\node at(4,3){$5$};\node at(8,-1){$1$};\node at(8,1){$4$};
\draw[red](4.5,3.5)--(5,4)--(7.5,1.5);
}
$\xrightarrow{1}$
\tikzpic{-0.5}{[x=0.5cm,y=0.5cm]
\draw[very thick](0,0)--(2,2)--(4,0)--(6,2)--(8,0)--(10,2)--(12,0);
\draw(2,2)--(5,5)--(11,-1)(1,-1)--(7,5)--(10,2)(1,1)--(3,-1)(3,3)--(7,-1)(4,4)--(6,2)--(8,4)
     (5,-1)--(9,3)(9,-1)--(11,1)(0,0)--(2,-2)--(4,0)--(6,-2)--(8,0)--(10,-2)--(12,0);
\node at(2,-1){$2$};\node at(2,1){$6$};\node at(6,1){$3$};\node at(6,3){$5$};
\node at(10,-1){$1$};\node at(10,1){$4$};
\draw[red](2.5,1.5)--(5,4)--(5.5,3.5);
}\\
$\xrightarrow{5}$
\tikzpic{-0.5}{[x=0.4cm,y=0.4cm]
\draw[very thick](0,0)--(2,2)--(4,0)--(6,2)--(7,1)--(8,2)--(10,0)--(12,2)--(14,0);
\draw(2,2)--(6,6)--(13,-1)(1,-1)--(7,5)(1,1)--(3,-1)(3,3)--(8,-2)(4,4)--(6,2)(5,5)--(8,2);
\draw(6,2)--(9,5)--(12,2)(8,2)--(10,4)(5,-1)--(7,1)(6,-2)--(11,3)(7,1)--(9,-1)(11,-1)--(13,1);
\draw(0,0)--(2,-2)--(4,0)--(7,-3)--(10,0)--(12,-2)--(14,0);
\node at(2,-1){$2$};\node at(2,1){$6$};\node at(6,5){$7$};\node at(7,0){$3$};\node at(8,3){$5$};
\node at(12,1){$4$};\node at(12,-1){$1$};
}
$\xrightarrow{8}$
\tikzpic{-0.5}{[x=0.4cm,y=0.4cm]
\draw[very thick](0,0)--(2,2)--(4,0)--(6,2)--(7,1)--(9,3)--(12,0)--(14,2)--(16,0);
\draw(2,2)--(6,6)--(9,3)--(11,5)--(14,2)(1,1)--(3,-1)(1,-1)--(7,5)(3,3)--(9,-3)
     (4,4)--(6,2)--(9,5)--(15,-1)(5,5)--(11,-1);
\draw(5,-1)--(7,1)(6,-2)--(12,4)(7,-3)--(13,3)(7,1)--(10,-2)(13,-1)--(15,1);
\draw(0,0)--(2,-2)--(4,0)--(8,-4)--(12,0)--(14,-2)--(16,0);
\node at(2,1){$6$};\node at(2,-1){$2$};\node at(6,5){$7$};\node at(9,4){$8$};
\node at(9,2){$5$};\node at(8,-1){$3$};\node at(14,1){$4$};\node at(14,-1){$1$};
}

\end{center}
\caption{Insertion procedure for $\mathbf{h}=(0,0,2,5,3,1,5,8)$. The boxes with red lines are 
the ribbon added in the process.}
\label{fig:exip}
\end{figure}
 
\begin{remark}
The original definition of Dyck tableaux uses a dotted box instead of a box
labeled by an integer. Then, to add a ribbon during the insertion procedure,
the notions of an eligible box and a special box are necessary. 
See the insertion procedure for a weighted Dyck word below.
\end{remark}

To show the insertion procedure has an inverse, we introduce 
the inverse insertion procedure for $\mathbf{h}$ as follows: 
\begin{enumerate}
\item Find two boxes with the labels $n$ and $n-1$.
If the box labeled by $n-1$ is right to the box labeled by $n$,
we delete the ribbon connecting these two boxes.
\item We delete the region between the lines $x=h_{n}$ and $x=h_n+2$, 
and combine these two region at the line $x=h_{n}$. 
We delete a path $UD$ from the Dyck path $\lambda$ at $x=h_n$, and 
denote by $\lambda'_{\mathrm{new}}$ the new Dyck path of length $2n-2$.
\item Suppose that a labeled box $B$ corresponds to a pair of an up step 
$s_u$ of $\lambda$ and a down step $s_d$ of $\lambda$.
If the line $x=h_{n}$ is placed between $s_u$ and $s_d$, we move the label
of $B$ by $(-1,1)$. If $s_u$ and $s_d$ are right to the line $x=h_{n}$, 
we move the label of $B$ by $(-2,0)$. Otherwise, we do not move the label.
\item We change the lower boundary $\underline{\lambda}$ 
to $\underline{\lambda'_{\mathrm{new}}}$. 
Since the top and lowest boundaries, and the positions of labeled boxes 
are fixed, we put single boxes in the remaining region.
\end{enumerate}
In this way, we obtain a Dyck tableau of size $n-1$.
The following proposition is obvious from the definition of inverse insertion
procedure.
\begin{prop}
\label{prop:iipDT}
The inverse insertion procedure is the inverse of the insertion procedure.
\end{prop}

\paragraph{\bf Insertion procedure for weighted Dyck words}
Let $\mathrm{DTab}(T)$ be a Dyck tableau and $w$ its weighted Dyck word.
We call {\it column addition} in $w$ is the substitutions
\begin{eqnarray}
\blacklozenge&\rightarrow& \blacklozenge UmD\blacklozenge,
\label{wdw:sub1} \\
\mathbb{N}^{p}&\rightarrow&U\mathbb{N}^{p}mD,
\label{wdw:sub2}
\end{eqnarray}
where $m=\mathrm{ch}(i,w)$. Here, $i$ is the position of the substituted 
letter in $\mathbb{N}$, that is $m$ in Eqn. (\ref{wdw:sub1}) and (\ref{wdw:sub2}), in $w$.
In case of (\ref{wdw:sub2}), let $p$ be the position of a letter in $\mathbb{N}$
right to the position $i$ and left to the leftmost $\blacklozenge$ right to the 
position $i$.
Then, we decrease all the values $p$'s satisfying the above condition by one.
This operation corresponds to the insertion of a column during the 
addition of a labeled box for a Dyck tableau.

For example, if we insert $U0D$ after the first $0$ in $\blacklozenge UUU0U00D0DDD\blacklozenge$
gives 
\begin{eqnarray*}
\blacklozenge UUU0U00D0DDD\blacklozenge &\rightarrow&\blacklozenge UUUU00DU00D0DDD\blacklozenge \\
&\rightarrow&\blacklozenge UUUU00D00U0DDDD\blacklozenge.
\end{eqnarray*}

\begin{defn}[Section 2 in \cite{ABDH11}]
Let $D$ be a Dyck tableau corresponding to a weighted Dyck word $w$.
Then, suppose that the weight $w(i)$ corresponds to a labeled box in the Dyck tableau
whose top path is on the top path of $D$.
An eligible weight is a letter $w(i-1)=U$ such that $w(i)=\mathrm{ch}(i,w)$.
A special weight is the right-most eligible weight in a weighted 
Dyck word.
\end{defn}

From Proposition 3 in \cite{ABDH11}, a weighted Dyck word has always 
a unique special weight. 
We denote by $s$ the special weight in a weighted Dyck word.

A {\it ribbon addition} in $w$ is an operation exchanging 
the step $D$ which is added in the column insertion and the 
step $U$ of the special weight if $s$ is right to the added 
step $D$.
All the other letters are not changed.
As in Proposition 2 in \cite{ABDH11}, a ribbon addition transforms a weighted 
Dyck tableau into another Dyck tableau of the same size.

We are ready to introduce the insertion procedure for a weighted 
Dyck tableau of size $n$.
The procedure consists of three steps:
\begin{enumerate}
\item Find the special weight $s$.
\item Perform a column addition (\ref{wdw:sub1}) or (\ref{wdw:sub2}) 
at the position of a $\blacklozenge$ or $\mathbb{N}^{p}$.
\item Perform a ribbon addition to the weighted Dyck word obtained 
in the second step.
\end{enumerate}
By the insertion procedure, we obtain a weighted Dyck word of size $n+1$ from 
a weighted Dyck word of size $n$.

\begin{theorem}
Every Dyck tableau (resp. equivalently weighted Dyck word) can be constructed 
from a box with the label $1$ (resp. $\blacklozenge U0D\blacklozenge$) 
by the insertion procedure recursively.
\end{theorem}
\begin{proof}
We prove Theorem by induction on the size $n$.
When $n=1$, we have a box with the label $1$ or $\blacklozenge U0D\blacklozenge$.
We assume that Theorem holds true up to the size $n$.
Let $D$ (resp. $w$) be a Dyck tableau (resp. a weighted Dyck word) of size 
$n+1$.
By inverse insertion procedure, we obtain a Dyck tableau $D'$ (resp. a weighted 
Dyck word $w'$) of size $n$.
By induction assumption, $D'$ (resp. $w'$) can be constructed from the Dyck tableau
(resp. the weighted Dyck word) of size $1$ using the insertion procedure.
From Proposition \ref{prop:iipDT}, $D$ (resp. $w$) can be constructed from 
$D'$ (resp. $w'$) by the insertion procedure. 
\end{proof}

We show the generation tree for Dyck tableaux in $\mathcal{T}_{n}$ up to $n=3$
in Figure \ref{GTDT3}.
The label $i$ on an arrow indicates the insertion procedure at $x=i$.
\begin{figure}[ht]
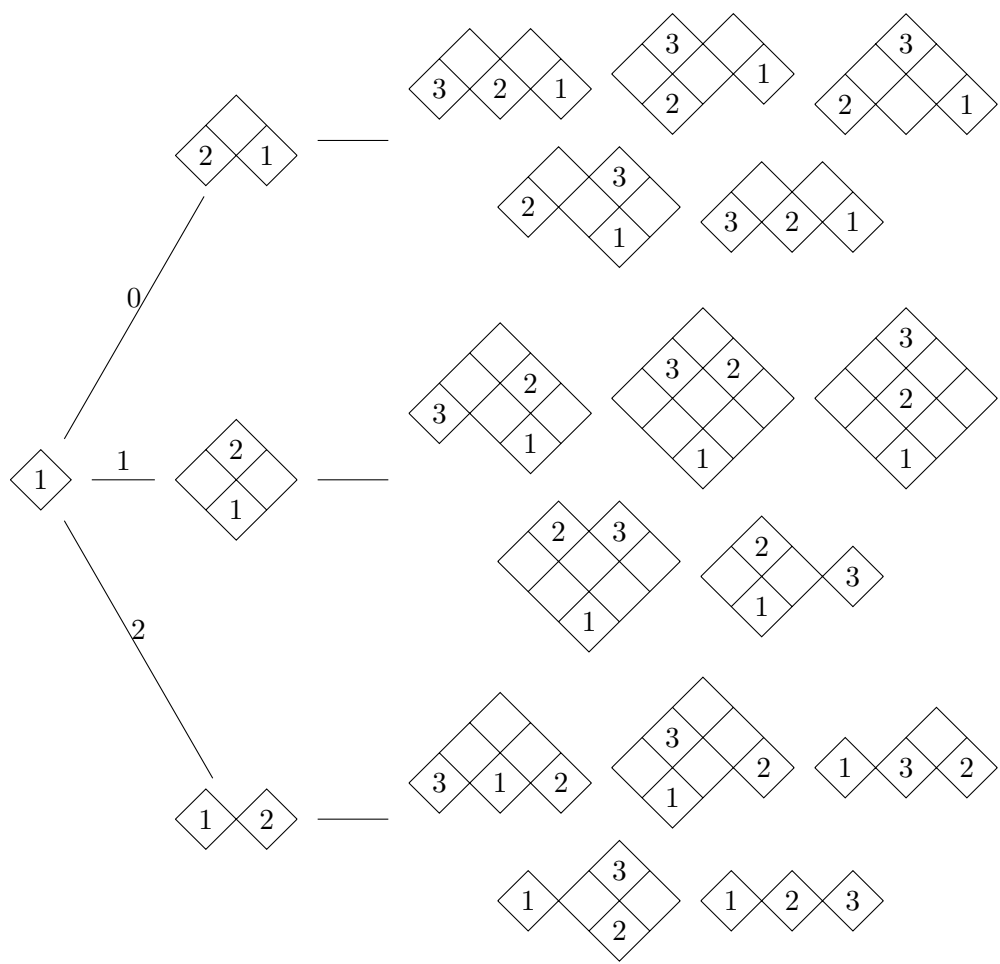

\tikzpic{-0.5}{[x=0.4cm,y=0.4cm,grow=right]
\tikzstyle{level 1}=[level distance=1.5cm, sibling distance=4.5cm]
\tikzstyle{level 2}=[level distance=2cm, sibling distance=1.5cm]
\node[text centered,anchor=east]{
\tikzpic{-0.5}{\draw(0,0)--(1,1)--(2,0)--(1,-1)--(0,0);
	       \node[anchor=center] at (1,0){$1$};
	       }
      }
child {
      node[anchor=west] {
           \tikzpic{-0.5}{\draw(0,0)--(1,1)--(3,-1)--(4,0)(0,0)--(1,-1)--(3,1)--(4,0);
                          \node[anchor=center] at (1,0){$1$};\node[anchor=center] at (3,0){$2$};
                         }
           }
      child {
             node[anchor=west]{$\begin{aligned}
                            \tikzpic{-0.5}{
                                          \draw(0,0)--(3,3)--(6,0)(0,0)--(1,-1)--(4,2)
                                               (1,1)--(3,-1)--(5,1)(2,2)--(5,-1)--(6,0);
                                          \node[anchor=center] at (1,0){$3$};
                                          \node[anchor=center] at (3,0){$1$};
                                          \node[anchor=center] at (5,0){$2$};
                                          }    
                            \tikzpic{-0.5}{
                                          \draw(0,0)--(3,3)--(6,0)(0,0)--(2,-2)--(5,1)
                                               (1,1)--(3,-1)(2,2)--(5,-1)--(6,0)(1,-1)--(4,2);
                                          \node[anchor=center] at (2,1){$3$};
                                          \node[anchor=center] at (5,0){$2$};
                                          \node[anchor=center] at (2,-1){$1$};
                                          }   
                            \tikzpic{-0.5}{
                                          \draw(0,0)--(1,1)--(2,0)--(4,2)--(6,0)(0,0)--(1,-1)--(2,0)--(3,-1)--(5,1)
                                               (3,1)--(5,-1)--(6,0);
                                          \node[anchor=center] at (5,0){$2$};
                                          \node[anchor=center] at (3,0){$3$};
                                          \node[anchor=center] at (1,0){$1$};
                                          } \\
                             \tikzpic{-0.5}{
                                          \draw(0,0)--(1,1)--(2,0)--(4,2)--(6,0)(0,0)--(1,-1)--(2,0)--(4,-2)--(6,0)
                                               (3,1)--(5,-1)(3,-1)--(5,1);
                                          \node[anchor=center] at (4,-1){$2$};
                                          \node[anchor=center] at (4,1){$3$};
                                          \node[anchor=center] at (1,0){$1$};
                                          }
                             \tikzpic{-0.5}{
                                          \draw(0,0)--(1,1)--(2,0)--(3,1)--(4,0)--(5,1)--(6,0)
                                               (0,0)--(1,-1)--(2,0)--(3,-1)--(4,0)--(5,-1)--(6,0);
                                          \node[anchor=center] at (5,0){$3$};
                                          \node[anchor=center] at (3,0){$2$};
                                          \node[anchor=center] at (1,0){$1$};
                                          }\hspace{1.5cm} \end{aligned}$ 
                            }
            }
      edge from parent node[above]{$2$}
      }
child {
      node[anchor=west] {
           \tikzpic{-0.5}{\draw(0,0)--(2,2)--(4,0)(0,0)--(2,-2)--(4,0)(1,1)--(3,-1)(1,-1)--(3,1);
                          \node[anchor=center] at (2,1){$2$};\node[anchor=center] at (2,-1){$1$};
                         }
           }
      child {
             node[anchor=west]{$\begin{aligned}
                            \tikzpic{-0.5}{
                                          \draw(0,0)--(3,3)--(6,0)(0,0)--(1,-1)--(4,2)
                                               (1,1)--(4,-2)--(6,0)(2,2)--(5,-1)(3,-1)--(5,1);
                                          \node[anchor=center] at (1,0){$3$};
                                          \node[anchor=center] at (4,1){$2$};
                                          \node[anchor=center] at (4,-1){$1$};
                                          }    
                            \tikzpic{-0.5}{
                                          \draw(0,0)--(3,3)--(6,0)(0,0)--(3,-3)--(6,0)
                                               (1,1)--(4,-2)(2,2)--(5,-1)(1,-1)--(4,2)(2,-2)--(5,1);
                                          \node[anchor=center] at (2,1){$3$};
                                          \node[anchor=center] at (4,1){$2$};
                                          \node[anchor=center] at (3,-2){$1$};
                                          }   
                            \tikzpic{-0.5}{
                                          \draw(0,0)--(3,3)--(6,0)(0,0)--(3,-3)--(6,0)
                                               (1,1)--(4,-2)(2,2)--(5,-1)(1,-1)--(4,2)(2,-2)--(5,1);
                                          \node[anchor=center] at (3,2){$3$};
                                          \node[anchor=center] at (3,0){$2$};
                                          \node[anchor=center] at (3,-2){$1$};
                                          } \\
                             \tikzpic{-0.5}{
                                          \draw(0,0)--(2,2)--(5,-1)(0,0)--(3,-3)--(6,0)
                                               (1,1)--(4,-2)(1,-1)--(4,2)--(6,0)(2,-2)--(5,1);
                                          \node[anchor=center] at (4,1){$3$};
                                          \node[anchor=center] at (2,1){$2$};
                                          \node[anchor=center] at (3,-2){$1$};
                                          }
                             \tikzpic{-0.5}{
                                          \draw(0,0)--(2,2)--(4,0)--(5,1)--(6,0)
                                               (0,0)--(2,-2)--(4,0)--(5,-1)--(6,0)
                                               (1,1)--(3,-1)(1,-1)--(3,1);
                                          \node[anchor=center] at (5,0){$3$};
                                          \node[anchor=center] at (2,1){$2$};
                                          \node[anchor=center] at (2,-1){$1$};
                                          }\hspace{1.5cm} \end{aligned}$ 
                            }
            }
            edge from parent node[above]{$1$}
      }
child {
      node[anchor=west] {
           \tikzpic{-0.5}{\draw(0,0)--(2,2)--(4,0)(0,0)--(1,-1)--(3,1)(1,1)--(3,-1)--(4,0);
                          \node[anchor=center] at (1,0){$2$};\node[anchor=center] at (3,0){$1$};
                         }
           }
      child{
           node[anchor=west]{$\begin{aligned}
                            \tikzpic{-0.5}{
                                          \draw(0,0)--(2,2)--(5,-1)--(6,0)(0,0)--(1,-1)--(4,2)--(6,0)
                                               (1,1)--(3,-1)--(5,1);
                                          \node[anchor=center] at (1,0){$3$};
                                          \node[anchor=center] at (3,0){$2$};
                                          \node[anchor=center] at (5,0){$1$};
                                          }    
                            \tikzpic{-0.5}{
                                          \draw(0,0)--(2,2)--(3,1)--(4,2)--(6,0)(0,0)--(2,-2)--(5,1)
                                               (1,1)--(3,-1)(1,-1)--(3,1)--(5,-1)--(6,0);
                                          \node[anchor=center] at (2,1){$3$};
                                          \node[anchor=center] at (2,-1){$2$};
                                          \node[anchor=center] at (5,0){$1$};
                                          }   
                            \tikzpic{-0.5}{
                                          \draw(0,0)--(2,2)--(5,-1)--(6,0)(0,0)--(1,-1)--(4,2)--(6,0)
                                               (1,1)--(3,-1)--(5,1)(2,2)--(3,3)--(4,2);
                                          \node[anchor=center] at (1,0){$2$};
                                          \node[anchor=center] at (3,2){$3$};
                                          \node[anchor=center] at (5,0){$1$};
                                          } \\
                            \tikzpic{-0.5}{
                                          \draw(0,0)--(2,2)--(5,-1)--(6,0)(0,0)--(1,-1)--(4,2)--(6,0)
                                               (1,1)--(3,-1)--(5,1)(3,-1)--(4,-2)--(5,-1);
                                          \node[anchor=center] at (1,0){$2$};
                                          \node[anchor=center] at (4,1){$3$};
                                          \node[anchor=center] at (4,-1){$1$};
                                          }
                             \tikzpic{-0.5}{
                                          \draw(0,0)--(2,2)--(5,-1)--(6,0)(0,0)--(1,-1)--(4,2)--(6,0)
                                               (1,1)--(3,-1)--(5,1);
                                          \node[anchor=center] at (1,0){$3$};
                                          \node[anchor=center] at (3,0){$2$};
                                          \node[anchor=center] at (5,0){$1$};
                                          }\hspace{1.5cm} \end{aligned}$ 
                            }
           }
      edge from parent node[above]{$0$}
      };
}
\caption{Generation tree for Dyck tableaux of size at most $3$}
\label{GTDT3}
\end{figure}

\subsection{Dyck tableaux and cover-inclusive Dyck tilings}

We construct a bijection from a Dyck tableau $\mathrm{DTab}(T_1)$ to a Dyck tiling $D$.
Recall that an anchor box in the zeroth floor corresponds to a pair of an up step
and a down step in $\lambda$.
When an anchor box in $\mathrm{DTab}(T_1)$ is in the $p$-th floor, 
we have $p$ non-trivial Dyck tiles (not a single box) above the corresponding pair 
of the up step and the down step.

The lower boundary of the Dyck tiling $D$ is given by the path $\lambda$.
The top path $\mu$ of $D$ is determined by the path satisfying
\begin{enumerate}
\item \label{DTtop1}
The path is above labeled boxes, parallel boxes and turn boxes.
\item \label{DTtop2}
The path is above the boxes forming a ribbon in the insertion procedure 
of $\mathrm{DTab}(T_{1})$.
\item The lowest path with the properties (\ref{DTtop1}) and (\ref{DTtop2}).
\end{enumerate}
We may have empty boxes below $\mu$.

Once we fix non-trivial Dyck tiles, the lowest and top paths $\lambda$ and $\mu$, we put 
single boxes in the remaining region.
In this way, we obtain a Dyck tiling.
We denote by $\phi_0$ the above map from Dyck tableaux to Dyck tilings.
Figure \ref{fig:TDTDTR} is an example of the map $\phi_{0}$.

\begin{figure}[ht]
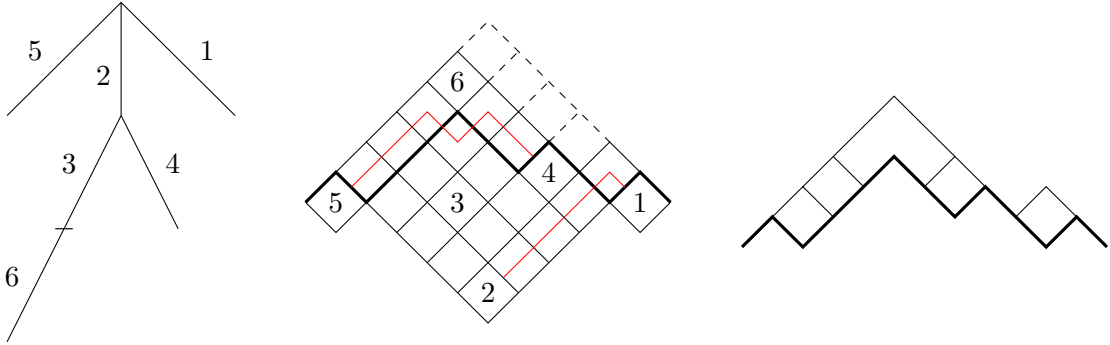

\tikzpic{-0.5}{
\coordinate
	child{coordinate (c5)}
	child{coordinate (c2)
	child{coordinate (c3)
	child{coordinate (c6)}
	child[missing]
	}
	child{coordinate (c4)}
	}
	child{coordinate (c1)};
\node[anchor=south east] at ($(0,0)!.6!(c5)$){5};
\node[anchor=east] at ($(0,0)!.65!(c2)$){2};
\node[anchor=south west] at ($(0,0)!.6!(c1)$){1};
\node[anchor=south east] at ($(c2)!.6!(c3)$){3};
\node[anchor=south east] at ($(c3)!.6!(c6)$){6};
\node[anchor=south west] at ($(c2)!.6!(c4)$){4};
\draw(c3)node{$-$};
}
\quad
\tikzpic{-0.5}{[x=0.4cm,y=0.4cm]
\draw[very thick](0,0)--(1,1)--(2,0)--(5,3)--(7,1)--(8,2)--(10,0)--(11,1)--(12,0);
\draw(0,0)--(1,-1)--(2,0)--(6,-4)--(10,0)--(11,-1)--(12,0);
\draw(3,-1)--(6,2)(4,-2)--(7,1)(5,-3)--(9,1);
\draw(3,1)--(7,-3)(4,2)--(8,-2)(7,1)--(9,-1);
\draw(1,0)node{$5$}
(5,4)node{$6$}(5,0)node{$3$}
(8,1)node{$4$}
(6,-3)node{$2$}(11,0)node{$1$};
\draw[dashed](5,5)--(6,6)--(10,2)(6,4)--(7,5)(7,3)--(8,4)(8,2)--(9,3);
\draw(1,1)--(5,5)--(8,2)(2,2)--(3,1)(3,3)--(4,2)(4,4)--(5,3)
(5,3)--(6,4)(6,2)--(7,3)(9,1)--(10,2)--(11,1);
\draw[red](6.5,-2.5)--(10,1)--(10.5,0.5)(1.5,0.5)--(4,3)--(5,2)--(6,3)--(7.5,1.5);
}
\quad
\tikzpic{-0.5}{[x=0.4cm,y=0.4cm]
\draw[very thick](0,0)--(1,1)--(2,0)--(5,3)--(7,1)--(8,2)--(10,0)--(11,1)--(12,0);
\draw(1,1)--(5,5)--(8,2)(9,1)--(10,2)--(11,1)(2,2)--(3,1)(3,3)--(4,2)
      (6,2)--(7,3);
}
\caption{A natural label of a tree (the left picture), a generalized Dyck tableau 
associated with the natural label (the middle picture) and the cover-inclusive 
Dyck tiling for the Dyck tableau (the right picture).
}
\label{fig:TDTDTR}
\end{figure}

\begin{theorem}
The map $\phi_0$ is a bijection between the cover-inclusive Dyck tilings whose lower
path is $\lambda$ and Dyck tableaux characterized by natural labels of $\mathrm{Tree}(\lambda)$.
\end{theorem}
\begin{proof}
We prove Theorem by induction with respect to the size $n$ of a Dyck tableau.
We have a unique Dyck tableau of size one, {\it i.e.,} a single labeled box.
The upper and lower paths of the Dyck tableau are a path $UD$. 
The DTR bijection for a tree with one edge gives a cover-inclusive 
Dyck tiling whose upper and lower paths are $UD$.
In both cases, the Dyck tableau and the DTR bijection give the same Dyck tiling.
Theorem is true for $n=1$.

Suppose that Theorem is true for the size $n-1$. 
Let $T_{1}$ be a natural label of a tree $\mathrm{Tree}(\lambda)$ of size $n-1$,
$\mathrm{DTab}(T_{1})$ be a Dyck tableau associated with $T_{1}$, 
and $\mathrm{DTR}(T_{1})$ be a Dyck tiling obtained by the DTR bijection on $T_{1}$. 
We denote by $T$ a natural label of size $n$ obtained from $T_{1}$ by attaching 
a single edge at a node of $T_{1}$. 
We want to show 
\begin{eqnarray}
\label{eqn:DTRDTab}
\mathrm{DTR}(T)=\phi_{0}(\mathrm{DTab}(T)).
\end{eqnarray}
Eqn. (\ref{eqn:DTRDTab}) for $n-1$ implies that the lowest paths and top paths of 
$\mathrm{DTR}(T_1)$ coincides with the ones of $\phi_{0}(\mathrm{DTab}(T_1))$.
It is clear that the insertion procedures of the DTR bijection and a Dyck tableau 
produce the same path $\lambda$.
Similarly, the top path after the spread of $\mathrm{DTR}(T_{1})$ at $x=m$ coincides with 
the top path after the addition of a labeled box.
We add a ribbon after the spread in case of the DTR bijection and after the addition of 
a labeled box in case of insertion process of a Dyck tableau.
Note that the top box at the special column in the DTR bijection is nothing but 
the box with the label $n-1$ in the Dyck tableau.
Thus, the top paths after the ribbon addition are the same.

Suppose we perform a spread of a Dyck tiling at $x=m$.
Let $p$ be the number of Dyck tiles above $\lambda$ at $x=m$ in the Dyck tiling.
We spread these $p$ Dyck tiles of length $l$ to the $p$ Dyck tiles of length $l+1$.
By induction assumption, the spread is equivalent to perform the addition of 
the labeled box in a Dyck tableau at $x=m$.
Then, the unoccupied anchor box is at the $p$-th floor at $x=m$.
By addition of the labeled box, we put the label $n$ at the $p$-th floor.
Thus, by $\phi_{0}$, $p$ is interpreted as the number of Dyck tiles in the DTR bijection
and as the $p$-th floor in the Dyck tableau. 

Above arguments implies that $\mathrm{DTR}(T)$ and $\phi_{0}(\mathrm{DTab}(T))$ have 
the same top and lowest paths, and all Dyck tiles of length larger than zero have the same length and 
these Dyck tiles are positioned at the same place.
These two Dyck tilings are to be same, that is, we have Eqn. (\ref{eqn:DTRDTab}).
\end{proof}

\subsection{Generalized patterns and shadow and clear boxes of Dyck tableaux}
In \cite{ABDH11}, they study several generalized patterns in permutations and 
their relations to Dyck tableaux.
The result in \cite{ABDH11} can be generalized to Dyck tableaux for general Dyck 
paths (not necessarily zigzag paths).
In this subsection, we study generalized patterns on a label of the 
tree $\mathrm{Tree}(\lambda)$ and their relations to Dyck tableaux.

\begin{defn}[Definition of shadow and clear boxes in \cite{ABDH11}]
\label{def:shadowclear}

In a Dyck tableau, let $a$ be a box, which is either a parallel or turn 
box, and $b$ be a labeled box.
We call $a$ a shadow (resp. clear) box if $a$ is above (resp. below) $b$ such 
that there is no labeled box between $a$ and $b$.

\end{defn}

We define shadow and clear boxes in a Dyck tableau in Definition \ref{def:shadowclear}.
For later purpose, we refine its definition as follows.
Let $D=\mathrm{DTab}(L)$ be a Dyck tableau for a natural label $L$.
Recall that $D$ consists of labeled boxes, paths connecting two labeled boxes
and empty boxes.
If we have a $\wedge$-turn box $t$ such that a box below $t$ is an empty 
boxes, we locally transform the path containing $t$ as Figure \ref{fig:lmove}.
\begin{figure}[ht]
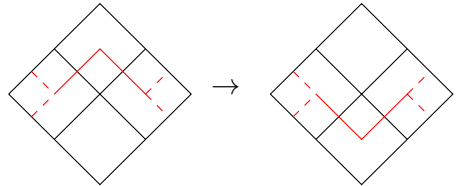

\tikzpic{-0.5}{[scale=0.6]
\draw(0,0)--(2,2)--(4,0)--(2,-2)--(0,0);
\draw(1,1)--(3,-1)(1,-1)--(3,1);
\draw[red](1,0)--(2,1)--(3,0);
\draw[red,dashed](0.5,0.5)--(1,0)(0.5,-0.5)--(1,0);
\draw[red,dashed](3,0)--(3.5,0.5)(3,0)--(3.5,-0.5);
}
$\rightarrow$
\tikzpic{-0.5}{[scale=0.6]
\draw(0,0)--(2,2)--(4,0)--(2,-2)--(0,0);
\draw(1,1)--(3,-1)(1,-1)--(3,1);
\draw[red](1,0)--(2,-1)--(3,0);
\draw[red,dashed](0.5,0.5)--(1,0)(0.5,-0.5)--(1,0);
\draw[red,dashed](3,0)--(3.5,0.5)(3,0)--(3.5,-0.5);
}
\caption{A local move for a path}
\label{fig:lmove}
\end{figure}
In the process of a local move, we never move the positions of 
labeled boxes.
We denote by $D'$ the tableau obtained from $D$ by local moves such that 
we can not apply a local move to any $\wedge$-turn box in $D'$.
This means that one cannot lower paths in $D'$.

Note that a shadow (resp. clear) box in $D$ is obviously bijective 
to a shadow (resp. clear) box in $D'$.

\begin{defn}
Let $D'$ be the tableau obtained from $D$ as above.
Let $s$ be a shadow (resp. clear) box above (resp. below) a labeled box 
$b$ in a Dyck tableau $D$. 
We call $s$ a proper shadow (resp. clear) box if there is neither 
empty boxes nor a labeled box below (resp. above) $s$
and above (resp. below) $b$ in $D'$.
\end{defn}

Let $T_1$ be a natural label of $\mathrm{Tree}(\lambda)$ and 
$e(i)$ for $1\le i\le n$ be the edge of $\mathrm{Tree}(\lambda)$ 
labeled by the integer $i$.
We have a position tree $\mathrm{PosTree}(\lambda)$ for $\lambda$ (see Definition \ref{def:PosTree})
and let $\mathrm{Pos}(e(i))$ be the label of the edge $e(i)$ 
in $\mathrm{PosTree}(\lambda)$.

A {\it pattern} $2^+2$ of a natural label $T_1$ is a relation 
of $e(a)$ and $e(b)$ such that $a=b+1$ and $e(a)\rightarrow e(b)$.
A {\it pattern} $2^+12$ of a natural label $T_{1}$ is a relation
among $e(a), e(b)$ and $e(c)$ such that 
\begin{enumerate}
\item $b<c$ and $a=c+1$,
\item $\mathrm{Pos}(e(a))<\mathrm{Pos}(e(b))<\mathrm{Pos}(e(c))$, and 
\item there is no $b'$ such that $\mathrm{Pos}(e(b))=\mathrm{Pos}(e(b'))$ and $b<b'<c$.
\end{enumerate}
A {\it pattern} $1^+21$ of a natural label $T_{1}$ is a relation 
among $e(a), e(b)$ and $e(c)$ such that 
\begin{enumerate}
\item $b>a$ and $a=c+1$, 
\item $\mathrm{Pos}(e(a))<\mathrm{Pos}(e(b))<\mathrm{Pos}(e(c))$, and 
\item there is no $b'$ such that $\mathrm{Pos}(e(b))=\mathrm{Pos}(e(b'))$ and $b>b'>a$.
\end{enumerate}

\begin{prop}[Generalization of Proposition 8 in \cite{ABDH11}]
\label{prop:2+2}
An added ribbon of a Dyck tableau is in bijection with the patterns $2^+2$ of 
$T_{1}$.
In $\mathrm{DTab}(T_1)$, if we read from left to right the labels of 
boxes that are connected by a ribbon, we get the pattern $2^+2$ for $T_1$.
\end{prop}
\begin{proof}
Let $e(a)$ and $e(b)$ be labeled boxes satisfying $2^{+}2$ pattern.
The box labeled by $a$ is inserted immediately after the box labeled by $b$.
Since $e(a)\rightarrow e(b)$, the box labeled by $a$ is left to the box labeled 
by $b$. Therefore, there is a ribbon between $a$ and $b$ in the Dyck tableau.

Conversely, suppose labeled boxes $a$ and $b$ are connected by a ribbon. 
By the inverse of insertion algorithm, we remove the box labeled by $a$
immediately before the box labeled by $b$ and $e(a)$ is left to $e(b)$. 
This means $a=b+1$.
Further, if $e(a)\uparrow e(b)$ with $a=b+1$, we have no ribbon connecting the labeled
boxes $a$ and $b$. Thus, we have $e(a)\rightarrow e(b)$.
\end{proof}

\begin{prop}[Generalization of Proposition 9 in \cite{ABDH11}]
Shadow boxes of $T$ are bijective to the patterns $2^{+}12$.
Clear boxes of $T$ are bijective to the patterns $1^{+}21$.
\end{prop}
\begin{proof}
Let $abc$ be a pattern $2^{+}12$ in a natural label $T_{1}$.
Since $\mathrm{Pos}(e(a))<\mathrm{Pos}(e(b))<\mathrm{Pos}(e(c))$,
the box labeled by $b$ is between the columns of the box labeled by 
$a$ and $c$.
Since $ac$ is a pattern $2^{+}2$, Proposition \ref{prop:2+2} implies that 
we add a ribbon after the insertion of the box labeled by $b$.
The condition about the box $b'$ implies that the column of $b$ intersects 
with the ribbon and this intersected box is a 
shadow box above $b$.
Suppose that $abc$ and $adc$ be two different $2^{+}12$ pattern.
The condition about the box $b'$ implies that 
$\mathrm{Pos}(b)\neq\mathrm{Pos}(d)$.
Thus, $abc$ and $adc$ give different  
shadow boxes.

Conversely, take a shadow box.
This box is the intersection of column $b$ and a ribbon connecting 
$a$ and $c$. 
It is obvious that $b<a$, which implies $abc$ is the generalized pattern $2^{+}12$.

The proof is the same for the pattern $1^{+}21$.
\end{proof}

\subsection{The shape of a Dyck tableau}
Let $\lambda$ and $\mu$ be Dyck paths satisfying $\lambda\le \mu$,
and $T_{1}$ a label of the tree $\mathrm{Tree}(\lambda)$.
Recall that given an anchor box $a$ at the zeroth floor, we have 
a corresponding pair of $U$ and $D$ steps in $\lambda$.
Let $i_{U}$ and $i_{D}$ be the position of these $U$ and $D$ steps 
in $\lambda$ from left. 
This pair of $U$ and $D$ steps corresponds to an edge $e$
in $\mathrm{Tree}(\lambda)$.
We denote by $T_{1}(e)$ the label of the edge $e$ in $T_1$ and 
by $e^{+}$ (resp. $e^{-}$) the edge whose label in $T_{1}$ is 
given by $T_{1}(e)+1$ (resp. $T_{1}(e)-1$).
We denote by $\mathrm{lb}(e)$ (resp. $\mathrm{rb}(e)$)
the step of the top path $\mu$ of the Dyck tableau $\mathrm{DTab}(T_{1})$ 
at position $i_{U}$ (resp. $i_{D}$).
We call $\mathrm{lb}(e)$ (resp. $\mathrm{rb}(e)$) left border 
(resp. right border) for $e$.

\begin{prop}
\label{prop:lbrb}
The left border for the edge $e$ in $\mathrm{Tree}(\lambda)$ 
is obtained by
\begin{eqnarray*}
\mathrm{lb}(e)=
\begin{cases}
U & \text{if $T_{1}(e)=n$}, \\
U & \text{if $e\rightarrow e^{+}$ or $e^{+}\uparrow e$}, \\
D & \text{if $e^{+}\rightarrow e$}.
\end{cases}
\end{eqnarray*}
The right border for the edge $e$ in $\mathrm{Tree}(\lambda)$ is obtained by
\begin{eqnarray*}
\mathrm{rb}(e)=
\begin{cases}
D & \text{if $T_{1}(e)=1$}, \\
D & \text{if $e^{-}\rightarrow e$ or $e\uparrow e^{-}$}, \\
U & \text{if $e\rightarrow e^{-}$}.
\end{cases}
\end{eqnarray*}
\end{prop}
\begin{proof}
In the insertion procedure of a Dyck tableau, we may add a ribbon 
between the box labeled $j$ and the box labeled by $j+1$.
Adding a ribbon indicates that we change the left border for the 
edge with the label $j$ from $U$ to $D$ and the right border for 
the edge with the label $j+1$ from $D$ to $U$. 
From this observation, it is enough to consider the entries $e^{-}$,
$e$ and $e^{+}$.

If $T_{1}(e)=1$, there is no ribbon starting at position $i_{D}$.
Thus we have $\mathrm{rb}(e)=D$.

If $T_{1}(e)\in[2,n]$, the right border for the edge $e$ depends 
on whether there is a ribbon starting from $e$:
\begin{enumerate}
\item if $e\rightarrow e^{-}$, we have a ribbon between the box labeled by $T_{1}(e)$ 
and the box labeled by $T_{1}(e)-1$, which means $\mathrm{rb}(e)=U$,
\item if $e^{-}\rightarrow e$ or $e\uparrow e^{-}$, there is no ribbon, which 
means $\mathrm{rb}(e)=D$.
\end{enumerate}

If $T_{1}(e)=n$, there is no ribbon ending at $i_{U}$.
Thus we have $\mathrm{lb}(e)=U$. 

If $T_{1}(e)\in[1,n-1]$, the left border for the edge $e$ depends 
on whether there is a ribbon ending at $e$:
\begin{enumerate}
\item if $e\rightarrow e^{+}$ or $e^{+}\uparrow e$, we have no ribbon, which 
means $\mathrm{lb}(e)=U$,
\item if $e^{+}\rightarrow e$, we have a ribbon between the box labeled by 
$T_{1}(e)+1$ and the box labeled by $T_{1}(e)$, which means $\mathrm{lb}(e)=D$.
\end{enumerate}
\end{proof}

\subsection{The (LR/RL)-(minima/maxima) of a generalized Dyck tableau}
In this subsection, we introduce (LR/RL)-(minima/maxima) 
of a natural label $T_{1}$ of a tree $\mathrm{Tree}(\lambda)$ and 
reveal its relation to a generalized Dyck tableau.
Given an edge $e$ in $\mathrm{Tree}(\lambda)$, we denote by 
$T_{1}(e)$ the label of the edge $e$ in $T_{1}$.
We introduce the notion of (LR/RL)-(minima/maxima)of $T_{1}$:
\begin{enumerate}
\item 
$T_{1}(e)$ is a {\it right-to-left minimum} (RL-minima) 
if and only if $e$ is connected to the root and 
such that $e\rightarrow e'$ $\Rightarrow$ $T_{1}(e)<T_{1}(e')$,
\item 
$T_{1}(e)$ is a {\it right-to-left maximum} (RL-maxima) 
if and only if  $e$ is connected to a leaf and 
such that $e\rightarrow e'$ $\Rightarrow$ $T_{1}(e)>T_{1}(e')$,
\item 
$T_{1}(e)$ is a {\it left-to-right minimum} (LR-minima) 
if and only if $e$ is connected to the root and 
such that $e'\leftarrow e$ $\Rightarrow$ $T_{1}(e)<T_{1}(e')$,

\item
$T_{1}(e)$ is a {\it left-to-right maximum} (LR-maxima) 
if and only if $e$ is connected to a leaf and 
such that $e'\leftarrow e$ $\Rightarrow$ $T_{1}(e)>T_{1}(e')$,
\end{enumerate}

\begin{prop}[Generalization of Proposition 12 in \cite{ABDH11}]
\label{prop:RLminima}
A RL-minima of $T_{1}$ is bijective to a dotted box $b$ in $\mathrm{DTab}(T_{1})$ 
such that $b$ is at the zeroth floor with a right border equal to $D$ and there are 
neither empty boxes nor dotted boxes below $b$.
\end{prop}
\begin{proof}
If $T_{1}(e)$ is a RL-minima, there is no ribbon below it.
If there exists such a ribbon, the ribbon connects labels $n_{1}$ and $n_{2}$ which 
satisfies $n_1<n_2<T_{1}(e)$ and the edge labeled by 
$n_{1}$ is strictly right to the edge $e$. 
This contradicts the fact that $T_{1}(e)$ is minimal.
Therefore, the RL-minima $T_{1}(e)$ is at the zeroth floor and denote by $b$ 
the box labeled by $T_{1}(e)$.
This indicates that there are neither empty boxes nor labeled boxes below $b$.
Similarly, $e^{-}$ has to be to the left of $e$ since $e$ is connected to the root.
From Proposition \ref{prop:lbrb}, the right border $\mathrm{rb}(e)=D$.

Conversely, let $T_{1}(e)$ be an entry in $\mathrm{DTab}(T_{1})$ 
corresponding to a dotted box at the zeroth floor with $\mathrm{rb}(e)=D$ 
and there are neither empty boxes nor dotted boxes below it.
This implies that $e$ is connected to the root in $T_{1}$ and $e^{-}$ is placed 
at the left of $e$. Since there is no ribbon below $T_{1}(e)$, an entry $j<T_{1}(e)$
is placed at the left of $T_{1}(e)$. Thus, $T_{1}(e)$ is the RL-minima.
\end{proof}

\begin{prop}[Generalization of Proposition 13 in \cite{ABDH11}]
A LR-maxima of $T_{1}$ is bijective to a dotted box in $\mathrm{DTab}(T_{1})$
at the maximal floor with a left border equal to $U$.
\end{prop}
\begin{proof}
The same argument as Proposition \ref{prop:RLminima}
\end{proof}

\begin{prop}[Gneralization of Proposition 14 in \cite{ABDH11}]
The box with the label $n$ in $\mathrm{DTab}(T_{1})$ corresponds to 
the rightmost dotted box at the maximal floor and its left border equal
to $U$.
The box with the label $1$ in $\mathrm{DTab}(T_{1})$ corresponds to 
the leftmost dotted box at the zeroth floor and its right border equal
to $D$.
\end{prop}
\begin{proof}
Since there is no ribbon above the box with the label $n$, it is at the 
maximal floor. Since there is no ribbon ending at the box with the label $n$,
its left border is $U$.
Let $i$ and $j$ be two labels such that $n\ge j>i$ and $j$ is to the left of $i$.
If such $i$ and $j$ do not exist, the box labeled by $n$ is the rightmost dotted box
which is at the maximal floor and $\mathrm{lb}(n)=U$.
When such $i$ and $j$ exist, we have $\mathrm{lb}(i)=D$ if $i+1$ is to the left of $i$ from 
Proposition \ref{prop:lbrb}, or at least one ribbon above $i$ if $i+1$ is to the right of $i$.
Thus, the box with the label $n$ is the right-most box among 
the dotted boxes which are at the maximal floor and with their 
left borders equal to $U$.

The same argument for the box with the label $1$.
\end{proof}

For any natural label $T_{1}$, it is clear that $n$ is a RL-maxima and 
$1$ is a LR-minima.
Other RL-maximas and LR-minimas are characterized as follows.

\begin{prop}[Generalization of Proposition 15 in \cite{ABDH11}]
\label{prop:RLmaxima}
A RL-maxima $j<n$ of $T_{1}$ is bijective to a dotted box in $\mathrm{DTab}(T_{1})$
at the maximal floor, with a left border equal to $D$ and right to the box with $n$. 
\end{prop}
\begin{proof}
If $T_{1}(e)$ is a RL-maxima, then 
\begin{enumerate}
\item 
\label{RLmaxima1}
there is no ribbon above $T_{1}(e)$, which implies it is at the maximal floor;
\item $e^{+}$ is to the left of $e$, which implies $\mathrm{lb}(e)=D$ from Proposition \ref{prop:lbrb};
\item
\label{RLmaxima3}
 a labeled box corresponding to $e$ is right to the box with $n$.
\end{enumerate}

Conversely, if $e$ satisfies above three properties from (\ref{RLmaxima1}) to (\ref{RLmaxima3}), 
it is obvious that $e$ is the RL-maxima.
\end{proof}

\begin{prop}[Generalization of Proposition 16 in \cite{ABDH11}]
A LR-minima $j>1$ of $T_{1}$ is bijective to a dotted box in $\mathrm{DTab}(T_{1})$ 
at the zeroth floor, with a right border equal to $U$ and left to the box with $1$. 
\end{prop}
\begin{proof}
The same argument as Proposition \ref{prop:RLmaxima}
\end{proof}

\section{Tree-like tableaux for general Dyck tilings}
\label{sec:TTab}
\subsection{Tree-like tableaux}

Given a Ferrers diagram $F$, the {\it half-perimeter} of $F$
is defined as the sum of its number of rows and its number of columns.
The {\it boundary edges} are the edges which are on the southeast boundary
of the diagram $F$.
Note that the number of boundary edges is equal to the half-perimeter of $F$.
The {\it boundary boxes} are the boxes which have a boundary edge.

We define a tree-like tableau following \cite{ABN11}:
\begin{defn}[Tree-like tableau]
\label{deftlTab}
A tree-like tableau is a Ferrers diagram in English notation where each box contains either 
$0$ or $1$ dot  with the following three conditions:
\begin{enumerate}
\item
\label{deftltab1}
the top left box of the diagram contains a dot. We call this dot the root;
\item 
\label{deftltab2}
every column and every row contains at least one dotted box.
\item 
\label{deftltab3}
for every non-root dotted box $b$, there exits a dotted box either above $b$ 
in the same column, or to its left in the same row, but not both.
\end{enumerate}
\end{defn}

We generalize Definition \ref{deftlTab} by relaxing the condition (\ref{deftltab3}).
For this purpose, we first introduce two classes of dotted boxes and consider 
a Ferrers diagram consisting of these two types of dotted boxes.
Then, we impose an admissible condition to get a notion of generalized tree-like 
tableaux.

\begin{defn}
We define two classes of dotted box:
\begin{enumerate}
\item
An off-diagonal dot (or point) is a non-root dotted box $b$ such that 
there exists a dotted box $b'$ either above $b$ in the same column, or 
to its left in the same row, but not both.
An off-diagonal dot $b$ is called row (resp. column) dot if there exists 
a dot $b'$ left to (resp. above) $b$  (resp. column).
\item
A diagonal dot (or point) is a non-root dotted box $b$ such that 
there exists dotted box neither above $b$ in the same column nor to its left 
in the same row.
\end{enumerate}
\end{defn}

Let $\widetilde{\mathcal{T}}_{n}$ be a set of Ferrers diagrams such that 
it consists of diagonal and off-diagonal dots which satisfy two conditions 
(\ref{deftltab1}) and (\ref{deftltab2}) in Definition \ref{deftlTab}.
Given $T\in\widetilde{\mathcal{T}}_{n}$, we may have several diagonal dots.
We enumerate all dots by $n,n-1,\ldots,1$ as follows.

\begin{defn}[Reverse insertion procedure]
\label{defn:iipTT}
The {\it reverse insertion procedure} 
$\mathrm{RI}:\widetilde{\mathcal{T}}_{n}\rightarrow\widetilde{\mathcal{T}}_{n-1}$, 
$T\mapsto T'$,  
is defined as an operation of the following two steps:
\begin{enumerate}
\item If there exists boundary dotted boxes whose both east and south edges are 
boundary edges, take the northeast-most box among them and label it by $n$.
Otherwise, take the northeast-most box $b$ whose south edge is a boundary edge.
Then, remove a maximal ribbon from $T$ starting from the east edge of $b$ and ending 
at the south edge of a boundary dotted box.
In both cases, we label $b$ by $n$.
\item If the box $b$ with the label $n$ is a diagonal box, we delete the row 
and the column where $b$ is placed. 
If the box $b$ is a row (resp. column) box, we delete the column (resp. row)
where $b$ is placed.
\end{enumerate}
\end{defn}

We get a new Ferrers diagram with $n-1$ dots in $\widetilde{\mathcal{T}}_{n-1}$.
By successive use of the above procedure, one can label the all dotted box
by integers in $[1,n]$.

See Figure \ref{fig:TtilRI} for an example of a diagram $\widetilde{\mathcal{T}}_{8}$
and the action of $\mathrm{RI}^{4}$ on the diagram.
\begin{figure}[ht]
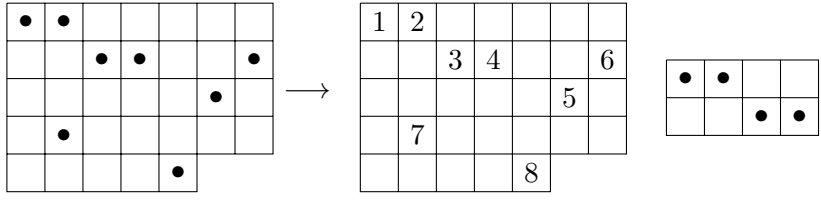

\tikzpic{-0.5}{[x=0.5cm,y=0.5cm]
\draw(0,0)--(7,0)--(7,-4)--(5,-4)--(5,-5)--(0,-5)--(0,0)
     (0,-1)--(7,-1)(0,-2)--(7,-2)(0,-3)--(7,-3)(0,-4)--(5,-4)
     (1,0)--(1,-5)(2,0)--(2,-5)(3,0)--(3,-5)(4,0)--(4,-5)(5,0)--(5,-4)(6,0)--(6,-4);
\draw(0.5,-0.5)node{$\bullet$}(1.5,-0.5)node{$\bullet$}(2.5,-1.5)node{$\bullet$}
     (3.5,-1.5)node{$\bullet$}(5.5,-2.5)node{$\bullet$}(6.5,-1.5)node{$\bullet$}
     (1.5,-3.5)node{$\bullet$}(4.5,-4.5)node{$\bullet$};
}$\longrightarrow$
\tikzpic{-0.5}{[x=0.5cm,y=0.5cm]
\draw(0,0)--(7,0)--(7,-4)--(5,-4)--(5,-5)--(0,-5)--(0,0)
     (0,-1)--(7,-1)(0,-2)--(7,-2)(0,-3)--(7,-3)(0,-4)--(5,-4)
     (1,0)--(1,-5)(2,0)--(2,-5)(3,0)--(3,-5)(4,0)--(4,-5)(5,0)--(5,-4)(6,0)--(6,-4);
\draw(0.5,-0.5)node{$1$}(1.5,-0.5)node{$2$}(2.5,-1.5)node{$3$}
     (3.5,-1.5)node{$4$}(5.5,-2.5)node{$5$}(6.5,-1.5)node{$6$}
     (1.5,-3.5)node{$7$}(4.5,-4.5)node{$8$};
}
\tikzpic{-0.5}{[x=0.5cm,y=0.5cm]
\draw(0,0)--(4,0)--(4,-2)--(0,-2)--(0,0)
     (0,-1)--(4,-1)(1,0)--(1,-2)(2,0)--(2,-2)(3,0)--(3,-2);
\draw(0.5,-0.5)node{$\bullet$}(1.5,-0.5)node{$\bullet$}(2.5,-1.5)node{$\bullet$}
     (3.5,-1.5)node{$\bullet$};
}
\caption{A diagram $D$ in $\widetilde{\mathcal{T}}_{8}$ (the left picture), and 
its labeling (the middle picture).
The right picture is $\mathrm{RI}^{4}(D)$.}
\label{fig:TtilRI}
\end{figure}

We consider a diagram in $\widetilde{\mathcal{T}}_{n}$ with circled 
vertices that are on the boundary edges, denoted by $T^{\circ}$.
We denote by $\widetilde{\mathcal{T}}_{n}^{\circ}$ the set of 
such diagrams with circled vertices.
We perform the first step of the reverse insertion procedure on 
a diagram, which is to remove the maximal ribbon from $T^{\circ}$.
Since we remove a ribbon, the number of the boundary edges in 
$T^{\circ}$ and that of the new diagram are equal.
We put a circle on vertices in a new diagram according to the circles
on vertices in $T^{\circ}$.
More precisely, if we enumerate the vertices of the boundary edges 
in $T^{\circ}$ and those of the new diagram by $1,2,\ldots$, the $i$-th vertex
in the new diagram has a circle if and only if the $i$-th vertex 
in $T^{\circ}$ has a circle.

We perform the second step of the reverse insertion procedure, which 
is to remove a row, a column, or both from the new diagram obtained 
from $T^{\circ}$.
We have three cases for the removal of a box $b$ with the label $n$.
\begin{enumerate}[{(RI}1{)}]
\item The box $b$ is a row dot. \\
We delete a circle of the south-east vertex of $b$ if it exists.
Then, we delete the column containing $b$ and all circles weakly right 
to $b$ are moved to the left by $(-1,0)$.

\item The box $b$ is a column dot. \\
We delete a circle of the south-east vertex of $b$ if it exists.
Then, we delete the row containing $b$ and all circles weakly 
below $b$ are moved to upward by $(0,1)$.

\item
\label{RIcircdiag}
The box $b$ is a diagonal dot.\\
We delete a circle of the south-east vertex of $b$ if it exists.
Then, we delete the row and the column which contain $b$. 
All circles weakly right to $b$ are move to the left by $(-1,0)$
and all circles weakly below $b$ are moved to upward by $(0,1)$.
We add a circle on the vertex which used to be the north-west vertex
of $b$.
\end{enumerate}

We denote by 
$\mathrm{RI}_{\circ}:\widetilde{\mathcal{T}}^{\circ}_{n}
\rightarrow\widetilde{\mathcal{T}}^{\circ}_{n-1}$ 
the reversed insertion procedure with circled vertices defined above.
Note that during the procedure (RI\ref{RIcircdiag}), we may have a chance to 
put more than one circles on the vertex which used to be the north-west vertex 
of $b$.

Let $T^{\circ}\in\widetilde{\mathcal{T}}^{\circ}_{n}$, and $S^{\circ}$ 
be the Ferrers diagram with dots obtained from $T^{\circ}$ by the first step of $\mathrm{RI}_{\circ}$.
We denote by $b$ the box with label $n$ in $S^{\circ}$.
We consider a subset $\mathcal{T}_{n}^{\circ}\subset\widetilde{\mathcal{T}}^{\circ}_{n}$. 
A diagram $T^{\circ}$ is in $\mathcal{T}_{n}^{\circ}$ if it satisfies the following 
conditions.
\begin{enumerate}
\item A diagram $S^{\circ}$ has a circle on the south-east vertex of the box $b$.
\item If the half-perimeter of $S^{\circ}$ is $m$, we have $2n+1-m$ circles on the boundary 
vertices of $S^{\circ}$.
\item The diagram $\mathrm{RI}_{\circ}^{n-1}(T^{\circ})$ is a single box with 
the label $1$ and with a circle on the south-east vertex.
\end{enumerate}

\begin{defn}
We say that $\mathrm{RI}_{\circ}$ on 
$T^{\circ}\in\widetilde{\mathcal{T}}_{n}\subset\widetilde{\mathcal{T}}^{\circ}_{n}$ 
is {\it admissible} if one add at most one circle on a vertex during the 
procedure (RI\ref{RIcircdiag}).
\end{defn}

\begin{defn}[admissibility]
A diagram $T\in\widetilde{\mathcal{T}}_{n}$ is admissible if 
there exists $T^{\circ}\in\mathcal{T}^{\circ}_{n}$ such that 
\begin{enumerate}
\item $\mathrm{RI}_{\circ}^{p}$ on $T^{\circ}$ for all 
$1\le p\le n-1$ are admissible, and 
\item the shape of $T^{\circ}$ and the position of dots in $T^{\circ}$ 
are the same as those of $T$.
\end{enumerate}
\end{defn}

\begin{defn}[Generalized tree-like tableau]
If a Ferrers diagram with dots, denoted by $T$, is called a generalized tree-like tableau
if $T\in\widetilde{\mathcal{T}}_{n}$ and $T$ is admissible.
We denote by $\mathcal{T}_{n}$ the set of generalized tree-like tableaux 
of size $n$.
\end{defn}

Figure \ref{fig:GTT} gives an example of a generalized tree-like tableau.

\begin{figure}[ht]
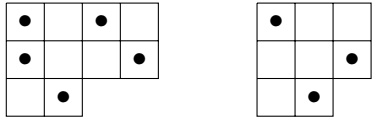

\tikzpic{-0.5}{[x=0.5cm,y=0.5cm]
\draw(0,0)--(4,0)--(4,-2)--(2,-2)--(2,-3)--(0,-3)--(0,0)
     (0,-1)--(4,-1)(0,-2)--(2,-2)(1,0)--(1,-3)(2,0)--(2,-2)(3,0)--(3,-2);
\node at(0.5,-0.5){$\bullet$};\node at(0.5,-1.5){$\bullet$};\node at(1.5,-2.5){$\bullet$};
\node at(2.5,-0.5){$\bullet$};\node at(3.5,-1.5){$\bullet$};
} \qquad
\tikzpic{-0.5}{[x=0.5cm,y=0.5cm]
\draw(0,0)--(3,0)--(3,-2)--(2,-2)--(2,-3)--(0,-3)--(0,0)
     (0,-1)--(3,-1)(0,-2)--(2,-2)(1,0)--(1,-3)(2,0)--(2,-2);
\node at(0.5,-0.5){$\bullet$};\node at(1.5,-2.5){$\bullet$};\node at(2.5,-1.5){$\bullet$};
}
\caption{A generalized tree-like tableau of size $5$ (the left picture) and 
a non-admissible configuration (the right picture).}
\label{fig:GTT}
\end{figure}

Let $T$ be a generalized tree-like tableau.
If the diagram $T$ has $n$ dots, $n$ is called the {\it size} of $T$.
Since $\mathcal{T}_{n}$ contains the set of (non-generalized) tree-like 
tableaux as a subset, we call a generalized tree-like tableau simply 
a tree-like tableau when it is clear from the context.

\begin{remark}
When $T\in\mathcal{T}_{n}$ is a (non-generalized) tree-like tableau defined 
by Definition \ref{deftlTab}, the half-perimeter of $T$ is equal to $n+1$.
On the other hand, if $T$ is a generalized tree-like tableau, the half-perimeter 
of $T$ is in $[n+1,2n]$.
The unique tree-like tableau of half-perimeter $2n$ is the one consisting of 
the root point and $n-1$ diagonal points.
\end{remark}

\subsection{Insertion procedure}

\paragraph{\bf Row, column and diagonal insertion}
Let $F$ be a Ferrers diagram with $F_{i}$ boxes in the $i$-th row, 
and $e$ be an boundary edge of $F$.
We denote by $F'$ the Ferrers diagram obtained from $F$ by an insertion.
Suppose $e$ is the right edge of a boundary box and it is in the $p$-th row. 
The insertion of a column (or simply column insertion) at $e$ is defined 
by $F'_{i}:=F_{i}+1$ for $1\le i\le p$ and $F'_{i}:=F_{i}$ the same 
for $i>p$.
Similarly, suppose $e$ is the bottom edge of a boundary box $b$ and 
$b$ is in the $p$-th row and in the $q$-th column.
The insertion of a row (or simply row insertion) at $e$
is defined by changing $F'_{p+1}:=q$, $F'_{i+1}:=F_{i}$ for $i\ge p+1$, 
and $F'_{j}:=F_{j}$ for $j<p$.
Let $v$ be a vertex on the boundary edges of $F$ which is neither the leftmost bottom
one nor the rightmost top one.
Suppose the coordinate of $v$ is $(p,q)$.
The diagonal insertion at $v$ is defined by changing
$F'_{p+1}:=q+1$, $F'_{i}:=F_{i}+1$ for $1\le i\le p$, and 
$F'_{j+1}:=F_{j}$ for $j>p$.
See Figure \ref{fig:DI} for an example of a diagonal insertion on a diagram.
\begin{figure}[ht]
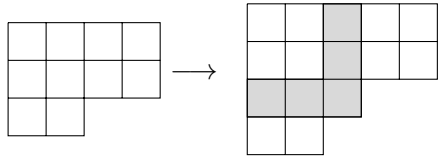

\tikzpic{-0.5}{[x=0.5cm,y=0.5cm]
\draw(0,0)--(4,0)--(4,-2)--(2,-2)--(2,-3)--(0,-3)--(0,0)
     (0,-1)--(4,-1)(0,-2)--(2,-2)(1,0)--(1,-3)(2,0)--(2,-2)(3,0)--(3,-2);
}$\longrightarrow$
\tikzpic{-0.5}{[x=0.5cm,y=0.5cm]
\draw[fill=gray!30!white](2,0)--(3,0)--(3,-3)--(0,-3)--(0,-2)--(2,-2)--(2,0);
\draw(0,0)--(5,0)--(5,-2)--(3,-2)--(3,-3)--(2,-3)--(2,-4)--(0,-4)--(0,0)
     (0,-1)--(5,-1)(0,-2)--(3,-2)(0,-3)--(2,-3)(1,0)--(1,-4)(2,0)--(2,-3)(3,0)--(3,-2)
     (4,0)--(4,-2);
}
\caption{An example of a diagonal insertion. The right picture is the diagonal insertion 
at (2,2) for the Ferrers diagram $(4,4,2)$. The shaded boxes are the added boxes by the insertion.}
\label{fig:DI}
\end{figure}

Let $T(n)\in\mathcal{T}_{n}$ of half-perimeter $m$. 
We enumerate the boundary edges of $T$ by $1,2,\ldots,m$ from the south-west edge
to the north-east edge.
We will define a sequence of integers $\mathbf{e}(n):=(e_{1},\ldots,e_{m})$ such that 
$e_{1}:=0$, $e_{i}\in[0,2n]$ for $1\le i\le m$ and $e_{i+1}=e_{i}+1$ or $e_{i}+2$. 
The integer $e_{i}$ is associated with the $i$-th boundary edge of $T$.
We say that $\mathbf{e}(n)$ is the label of the boundary edges in $T$, and the $i$-th 
edge has the label $e_{i}$.
The vertex that connects the $i$-th and $i+1$-th edges is 
said to be {\it valid} (resp. {\it invalid}) if $e_{i+1}=e_{i}+2$ 
(resp. $e_{i+1}=e_{i}+1$).
We define the {\it label} of a valid vertex as $v_{i}=e_{i}+1$.
We put a circle $\circ$ on a valid vertex and put a cross $\times$ on an invalid vertex.

We define $\mathbf{e}(n)$ recursively starting from $\mathbf{e}(1)$ by an insertion procedure.
We have a unique tableau when $n=1$, that is a single box with the label $1$.
We define $\mathbf{e}(1):=(0,2)$ and depict it as
\begin{eqnarray}
\label{eqn:TT1}
\tikzpic{-0.5}{[x=0.6cm,y=0.6cm]
\draw(0,0)--(1,0)node[anchor=center,circle,inner sep=1mm,draw]{}--(1,1);
\draw(0,0)--(0,1)--(1,1);
\node at(0.5,0.5){$1$};
}
\end{eqnarray}

Note that once boundary edges and valid vertices are given, the conditions $e_{1}=0$
and $e_{i+1}=e_{i}+1$ or $e_{i}+2$ determine the label of a boundary edge.
Therefore, we do not write a label on a boundary edge explicitly.

Following \cite{ABN11}, we introduce the notion of {\it special point}:
\begin{defn}
Let $T$ be a tableau with dots in $\widetilde{\mathcal{T}}_{n}$.
The {\it special point} of $T$ is the northeast-most dot 
that is placed at the bottom of a column.
\end{defn}

Let $\mathcal{T}^{\times}_{1}$ be the set of the tree-like tableau of size one, 
{\it i.e.,} the set $\mathcal{T}^{\times}_{1}$ consists of the unique tree-like 
tableau depicted in Eqn. (\ref{eqn:TT1}).
To define $\mathcal{T}^{\times}_{n}\subset\widetilde{\mathcal{T}}_{n}$ with $n\ge2$, we introduce the insertion 
procedure as follows.

\begin{defn}[Insertion procedure]
\label{defn:IPforTT}
The insertion procedure $\mathrm{IP}:\widetilde{\mathcal{T}}_{n}\rightarrow\widetilde{\mathcal{T}}_{n+1}$,
$T\mapsto T'$
is defined as an operation of the following two steps:
\begin{enumerate}
\item Take a boundary edge $e$ or a valid vertex $v$. 
\begin{enumerate}
\item If we have $e$ which is horizontal (resp. vertical), we perform a row (resp. column) insertion
at $e$. 
We locally change a label of the boundary edges of $T$ as in Figure \ref{IPbrc}.
\item If we have $v$, we perform a diagonal insertion at $v$. 
We locally change a label of the boundary edges of $T$ as in Figure \ref{IPbd}.
\end{enumerate}
We put a dot at the box $b$ which is added in the insertion process and southeast-most box.
We denote by $T''$ the diagram with circled vertices on the boundary edges obtained by this insertion.
\item If there exists the special point $s$ right to the box $b$, we add a ribbon 
starting from the east edges of $b$ to the south edge of $s$.
The label of boundary edges of $T'$ is the same as $T''$.
\end{enumerate}
When a boundary edge $e$ or a valid vertex $v$ has a label $p$ with $0\le p\le 2n$,
the insertion is said to be the insertion procedure at $p$.
\end{defn}

\begin{figure}[ht]
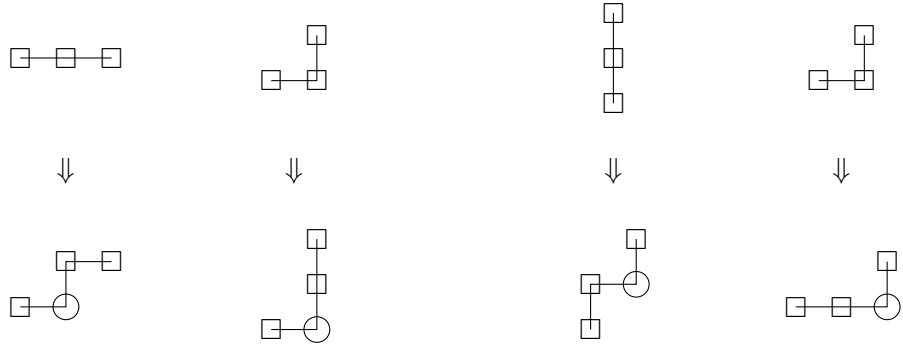

\tikzpic{-0.5}{[x=0.6cm,y=0.6cm]
	\node[anchor=center]at(0,0){
		\tikzpic{-0.5}{
		\draw(0,0)node[anchor=center]{$\square$}
		     --(1,0)node[anchor=center]{$\square$}
		     --(2,0)node[anchor=center]{$\square$};
		}
	};
	\node[anchor=center]at(5,0){
		\tikzpic{-0.5}{
		\draw(0,0)node[anchor=center]{$\square$}
		     --(1,0)node[anchor=center]{$\square$}
		     --(1,1)node[anchor=center]{$\square$};
		}
	};
	\node[anchor=center]at(12,0){
		\tikzpic{-0.5}{
		\draw(0,0)node[anchor=center]{$\square$}
		     --(0,1)node[anchor=center]{$\square$}
		     --(0,2)node[anchor=center]{$\square$};
		}
	};
	\node[anchor=center]at(17,0){
		\tikzpic{-0.5}{
		\draw(0,0)node[anchor=center]{$\square$}
		     --(1,0)node[anchor=center]{$\square$}
		     --(1,1)node[anchor=center]{$\square$};
		}
	};
	\node[anchor=center]at(0,-5){
		\tikzpic{-0.5}{
		\draw(0,0)node[anchor=center]{$\square$}
		     --(1,0)node[anchor=center,circle,inner sep=1.2mm,draw]{}
		     --(1,1)node[anchor=center]{$\square$}
		     --(2,1)node[anchor=center]{$\square$};
		}
	};
	\node[anchor=center]at(5,-5){
		\tikzpic{-0.5}{
		\draw(0,0)node[anchor=center]{$\square$}
		     --(1,0)node[anchor=center,circle,inner sep=1.2mm,draw]{}
		     --(1,1)node[anchor=center]{$\square$}
		     --(1,2)node[anchor=center]{$\square$};
		}
	};
	\node[anchor=center]at(12,-5){
		\tikzpic{-0.5}{
		\draw(0,0)node[anchor=center]{$\square$}
		     --(0,1)node[anchor=center]{$\square$}
		     --(1,1)node[anchor=center,circle,inner sep=1.2mm,draw]{}
		     --(1,2)node[anchor=center]{$\square$};
		}
	};
	\node[anchor=center]at(17,-5){
		\tikzpic{-0.5}{
		\draw(0,0)node[anchor=center]{$\square$}
		     --(1,0)node[anchor=center]{$\square$}
		     --(2,0)node[anchor=center,circle,inner sep=1.2mm,draw]{}
		     --(2,1)node[anchor=center]{$\square$};
		}
	};
\node[anchor=center]at(0,-2.5){$\Downarrow$};
\node[anchor=center]at(5,-2.5){$\Downarrow$};
\node[anchor=center]at(12,-2.5){$\Downarrow$};
\node[anchor=center]at(17,-2.5){$\Downarrow$};
}
\caption{The change of a label of boundary edges by the row insertion (left pictures) 
and column insertion (right pictures). $\square$ is either $\circ$ or $\times$.}
\label{IPbrc}
\end{figure}

\begin{figure}[ht]
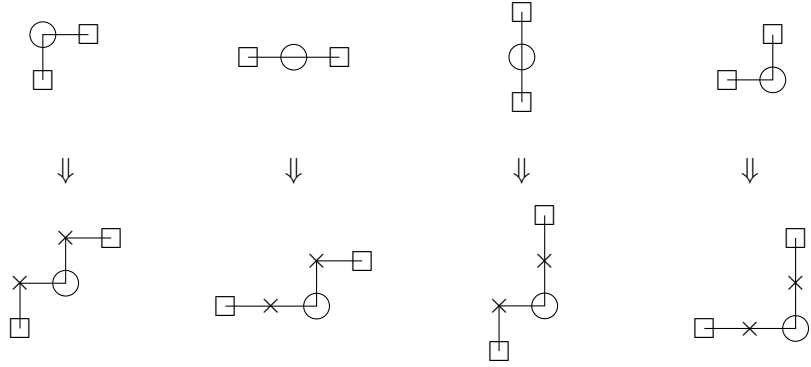

\tikzpic{-0.5}{[x=0.6cm,y=0.6cm]
	\node[anchor=center] at(0,0){
		\tikzpic{-0.5}{
		\draw(0,0)node[anchor=center]{$\square$}
		--(0,1)node[anchor=center,circle,inner sep=1.2mm,draw]{}
		--(1,1)node[anchor=center]{$\square$};
			   }
	};
	\node[anchor=center]at(5,0){
	\tikzpic{-0.5}{
		\draw(0,0)node[anchor=center]{$\square$}
		--(1,0)node[anchor=center,circle,inner sep=1.2mm,draw]{}
		--(2,0)node[anchor=center]{$\square$};
		   }
        };

	\node[anchor=center]at(10,0){
		\tikzpic{-0.5}{
			\draw(0,0)node[anchor=center]{$\square$}
			--(0,1)node[anchor=center,circle,inner sep=1.2mm,draw]{}
			--(0,2)node[anchor=center]{$\square$};
			}
	};
	\node[anchor=center]at(15,0){
		\tikzpic{-0.5}{
			\draw(0,0)node[anchor=center]{$\square$}
			--(1,0)node[anchor=center,circle,inner sep=1.2mm,draw]{}
			--(1,1)node[anchor=center]{$\square$};
			}
	};
	\node[anchor=center] at(0,-5){
		\tikzpic{-0.5}{
		\draw(0,0)node[anchor=center]{$\square$}--(0,1)node[anchor=center]{$\times$}
			--(1,1)node[anchor=center,circle,inner sep=1.2mm,draw]{}
			--(1,2)node[anchor=center]{$\times$}
			--(2,2)node[anchor=center]{$\square$};
			   }
	}; 
	\node[anchor=center]at(5,-5){
		\tikzpic{-0.5}{
			\draw(0,0)node[anchor=center]{$\square$}
			      --(1,0)node[anchor=center]{$\times$}
			      --(2,0)node[anchor=center,circle,inner sep=1.2mm,draw]{}
			      --(2,1)node[anchor=center]{$\times$}
			      --(3,1)node[anchor=center]{$\square$};
		}
	};
	\node[anchor=center]at(10,-5){
		\tikzpic{-0.5}{
			\draw(0,0)node[anchor=center]{$\square$}
			      --(0,1)node[anchor=center]{$\times$}
			      --(1,1)node[anchor=center,circle,inner sep=1.2mm,draw]{}
			      --(1,2)node[anchor=center]{$\times$}
			      --(1,3)node[anchor=center]{$\square$};
		}
	};
	\node[anchor=center]at(15,-5){
		\tikzpic{-0.5}{
			\draw(0,0)node[anchor=center]{$\square$}--(1,0)node[anchor=center]{$\times$}
			--(2,0)node[anchor=center,circle,inner sep=1.2mm,draw]{}
			--(2,1)node[anchor=center]{$\times$}
			--(2,2)node[anchor=center]{$\square$};
		}
	};
\node[anchor=center]at(0,-2.5){$\Downarrow$};
\node[anchor=center]at(5,-2.5){$\Downarrow$};
\node[anchor=center]at(10,-2.5){$\Downarrow$};
\node[anchor=center]at(15,-2.5){$\Downarrow$};
}
\caption{The change of a label of boundary edges by the 
diagonal insertion. $\square$ is either $\circ$ or $\times$.}
\label{IPbd}
\end{figure}

\begin{remark}
We may add a ribbon in the second step of an insertion procedure.
In this case, note that we have the same half-perimeter before 
and after adding the ribbon.
\end{remark}

\begin{defn}
We define $\mathcal{T}^{\times}_{n}\subset\widetilde{\mathcal{T}}_{n}$ with $n\ge2$
as a set of diagrams of size $n$ that is obtained from a single box 
with the label $1$ and the boundary label $\mathbf{e}=(0,2)$ by successive 
insertion procedures defined in Definition \ref{defn:IPforTT}.
\end{defn}

\begin{prop}
\label{propem}
We have $e_{m}=2n$ for the label $\mathbf{e}(n)$ of the boundary edges.
\end{prop}
\begin{proof}
By an insertion procedure, the number of labels increases 
by two from Figure \ref{IPbrc} and Figure \ref{IPbd}.
When $n=1$, we have $\mathbf{e}=(0,2)$, {\it i.e.,} $e_{2}=2$.
In general, we have $e_{m}=2+2(n-1)=2n$.
\end{proof}

\begin{prop}
The number of valid vertices is equal to the number of off-diagonal points.
\end{prop}
\begin{proof}
Let $T$ be a diagram in $\mathcal{T}^{\times}_{n}$ and $m$ be the number 
of diagonal points in $T$.
We increase the half-perimeter by one by a row or column insertion, and 
by two by a diagonal insertion.
Thus, the perimeter is given by $(n-m)+2m+1=n+m+1$, which is the number of 
boundary edges.
The number of labels on boundary edges and valid vertices is $2n+1$.
The number of valid vertices is $2n+1-(n+m+1)=n-m$, which is the number 
of off-diagonal points. 
\end{proof}

Let $\mathbf{h}(n):=(h_{1},\ldots,h_{n})$ and 
$\mathbf{h}(n+1):=(h_{1},\ldots,h_{n+1})$ be insertion histories of length
$n$ and $n+1$.
We have $h_{i}\in[0,2(i-1)]$ for $1\le i\le n+1$.
By definition, we have $\mathbf{h}(1)=(0)$.
We define a tree-like tableau $\mathrm{TTab}(\mathbf{h}(1))$ as 
a single box with the label $1$.
We denote by $\mathrm{TTab}(\mathbf{h}(n))$ a tree-like tableau
of size $n$ associated with $\mathbf{h}(n)$.
We recursively define the tableau $\mathrm{TTab}(\mathbf{h}(n+1))$ 
as the tableau obtained from $\mathrm{TTab}(\mathbf{h}(n))$ by
the insertion procedure at $h_{n+1}$.
See Figure \ref{fig:TTIP} is an example of the insertion procedure 
for an insertion history.
\begin{figure}[ht]
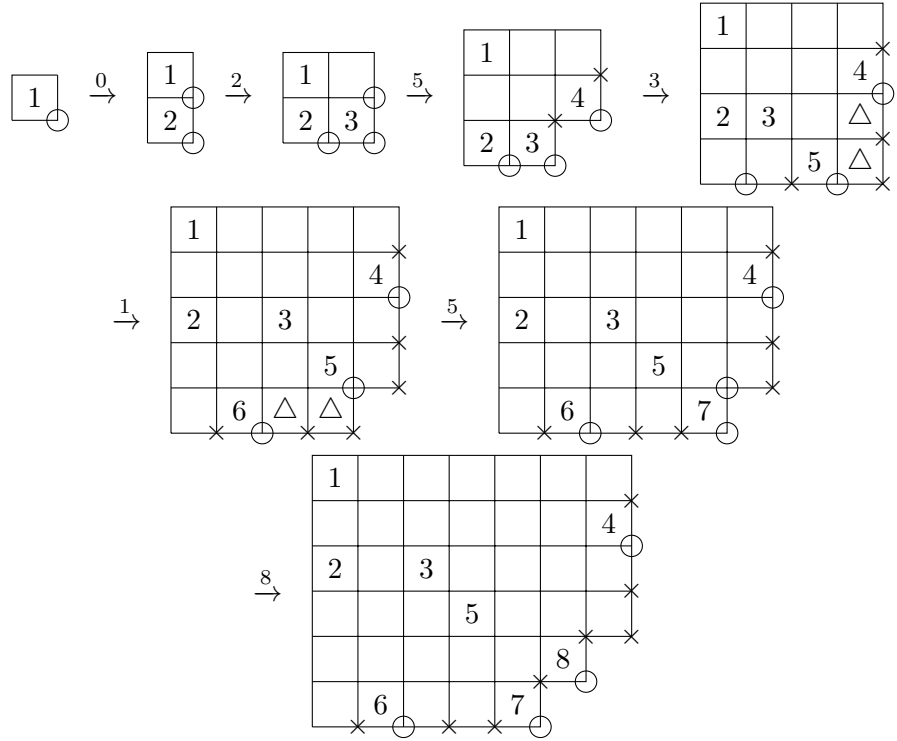

\tikzpic{-0.5}{[x=0.6cm,y=0.6cm]
\draw(0,0)--(1,0)node[anchor=center,circle,inner sep=1mm,draw]{}--(1,1);
\draw(0,0)--(0,1)--(1,1);
\node at(0.5,0.5){$1$};
}
$\xrightarrow{0}$
\tikzpic{-0.5}{[x=0.6cm,y=0.6cm]
\draw(0,0)--(1,0)node[anchor=center,circle,inner sep=1mm,draw]{}
     --(1,1)node[anchor=center,circle,inner sep=1mm,draw]{}--(1,2)--(0,2)--(0,0);
\draw(0,1)--(1,1);
\draw(0.5,0.5)node[anchor=center]{$2$}(0.5,1.5)node[anchor=center]{$1$};
}
$\xrightarrow{2}$
\tikzpic{-0.5}{[x=0.6cm,y=0.6cm]
\draw(0,0)--(1,0)node[anchor=center,circle,inner sep=1mm,draw]{}
  --(2,0)node[anchor=center,circle,inner sep=1mm,draw]{}
  --(2,1)node[anchor=center,circle,inner sep=1mm,draw]{}--(2,2)--(0,2)--(0,0);
\draw(1,0)--(1,2)(0,1)--(2,1);
\draw(0.5,1.5)node[anchor=center]{$1$}(0.5,0.5)node[anchor=center]{$2$}
    (1.5,0.5)node[anchor=center]{$3$};
}
$\xrightarrow{5}$
\tikzpic{-0.5}{[x=0.6cm,y=0.6cm]
\draw(0,0)--(1,0)node[anchor=center,circle,inner sep=1mm,draw]{}
   --(2,0)node[anchor=center,circle,inner sep=1mm,draw]{}
   --(2,1)node[anchor=center]{$\times$}
   --(3,1)node[anchor=center,circle,inner sep=1mm,draw]{}
   --(3,2)node[anchor=center]{$\times$}--(3,3)--(0,3)--(0,0);
\draw(0,1)--(2,1)(0,2)--(3,2)(1,0)--(1,3)(2,1)--(2,3);
\draw(0.5,2.5)node[anchor=center]{$1$}(0.5,0.5)node[anchor=center]{$2$}
   (1.5,0.5)node[anchor=center]{$3$}(2.5,1.5)node[anchor=center]{$4$};
}
$\xrightarrow{3}$
\tikzpic{-0.5}{[x=0.6cm,y=0.6cm]
\draw(0,0)--(1,0)node[anchor=center,circle,inner sep=1mm,draw]{}
    --(2,0)node[anchor=center]{$\times$}--(3,0)node[anchor=center,circle,inner sep=1mm,draw]{}
    --(4,0)node[anchor=center]{$\times$}--(4,1)node[anchor=center]{$\times$}
    --(4,2)node[anchor=center,circle,inner sep=1mm,draw]{}--(4,3)node[anchor=center]{$\times$}
    --(4,4)--(0,4)--(0,0);
\draw(0,1)--(4,1)(0,2)--(4,2)(0,3)--(4,3)(1,0)--(1,4)(2,0)--(2,4)(3,0)--(3,4);
\draw(0.5,3.5)node[anchor=center]{$1$}(0.5,1.5)node[anchor=center]{$2$}
     (1.5,1.5)node[anchor=center]{$3$}(3.5,2.5)node[anchor=center]{$4$}
     (2.5,0.5)node[anchor=center]{$5$};
\draw(3.5,0.5)node[anchor=center]{$\triangle$}(3.5,1.5)node[anchor=center]{$\triangle$};
} \\
$\xrightarrow{1}$
\tikzpic{-0.5}{[x=0.6cm,y=0.6cm]
\draw(0,0)--(1,0)node[anchor=center]{$\times$}--(2,0)node[anchor=center,circle,inner sep=1mm,draw]{}
   --(3,0)node[anchor=center]{$\times$}--(4,0)node[anchor=center]{$\times$}
   --(4,1)node[anchor=center,circle,inner sep=1mm,draw]{}--(5,1)node[anchor=center]{$\times$}
   --(5,2)node[anchor=center]{$\times$}--(5,3)node[anchor=center,circle,inner sep=1mm,draw]{}
   --(5,4)node[anchor=center]{$\times$}--(5,5)--(0,5)--(0,0);
\draw(0,1)--(4,1)(0,2)--(5,2)(0,3)--(5,3)(0,4)--(5,4)(1,0)--(1,5)(2,0)--(2,5)(3,0)--(3,5)
     (4,1)--(4,5);
\draw(0.5,4.5)node[anchor=center]{$1$}(0.5,2.5)node[anchor=center]{$2$}
     (2.5,2.5)node[anchor=center]{$3$}(4.5,3.5)node[anchor=center]{$4$}
     (3.5,1.5)node[anchor=center]{$5$}(1.5,0.5)node[anchor=center]{$6$};
\draw(2.5,0.5)node[anchor=center]{$\triangle$}(3.5,0.5)node[anchor=center]{$\triangle$};
}
$\xrightarrow{5}$
\tikzpic{-0.5}{[x=0.6cm,y=0.6cm]
\draw(0,0)--(1,0)node[anchor=center]{$\times$}--(2,0)node[anchor=center,circle,inner sep=1mm,draw]{}
   --(3,0)node[anchor=center]{$\times$}--(4,0)node[anchor=center]{$\times$}
   --(5,0)node[anchor=center,circle,inner sep=1mm,draw]{}--(5,1)node[anchor=center,circle,inner sep=1mm,draw]{}
   --(6,1)node[anchor=center]{$\times$}--(6,2)node[anchor=center]{$\times$}
   --(6,3)node[anchor=center,circle,inner sep=1mm,draw]{}--(6,4)node[anchor=center]{$\times$}
   --(6,5)--(0,5)--(0,0);
\draw(0,1)--(5,1)(0,2)--(6,2)(0,3)--(6,3)(0,4)--(6,4)(1,0)--(1,5)(2,0)--(2,5)(3,0)--(3,5)(4,0)--(4,5)
     (5,1)--(5,5);
\draw(0.5,4.5)node[anchor=center]{$1$}(0.5,2.5)node[anchor=center]{$2$}
     (2.5,2.5)node[anchor=center]{$3$}(5.5,3.5)node[anchor=center]{$4$}
     (3.5,1.5)node[anchor=center]{$5$}(1.5,0.5)node[anchor=center]{$6$}
     (4.5,0.5)node[anchor=center]{$7$};
} \\
$\xrightarrow{8}$
\tikzpic{-0.5}{[x=0.6cm,y=0.6cm]
\draw(0,0)--(1,0)node[anchor=center]{$\times$}--(2,0)node[anchor=center,circle,inner sep=1mm,draw]{}
     --(3,0)node[anchor=center]{$\times$}--(4,0)node[anchor=center]{$\times$}
     --(5,0)node[anchor=center,circle,inner sep=1mm,draw]{}
     --(5,1)node[anchor=center]{$\times$}--(6,1)node[anchor=center,circle,inner sep=1mm,draw]{}
     --(6,2)node[anchor=center]{$\times$}--(7,2)node[anchor=center]{$\times$}
     --(7,3)node[anchor=center]{$\times$}--(7,4)node[anchor=center,circle,inner sep=1mm,draw]{}
     --(7,5)node[anchor=center]{$\times$}--(7,6)--(0,6)--(0,0);
\draw(1,0)--(1,6)(2,0)--(2,6)(3,0)--(3,6)(4,0)--(4,6)(5,1)--(5,6)(6,2)--(6,6)
     (0,1)--(5,1)(0,2)--(6,2)(0,3)--(7,3)(0,4)--(7,4)(0,5)--(7,5);
\draw(0.5,5.5)node[anchor=center]{$1$}(0.5,3.5)node[anchor=center]{$2$}
     (2.5,3.5)node[anchor=center]{$3$}(6.5,4.5)node[anchor=center]{$4$}
     (3.5,2.5)node[anchor=center]{$5$}(1.5,0.5)node[anchor=center]{$6$}
     (4.5,0.5)node[anchor=center]{$7$}(5.5,1.5)node[anchor=center]{$8$};
}
\caption{Insertion procedure for $\mathbf{h}=(0,0,2,5,3,1,5,8)$. 
We have $\mathbf{e}=(0,1,3,4,5,7,8,10,11,12,13,15,16)$. The boxes 
with $\triangle$ are the ribbon added in the insertion process.}
\label{fig:TTIP}
\end{figure}

We will give an alternative description of the boundary label $\mathbf{e}(n)$.
With $\mathbf{e}(1)=(0,2)$, we associate a unique Dyck path $\mu_{1}$ of length $2$, {\it i.e.,}
$\mu_{1}=UD$, with the insertion history $\mathbf{h}(1)=(0)$.
Given a duple $(\mu_{n},\mathbf{h}(n))$ and $h_{n+1}$, we define a Dyck path $\mu_{n+1}$ 
of length $2(n+1)$ by the following operation.
We insert the Dyck path $UD$ into $\mu_{n}$ at the $h_{n+1}$-th vertices from left, and 
denote the new Dyck path by $\mu_{n+1}$.

Since the leftmost step of a Dyck path is always an up step, we put zero on this step.
Then, given a Dyck path $\mu_{n}$, we put a label on each step according to 
the local configuration depicted in Figure \ref{fig:DPlabel}.
\begin{figure}[ht]
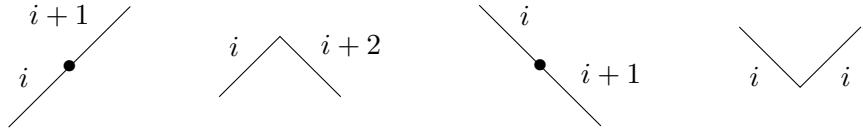

\tikzpic{-0.5}{[x=0.8cm,y=0.8cm]
\draw(0,0)--(2,2);
\node at(1,1){$\bullet$};
\node[anchor=south east] at(0.5,0.5){$i$};
\node[anchor=south east] at(1.5,1.5){$i+1$};
}\qquad
\tikzpic{-0.5}{[x=0.8cm,y=0.8cm]
\draw(0,0)--(1,1)--(2,0);
\node[anchor=south east]at(0.5,0.5){$i$};
\node[anchor=south west]at(1.5,0.5){$i+2$};
}\qquad
\tikzpic{-0.5}{[x=0.8cm,y=0.8cm]
\draw(0,2)--(2,0);
\node at(1,1){$\bullet$};
\node[anchor=south west]at(0.5,1.5){$i$};
\node[anchor=south west]at(1.5,0.5){$i+1$};
}\qquad
\tikzpic{-0.5}{[x=0.8cm,y=0.8cm]
\draw(0,1)--(1,0)--(2,1);
\node[anchor=north east]at(0.5,0.5){$i$};
\node[anchor=north west]at(1.5,0.5){$i$};
}
\caption{Local configurations of labels on a Dyck path.}
\label{fig:DPlabel}
\end{figure}
Note that we have a unique label for a Dyck path since we have 
zero at the first step.
Let $\tilde{\mathbf{e}}:=(\widetilde{e}_{1},\ldots,\widetilde{e}_{2n})$
be a sequence of labels on steps from left to right in a Dyck path $\mu_{n}$.
We may have duplicated integers in $\tilde{\mathbf{e}}$ if we have 
a $DU$-shape in a Dyck path.
We define $\mathbf{e}':=(e'_{1},\ldots,e'_{m})$ by deleting 
one of duplicated integers from $\tilde{\mathbf{e}}$.

For example, $\mathbf{h}=(0,0,2,5,3,1,5,8)$ implies the following 
Dyck path and its label:
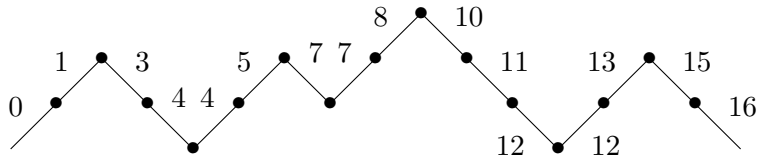
\begin{eqnarray*}
\tikzpic{-0.5}{[x=0.6cm,y=0.6cm]
\draw(0,0)--(1,1)node[anchor=center]{$\bullet$}--(2,2)node[anchor=center]{$\bullet$}
     --(3,1)node[anchor=center]{$\bullet$}--(4,0)node[anchor=center]{$\bullet$}
     --(5,1)node[anchor=center]{$\bullet$}--(6,2)node[anchor=center]{$\bullet$}
     --(7,1)node[anchor=center]{$\bullet$}--(8,2)node[anchor=center]{$\bullet$}
     --(9,3)node[anchor=center]{$\bullet$}--(10,2)node[anchor=center]{$\bullet$}
     --(11,1)node[anchor=center]{$\bullet$}--(12,0)node[anchor=center]{$\bullet$}
     --(13,1)node[anchor=center]{$\bullet$}--(14,2)node[anchor=center]{$\bullet$}
     --(15,1)node[anchor=center]{$\bullet$}--(16,0);
\draw(0.5,0.5)node[anchor=south east]{$0$}(1.5,1.5)node[anchor=south east]{$1$}
(2.5,1.5)node[anchor=south west]{$3$}(3.3,0.7)node[anchor=south west]{$4$}
(4.7,0.7)node[anchor=south east]{$4$}(5.5,1.5)node[anchor=south east]{$5$}
(6.3,1.7)node[anchor=south west]{$7$}(7.7,1.7)node[anchor=south east]{$7$}
(8.5,2.5)node[anchor=south east]{$8$}(9.5,2.5)node[anchor=south west]{$10$}
(10.5,1.5)node[anchor=south west]{$11$}(11.5,0.5)node[anchor=north east]{$12$}
(12.5,0.5)node[anchor=north west]{$12$}(13.5,1.5)node[anchor=south east]{$13$}
(14.5,1.5)node[anchor=south west]{$15$}(15.5,0.5)node[anchor=south west]{$16$};
}
\end{eqnarray*}
We obtain $\mathbf{e}'=(0,1,3,4,5,7,8,10,11,12,13,15,16)$.

\begin{prop}
We have $\mathbf{e}=\mathbf{e}'$.
\end{prop}
\begin{proof}
We prove Proposition by induction with respect to the size $n$ of a Dyck tableau.
When $n=1$, we have $\mathbf{e}=(0,2)$ by definition. In this case, a Dyck path
associated with $\mathbf{h}=(0)$ is a path $UD$ and we have $\mathbf{e}'=(0,2)$.
Thus, we have $\mathbf{e}=\mathbf{e}'$.

We assume that Proposition is true for $n-1$.
This means that we have a valid vertex when we have a $UD$ path in $\mu_{n-1}$, 
and vice versa. 
The action of the diagonal insertion at a valid vertex changes the local 
$\mathbf{e}=(i,i+2,\ldots)$ to $(i,i+1,i+3,i+4,\ldots)$ (see Figure \ref{IPbd}). 
In case of $\mathbf{e}'$, the local change is given by
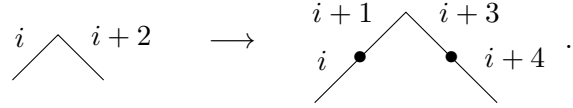
\begin{eqnarray*}
\tikzpic{-0.5}{[x=0.6cm,y=0.6cm]
\draw(0,0)--(1,1)--(2,0);
\node[anchor=south east]at(0.5,0.5){$i$};
\node[anchor=south west]at(1.5,0.5){$i+2$};
}\quad\longrightarrow\quad
\tikzpic{-0.5}{[x=0.6cm,y=0.6cm]
\draw(0,0)--(1,1)node[anchor=center]{$\bullet$}--(2,2)--(3,1)node[anchor=center]{$\bullet$}--(4,0);
\node[anchor=south east]at(0.5,0.5){$i$};
\node[anchor=south west]at(3.5,0.5){$i+4$};
\node[anchor=south east]at(1.5,1.5){$i+1$};
\node[anchor=south west]at(2.5,1.5){$i+3$};
}.
\end{eqnarray*}
This implies that $\mathbf{e}'$ is locally changed from $(i,i+2,\ldots)$ to $(i,i+1,i+3,i+4,\ldots)$.
Thus we have $\mathbf{e}=\mathbf{e}'$ in case of size $n$.

When we have a row or column insertion, a local $\mathbf{e}$ corresponding to the boundary edge is $(i)$.
From Figure \ref{IPbrc}, the former local $\mathbf{e}$ is changed to $(i,i+2)$ by a row insertion.
On the other hand, in case of $\mathbf{e}'$, we insert a $UD$-path 
in the middle of a local path $UU$, $DU$ or $DD$ of $\mu_{n-1}$.
We have $UUDU$, $DUDU$ and $DUDD$, and the local $\mathbf{e}'$ is changed 
from $(i,i+1)$, $(i)$ and $(i,i+1)$ to $(i,i+1,i+3)$, $(i,i+2)$ and $(i,i+2,i+3)$ 
respectively. 
In all cases, $(j)$ is mapped to $(j,j+2)$ for some $j$.

By a row, column or diagonal insertion, we add one valid vertex on the boundary edges.
The position of the valid vertex is corresponding to the position of $\wedge$-peak 
in $\mu_{n}$. 
Once all the positions of valid vertices are fixed, one can determine $\mathbf{e}$ and $\mathbf{e}'$
uniquely. 
Thus, we have $\mathbf{e}=\mathbf{e}'$.
\end{proof}

\begin{theorem}
\label{thrmtLTab}
We have $\mathcal{T}^{\times}_{n}=\mathcal{T}_{n}$.
Especially, every tree-like tableau can be constructed from a box 
with the label $1$ by the insertion procedure recursively.
\end{theorem}
\begin{proof}
We prove Theorem by induction on the size $n$.
When $n=1$, we have a unique tree-like tableau, a box with the 
label $1$. Thus, $\mathcal{T}^{\times}_1=\mathcal{T}_{1}$.
We assume that Theorem holds up to $n$.
Let $T\in\mathcal{T}^{\times}_{n+1}$ be a tree-like tableau of size $n+1$. 
By Definition \ref{defn:iipTT}, if we perform the reverse insertion procedure,
we obtain a tree-like tableau $T'$ of size $n$.
By induction assumption, $T'$ can be obtained from the tree-like tableau of size $1$
by using the insertion procedure.
Further, since $\mathcal{T}^{\times}_{n}=\mathcal{T}_{n}$, $T'$ is unique.
Since reverse insertion procedure is the inverse of the insertion procedure,
$T$ can be obtained from $T'$ by the insertion procedure.

When $T'\in\mathcal{T}_{n}$, it is clear that a tree-like tableau constructed 
from $T'$ by the insertion procedure is in $\mathcal{T}_{n+1}$.
Thus, we have $T$ can be constructed by the insertion procedures and 
$\mathcal{T}_{n+1}=\mathcal{T}^{\times}_{n+1}$.
\end{proof}

\begin{cor}
The number of generalized tree-like tableaux in $\mathcal{T}_{n}$ is $(2n-1)!!$.
\end{cor}
\begin{proof}
From Proposition \ref{propem}, we have $e_{m}=2n$. This means that we have 
$2n+1$ ways to perform an insertion on $T\in\mathcal{T}^{\times}_{n}$.
Together with Theorem \ref{thrmtLTab}, we have 
$|\mathcal{T}_{n}|=|\mathcal{T}_{n}^{\times}|=(2n-1)|\mathcal{T}_{n-1}^{\times}|=(2n-1)!!$.
\end{proof}

Figure \ref{fig:gtTT} shows the tree-like tableaux in $\mathcal{T}_{n}$ 
at most $n=3$.
\begin{figure}[ht]
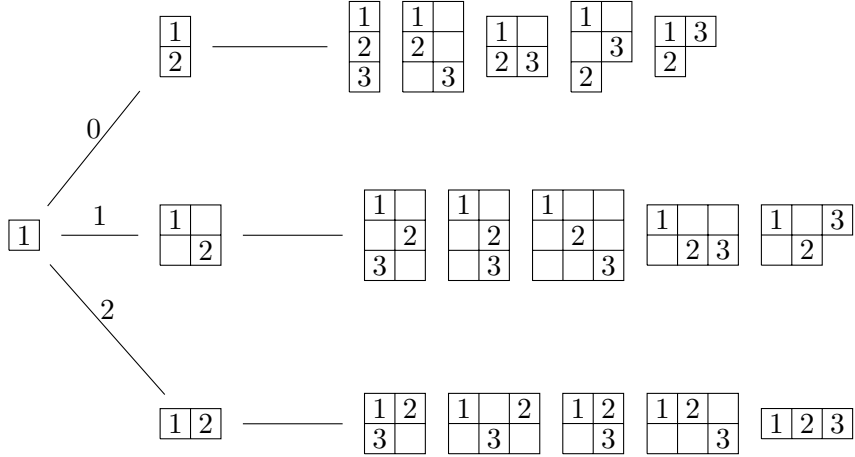

\tikzpic{-0.5}{[x=0.4cm,y=0.4cm,grow=right]
\tikzstyle{level 1}=[level distance=1.5cm, sibling distance=2.5cm]
\tikzstyle{level 2}=[level distance=2cm, sibling distance=1.5cm]
\node[text centered,anchor=east]{
\tikzpic{-0.5}{\draw(0,0)--(1,0)--(1,1)--(0,1)--(0,0);
	       \node[anchor=center] at (0.5,0.5){$1$};
	       }
      }
child {
      node[anchor=west] {
           \tikzpic{-0.5}{\draw(0,0)--(2,0)--(2,1)--(0,1)--(0,0)(1,0)--(1,1);
                          \node[anchor=center] at (0.5,0.5){$1$};\node[anchor=center] at (1.5,0.5){$2$};
                         }
           }
      child {
             node[anchor=west]{$\begin{aligned}
                            \tikzpic{-0.5}{
                                          \draw(0,0)--(2,0)--(2,2)--(0,2)--(0,0)
                                               (0,1)--(2,1)(1,0)--(1,2);
                                          \node[anchor=center] at (0.5,0.5){$3$};
                                          \node[anchor=center] at (0.5,1.5){$1$};
                                          \node[anchor=center] at (1.5,1.5){$2$};
                                          }    
                            \tikzpic{-0.5}{
                                          \draw(0,0)--(3,0)--(3,2)--(0,2)--(0,0)
                                               (0,1)--(3,1)(1,0)--(1,2)(2,0)--(2,2);
                                          \node[anchor=center] at (1.5,0.5){$3$};
                                          \node[anchor=center] at (2.5,1.5){$2$};
                                          \node[anchor=center] at (0.5,1.5){$1$};
                                          }   
                            \tikzpic{-0.5}{
                                          \draw(0,0)--(2,0)--(2,2)--(0,2)--(0,0)
                                               (1,0)--(1,2)(0,1)--(2,1);
                                          \node[anchor=center] at (0.5,1.5){$1$};
                                          \node[anchor=center] at (1.5,1.5){$2$};
                                          \node[anchor=center] at (1.5,0.5){$3$};
                                          }
                             \tikzpic{-0.5}{
                                          \draw(0,0)--(3,0)--(3,2)--(0,2)--(0,0)(0,1)--(3,1)(1,0)--(1,2)
                                               (2,0)--(2,2);
                                          \node[anchor=center] at (0.5,1.5){$1$};
                                          \node[anchor=center] at (1.5,1.5){$2$};
                                          \node[anchor=center] at (2.5,0.5){$3$};
                                          }
                             \tikzpic{-0.5}{
                                          \draw(0,0)--(3,0)--(3,1)--(0,1)--(0,0)
                                               (1,0)--(1,1)(2,0)--(2,1);
                                          \node[anchor=center] at (0.5,0.5){$1$};
                                          \node[anchor=center] at (1.5,0.5){$2$};
                                          \node[anchor=center] at (2.5,0.5){$3$};
                                          }\hspace{1.5cm} \end{aligned}$ 
                            }
            }
      edge from parent node[above]{$2$}
      }
child {
      node[anchor=west] {
           \tikzpic{-0.5}{\draw(0,0)--(2,0)--(2,2)--(0,2)--(0,0)(0,1)--(2,1)(1,0)--(1,2);
                          \node[anchor=center] at (0.5,1.5){$1$};\node[anchor=center] at (1.5,0.5){$2$};
                         }
           }
      child {
             node[anchor=west]{$\begin{aligned}
                            \tikzpic{-0.5}{
                                          \draw(0,0)--(2,0)--(2,3)--(0,3)--(0,0)
                                               (0,1)--(2,1)(0,2)--(2,2)(1,0)--(1,3);
                                          \node[anchor=center] at (0.5,0.5){$3$};
                                          \node[anchor=center] at (1.5,1.5){$2$};
                                          \node[anchor=center] at (0.5,2.5){$1$};
                                          }    
                            \tikzpic{-0.5}{
                                          \draw(0,0)--(2,0)--(2,3)--(0,3)--(0,0)
                                               (0,1)--(2,1)(0,2)--(2,2)(1,0)--(1,3);
                                          \node[anchor=center] at (1.5,0.5){$3$};
                                          \node[anchor=center] at (1.5,1.5){$2$};
                                          \node[anchor=center] at (0.5,2.5){$1$};
                                          }   
                            \tikzpic{-0.5}{
                                          \draw(0,0)--(3,0)--(3,3)--(0,3)--(0,0)
                                               (0,1)--(3,1)(0,2)--(3,2)(1,0)--(1,3)(2,0)--(2,3);
                                          \node[anchor=center] at (2.5,0.5){$3$};
                                          \node[anchor=center] at (1.5,1.5){$2$};
                                          \node[anchor=center] at (0.5,2.5){$1$};
                                          } 
                             \tikzpic{-0.5}{
                                          \draw(0,0)--(3,0)--(3,2)--(0,2)--(0,0)
                                               (0,1)--(3,1)(1,0)--(1,2)(2,0)--(2,2);
                                          \node[anchor=center] at (2.5,0.5){$3$};
                                          \node[anchor=center] at (1.5,0.5){$2$};
                                          \node[anchor=center] at (0.5,1.5){$1$};
                                          }
                             \tikzpic{-0.5}{
                                          \draw(0,0)--(2,0)--(2,1)--(3,1)--(3,2)--(0,2)--(0,0)
                                               (0,1)--(2,1)--(2,2)(1,0)--(1,2);
                                          \node[anchor=center] at (2.5,1.5){$3$};
                                          \node[anchor=center] at (1.5,0.5){$2$};
                                          \node[anchor=center] at (0.5,1.5){$1$};
                                          }\hspace{1.5cm} \end{aligned}$ 
                            }
            }
            edge from parent node[above]{$1$}
      }
child {
      node[anchor=west] {
           \tikzpic{-0.5}{\draw(0,0)--(1,0)--(1,2)--(0,2)--(0,0)(0,1)--(1,1);
                          \node[anchor=center] at (0.5,0.5){$2$};\node[anchor=center] at (0.5,1.5){$1$};
                         }
           }
      child{
           node[anchor=west]{$\begin{aligned}
                            \tikzpic{-0.5}{                                                  
	                                  \draw(0,0)--(1,0)--(1,3)--(0,3)--(0,0)
                                               (0,1)--(1,1)(0,2)--(1,2);
                                          \node[anchor=center] at (0.5,0.5){$3$};
                                          \node[anchor=center] at (0.5,1.5){$2$};
                                          \node[anchor=center] at (0.5,2.5){$1$};
                                          }    
                            \tikzpic{-0.5}{
                                          \draw(0,0)--(2,0)--(2,3)--(0,3)--(0,0)(0,1)--(2,1)
                                               (0,2)--(2,2)(1,0)--(1,3);
                                          \node[anchor=center] at (1.5,0.5){$3$};
                                          \node[anchor=center] at (0.5,1.5){$2$};
                                          \node[anchor=center] at (0.5,2.5){$1$};
                                          }   
                            \tikzpic{-0.5}{
                                          \draw(0,0)--(2,0)--(2,2)--(0,2)--(0,0)
                                               (0,1)--(2,1)(1,0)--(1,2);
                                          \node[anchor=center] at (0.5,0.5){$2$};
                                          \node[anchor=center] at (1.5,0.5){$3$};
                                          \node[anchor=center] at (0.5,1.5){$1$};
                                          } 
                            \tikzpic{-0.5}{
                                          \draw(0,0)--(1,0)--(1,1)--(2,1)--(2,3)--(0,3)--(0,0)
                                               (0,1)--(1,1)(0,2)--(2,2)(1,1)--(1,3);
                                          \node[anchor=center] at (1.5,1.5){$3$};
                                          \node[anchor=center] at (0.5,0.5){$2$};
                                          \node[anchor=center] at (0.5,2.5){$1$};
                                          }
                             \tikzpic{-0.5}{
                                          \draw(0,0)--(1,0)--(1,1)--(2,1)--(2,2)--(0,2)--(0,0)
                                               (0,1)--(1,1)--(1,2);
                                          \node[anchor=center] at (0.5,0.5){$2$};
                                          \node[anchor=center] at (0.5,1.5){$1$};
                                          \node[anchor=center] at (1.5,1.5){$3$};
                                          }\hspace{1.5cm} \end{aligned}$ 
                            }
           }
      edge from parent node[above]{$0$}
      };
}
\caption{Generation tree for tree-like tableaux of size at most $3$}
\label{fig:gtTT}
\end{figure}

We have $\widetilde{\mathcal{T}}_{n}=\mathcal{T}_{n}$ for $n=1$ and $2$.
We have two configurations in $\widetilde{\mathcal{T}}_{3}\setminus\mathcal{T}_{3}$,
which are non-admissible configurations. See Figure \ref{fig:naGTT} for them.
\begin{figure}[ht]
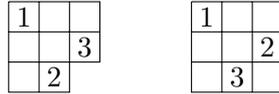

\tikzpic{-0.5}{[x=0.4cm,y=0.4cm]
\draw(0,0)--(3,0)--(3,-2)--(2,-2)--(2,-3)--(0,-3)--(0,0)
     (0,-1)--(3,-1)(0,-2)--(2,-2)(1,0)--(1,-3)(2,0)--(2,-2);
\node at(0.5,-0.5){$1$};\node at(1.5,-2.5){$2$};\node at(2.5,-1.5){$3$};
}\qquad
\tikzpic{-0.5}{[x=0.4cm,y=0.4cm]
\draw(0,0)--(3,0)--(3,-3)--(0,-3)--(0,0)(0,-1)--(3,-1)(0,-2)--(3,-2)(1,0)--(1,-3)(2,0)--(2,-3);
\node at(0.5,-0.5){$1$};\node at(1.5,-2.5){$3$};\node at(2.5,-1.5){$2$};
}
\caption{Non-admissible configurations in $\widetilde{\mathcal{T}}_{n}$.}
\label{fig:naGTT}
\end{figure}

\subsection{\texorpdfstring{Enumerations of diagrams in $\widetilde{\mathcal{T}}_{n}$}{
Enumerations of diagrams in Ttilde}}
Let $C(z):=\sum_{n\ge0}c_{n}z^{n}/n!$ be a formal power series of $z$ 
satisfying the initial condition $c_{0}=1$ and satisfying
\begin{eqnarray}
\label{Intformula}
C(z)=1+C(z)^{2}\int \frac{dz}{C(z)}.
\end{eqnarray}
First few values of $c_{n}$'s are 
\begin{eqnarray*}
\begin{array}{c|ccccccc}
n & 0 & 1 & 2 & 3 & 4 & 5 \\ \hline
c_{n} & 1 & 1 & 3 & 17 & 147 & 1729
\end{array},
\end{eqnarray*}
and corresponds to the sequence A234289 in \cite{Slo}.

\begin{theorem}
The number of diagrams in $\widetilde{\mathcal{T}}_{n}$ is given by $c_{n}$.
\end{theorem}

\begin{proof}
We denote by $c_{n}(k,l)$ be the number of diagrams in $\widetilde{\mathcal{T}}_{n}$
with $k$ rows and $l$ columns.
Let $\Delta_n$ be the set of points defined by
\begin{eqnarray*}
\Delta_{n}:=\{(k,l): 1\le k\le n, 1\le l\le n, n+1\le k+l\le 2n\}.
\end{eqnarray*}
Given $T\in\widetilde{\mathcal{T}}_{n}$ with $k$ rows and $l$ columns,
we have $k$ ways for the column insertion to produce a diagram with $k$ rows and $l+1$ columns, 
$l$ ways for the column insertion to produce a diagram with $k+1$ rows and $l$ columns,
and $k+l-1$ ways for the diagonal insertion to produce a diagram with $k+1$ rows and $l+1$ columns.
The $c_{n}(k,l)$ satisfies the recurrence relation
\begin{eqnarray*}
c_{n+1}(k,l)=k c_{n}(k,l-1)+l c_{n}(k-1,l)+(k+l-3)c_{n}(k-1,l-1),
\end{eqnarray*}
with the initial condition $c_{1}(1,1)=1$.
Note that if we define $c_{0}(0,0):=-1$, 
we have $ts((1+t)\partial_{t}+(1+s)\partial_{s}-1)c_{0}(0,0)=ts$.
Thus, we define the generating function $C(z,t,s)$ by
\begin{eqnarray*}
C(z,t,s):=-1+\sum_{n\ge 1}\sum_{(k,l)\in\Delta_{n}}c_{n}(k,l)\frac{z^{n}}{n!}t^{k}s^{l}.
\end{eqnarray*}

The recurrence relation can be written in terms of a partial differential equation:
\begin{eqnarray*}
\partial_{z}C(z,t,s)=ts((1+t)\partial_{t}+(1+s)\partial_{s}-1)C(z,t,s).
\end{eqnarray*}
We have a symmetry between $t$ and $s$, {\it i.e.,} $C_{n}(z,t,s)=C_{n}(z,s,t)$, 
and $\lim_{s\rightarrow t}(\partial_{t}+\partial_{s})t^{k}s^{l}=(k+l)t^{k+l-1}$ 
is equal to $\partial_{t}t^{k+l}$.
Therefore, by defining $C(z,t):=C(z,t,t)$, $C(z,t)$ satisfies
\begin{eqnarray}
\label{PDEforC}
\partial_{z}C(z,t)=t^{2}((1+t)\partial_{t}-1)C(z,t),
\end{eqnarray}
where $C(z,t)$ has an expansion
\begin{eqnarray*}
C(z,t)=t+\sum_{n\ge1}\sum_{(k,l)\in\Delta_{n}}c_{n}(k,l)\frac{z^{n}}{n!}t^{k+l}.
\end{eqnarray*}

If we expand $C(z,t):=\sum_{n\ge0}c_{n}z^{n}/n!$, Eqn.(\ref{PDEforC}) implies that 
the coefficients $\{c_{n}:n\ge0\}$ satisfy 
\begin{eqnarray}
\label{PDFforc}
\partial_{t}c_{n}=\frac{c_{n+1}+t^{2}c_{n}}{t^{2}(t+1)}.
\end{eqnarray}
If we define 
\begin{eqnarray}
\label{RRforc}
\alpha_{n+1}:=-(1+t)\frac{c_{n+1}}{n!}+\sum_{k=0}^{n}\left(
\frac{c_{k}}{k!}\frac{c_{n-k}}{(n-k)!}+\frac{c_{k+1}}{k!}\frac{c_{n-k}}{(n-k)!}
\right),
\end{eqnarray}
we obtain 
\begin{eqnarray}
\label{PDEforalpha}
\partial_{t}\alpha_{n}=\frac{1}{t^{2}(t+1)}(2\alpha_{n}+n\alpha_{n+1}),
\end{eqnarray}
where we have used Eqn.(\ref{PDFforc}).
Since we have $c_{0}=t$ and $c_{1}=t^{2}$, we have 
$\alpha_{1}=c_0^2-(1+t)c_{1}+c_{1}c_{0}=0$.
From Eqn.(\ref{PDEforalpha}) and $\alpha_1=0$, we have $\alpha_{n}=0$ for $n\ge0$.
The right-hand side of Eqn.(\ref{RRforc}) can be written in terms of a partial differential 
equation:
\begin{eqnarray}
\label{PDEforCzt}
C(z,t)^2-(1+t)\partial_{z}C(z,t)+C(z,t)\partial_{z}C(z,t)=0,
\end{eqnarray}
with the initial condition $C(0,t)=t$.
Note that the number of diagrams in $\widetilde{T}_{n}$ is given by 
$c_{n}|_{t=1}$.
Thus, by putting $t=1$ in Eqn.(\ref{PDEforCzt}), we get 
\begin{eqnarray*}
C(z,1)^{2}=2\frac{dC(z,1)}{dz}-C(z,1)\frac{dC(z,1)}{dz}.
\end{eqnarray*} 
If we integrate this equation, we obtain Eqn.(\ref{Intformula}).
\end{proof}

A few explicit evaluations of $c_{n}(k,l)$ are given by 
\begin{eqnarray*}
c_{n}(n,n)&=&
\begin{cases}
1 & n=1, \\
(2n-3)!! & n\ge2, 
\end{cases}
\\
c_{n}(n,n-1)&=&\frac{2^{-n}(2n-3)!!}{3\cdot(n-2)!}(2^{n}(n-2)\cdot(n-2)!+12), \\
c_{n}(n-1,n-1)&=&\frac{2}{9}(n^{3}-n-6)\cdot(2n-5)!!
\end{eqnarray*}
and 
\begin{eqnarray*}
c_{n}(1,m)&=&\delta_{n,m}, \\
c_{n}(2,m)&=&(2^{m+1}-m-2)\delta_{n-1,m}+(2^{m}-m-1)\delta_{n,m}, 
\end{eqnarray*}
\begin{multline*}
c_{n}(3,m)=\left( 3^{m+2}-\frac{1}{2}(m+3)(2^{m+3}-m-2) \right)\delta_{n-2,m}  \\
+\left( m^{2}-4m(2^{m}-1)+\frac{3}{2}(5\cdot 3^{m}+2^{m+3}+3) \right)\delta_{n-1,m} \\
+\frac{1}{2}\left( m(m-2^{m+1}+3)+3^{m+1}-3\cdot 2^{m+1}+3 \right)\delta_{n,m},
\end{multline*}
where $\delta_{n,m}$ is the Kronecker delta.

\begin{prop}
The average half-perimeter in a tree-like tableau of size $n$ in 
$\widetilde{\mathcal{T}}_{n}$ is 
\begin{eqnarray*}
H_{n}=\frac{1}{2}\left(1+\frac{c_{n+1}}{c_{n}}\right),
\end{eqnarray*}
and equivalently, the average number of diagonal points is
\begin{eqnarray}
\label{eqn:avdp}
D_{n}=\frac{1}{2}\left(\frac{c_{n+1}}{c_{n}}-2n-1\right).
\end{eqnarray}
\end{prop}
\begin{proof}
Recall that the left hand side of Eqn. (\ref{PDFforc}) 
is equal to 
\begin{eqnarray*}
\sum_{(k,l)\in\Delta_{n}}(k+l)c_{n}(k,l)t^{k+l-1}.
\end{eqnarray*}
The half-perimeter of a tree-like tableau is $k+l$.
Thus, setting $t=1$ in Eqn. (\ref{PDFforc}) and dividing it by $c_{n}(t=1)$ 
gives the average half-perimeter in $\widetilde{\mathcal{T}}_{n}$, {\it i.e.},
\begin{eqnarray*}
H_{n}&=&\frac{\partial_{t}c_{n}}{c_{n}}\bigg|_{t=1}, \\
&=&\frac{1}{2}\left(1+\frac{c_{n+1}}{c_{n}}\right).
\end{eqnarray*}
The average number of diagonal points $D_{n}$ is related to the 
average half-perimeter $H_{n}$ by 
\begin{eqnarray*}
D_{n}=H_{n}-n-1.
\end{eqnarray*}
Thus, we obtain Eqn. (\ref{eqn:avdp}).
\end{proof}

As in \cite{ABN11}, we define the {\it crossing} boxes as the boxes which form
a ribbon added in the insertion procedure.

Let $\mathrm{Cr}(n,h)$ be the total number of crossing boxes in the set 
of tree-like tableaux of size $n$ and half-perimeter $h$.
We denote by $A_{n}(h)$ the number of tree-like tableaux of size $n$ 
and of half-perimeter $h$.

\begin{prop}
The number $\mathrm{Cr}(n,h)$ satisfies the recurrence relation:
\begin{multline}
\label{eqn:Crrec1}
\mathrm{Cr}(n+1,h)=(h-3)\cdot\mathrm{Cr}(n,h-2)+(h-1)\cdot\mathrm{Cr}(n,h-1) \\
+\frac{1}{6}(h-1)(h-2)(h-3)\cdot A_{n-1}(h-2) \\
+\frac{1}{3}(h-2)(h-3)(h-4)\cdot A_{n-1}(h-3) \\
+\frac{1}{6}(h-3)(h-4)(h-5)\cdot A_{n-1}(h-4).
\end{multline}
\end{prop} 
\begin{proof}
Let $T$ be a tree-like tableau of size $n$ and half-perimeter $h$.
We label its boundary edges $e_{0}(T),\ldots,e_{h-1}(T)$ from the southwest 
to the northeast edge. 
We also label its boundary vertices $v_{0}(T),\ldots,v_{h-2}(T)$ from 
the southwest to the northeast boundary vertices.
We have $h$ ways to perform a row or column insertion, and $h-1$ ways 
to perform a diagonal insertion.

Recall that $A_{n-1}(h)$ is the number of tree-like tableaux $T'$ of size $n-1$
and half-perimeter $h$.
To obtain a tree-like tableau of size $n+1$, we perform two successive a row, column,
or diagonal insertions.
We denote by $T''$ a tree-like tableau of size $n$ after one insertions.
We have three cases for insertions to obtain a tree-like tableau of 
size $n+1$ from $T'$: 
1) two insertions are row or column insertions,
2) one of the two insertions is a row or column insertion and the other is a diagonal insertion,
and 3) both insertions are diagonal ones.

\paragraph{\bf Case 1)}
When we perform two row or column insertions, the half-perimeter of $T'$ is increased 
by two. Further, when the insertion point is $e_{j}(T')$ and $e_{i}(T'')$ with $i<j$,
we add $j-i$ boxes as a ribbon. 
The total number of crossings in the second insertion is given by $1+2+\ldots+j=j(j+1)/2$ 
with $0\le j\le h-1$.

\paragraph{\bf Case 2)}
The half-perimeter of $T'$ is increased by three by the two insertions.
When the first insertion point is $e_{j}(T')$ and the second insertion point is $v_{i}(T'')$ 
with $j>i$, we add $j-i$ boxes as a ribbon.
The total number of crossing boxes added in the second insertion is given by $1+2+\ldots+j=j(j+1)/2$
with $1\le j\le h-1$.
Similarly, if the first insertion point is $v_{j}(T')$ and the second insertion point 
is $e_{i}(T'')$ with $j\ge i$, we add $j-i+1$ boxes as a ribbon.
The total number of crossing boxes added in the second insertion is given by $1+2+\ldots+(j+1)=(j+1)(j+2)/2$
with $0\le j\le h-2$.

\paragraph{\bf Case 3)}
The half-perimeter of $T'$ is increased by four by the two insertions.
When the first insertion point is $v_{j}(T')$ and the second insertion point is $v_{i}(T'')$ with $j\ge i$,
we add $j-i+1$ boxes as a ribbon.
Thus, the total number of crossing boxes added in the second insertion is given 
by $1+2+\ldots+(j+1)=(j+1)(j+2)/2$ with $0\le j\le h-2$.

From these observations, we have 
\begin{multline}
\label{eqn:Crrec2}
\mathrm{Cr}(n+1,h)-(h-3)\mathrm{Cr}(n,h-2)-(h-1)\mathrm{Cr}(n,h-1) \\
=A_{n-1}(h-2)\sum_{j<h-2}\frac{1}{2}j(j+1)
+2 A_{n-1}(h-3)\sum_{j<h-3}\frac{1}{2}j(j+1) \\	
+A_{n-1}(h-4)\sum_{j<h-4}\frac{1}{2}j(j+1).
\end{multline}
Substituting 
$\sum_{j\le s}\frac{1}{2}j(j+1)=\frac{1}{6}s(s+1)(s+2)$ into 
Eqn. (\ref{eqn:Crrec2}), we obtain Eqn. (\ref{eqn:Crrec1}).
\end{proof}

First few values of $\mathrm{Cr}(n,h)$ are 
\begin{eqnarray*}
&&\mathrm{Cr}(1,2)=0, \\
&&\mathrm{Cr}(2,3)=0,\ \mathrm{Cr}(2,4)=0, \\
&&\mathrm{Cr}(3,4)=1,\ \mathrm{Cr}(3,5)=2,\ \mathrm{Cr}(3,6)=1, \\
&&\mathrm{Cr}(4,5)=12,\ \mathrm{Cr}(4,6)=39,\  \mathrm{Cr}(4,7)=42,\  \mathrm{Cr}(4,8)=15. 
\end{eqnarray*}

Let $B_{n}(k,l)$ with $n\ge1$, $1\le k,l\le n$ and $n+1\le k+l\le 2n$
be an integers satisfying 
\begin{eqnarray}
\label{RRBn1}
B_{n}(k,l)&=&B_{n-1}(k-1,l)+k, \\
\label{RRBn2}
B_n(k,l)&=&B_{n-1}(k,l-1)+l,
\end{eqnarray}
with the initial condition $B_{1}(1,1)=1$.
This recurrence relation gives 
\begin{eqnarray*}
B_{n}(k,l)=\frac{1}{2}k(k+1)+\frac{1}{2}l(l+1)-1.
\end{eqnarray*}

Let $T\in\mathcal{T}_{n}$ be a diagram with $k$ rows and $l$ rows.
Recall that we have three types of the insertion procedures.
When we perform a row or column insertion at the $i$-th boundary edge in 
a row or column, we add $i$ boxes to $T$. 
We define the weight of these boxes is one.
On the other hand, we define a weight of boxes associated with 
a diagonal insertion on $T$ as follows.
Let $b$ be the added box with the label $n+1$.
The diagonal insertion means that we have no boxes with a label
above in the same column as $b$ and left to $b$ in the same row.
The weight of the box $b$ with the label $n+1$ is one.
We call the boxes above $b$ in the same column arm boxes and 
the boxes left to $b$ in the same row leg boxes.
We define the weight of arm and leg boxes as follows:
an arm (resp. leg) box has weight one if it is not in the same row or 
column as an arm or leg box associated with another diagonal point
whose label is larger than $n+1$. 
Otherwise, we define the weight of the box is zero.
Let $\mathrm{arm}(b)$ (resp. $\mathrm{leg}(b)$) be the weighted 
sum of the arm (resp. leg) boxes associated with $b$.
Then, the weighted sum of non-crossing boxes associated with 
a diagonal point $b$ is defined as the absolute value 
\begin{eqnarray}
\label{wtdiag}
|\mathrm{arm}(b)-\mathrm{leg}(b)|+1,
\end{eqnarray}
where the plus one comes from the weight of $b$.

Let $\mathrm{NCr}(n,k,l)$ be the total number of non-crossing boxes 
in the set of tree-like tableaux of size $n$ with $k$ rows and $l$ columns.
We denote by $A_{n}(k,l)$ the total number of tree-like tableaux of size 
$n$ with $k$ rows and $l$ columns.

\begin{prop}
The number $\mathrm{NCr}(n,k,l)$ satisfies the recurrence relation:
\begin{multline}
\label{eqn:recNCr}
\mathrm{NCr}(n+1,k,l)=(k+l-3)\cdot\mathrm{NCr}(n,k-1,l-1)+l\cdot\mathrm{NCr}(n,k-1,l)
+k\cdot\mathrm{NCr}(n,k,l-1) \\
+A_{n}(k-1,l-1)\left\{\frac{1}{2}k(k-1)+\frac{1}{2}l(l-1)-1\right\} \\
+\frac{1}{2}l(l+1)\cdot A_{n}(k-1,l)+\frac{1}{2}k(k+1)\cdot A_{n}(k,l-1).
\end{multline}
\end{prop}
\begin{proof}
Let $T$ be a tree-like tableau of size $n$, $k$ rows and $l$ columns.
If we perform a diagonal insertion, the numbers of rows and columns are increased 
by one. There are $k+l-1$ ways to perform a diagonal insertion.
When we perform a row or column insertion, the number of rows or columns is 
increased by one respectively. 
There are $l$ or $k$ ways to perform the row or column insertion respectively.
From this observation, we have 
\begin{eqnarray*}
(k+l-3)\cdot\mathrm{NCr}(n,k-1,l-1)+l\cdot\mathrm{NCr}(n,k-1,l)
+k\cdot\mathrm{NCr}(n,k,l-1).
\end{eqnarray*}

Suppose that $T$ has $k$ rows and $l$ columns. 
By a row/column insertion, we add $j$ boxes if the insertion point is 
the $j$-th boundary edge.
Since the weight of these boxes are one, the contribution of the row (resp. column)
insertion to the weighted sum of non-crossing boxes is 
given by $1+2+\cdots+l=l(l+1)/2$ (resp. $k(k+1)/2$).
Thus, we have 
\begin{eqnarray*}
\frac{1}{2}l(l+1)\cdot A_{n}(k-1,l)+\frac{1}{2}k(k+1)\cdot A_{n}(k,l-1)
\end{eqnarray*}

We compute the contribution of arm and leg boxes associated with a diagonal point.
If we remove a diagonal points and their arm and leg boxes, we have a diagram 
$T'$ with $k-1$ rows and $l-1$ columns.
We perform a diagonal insertions of this reduced diagram $T'$.
When we have a local up-right configuration on the boundary edges, {\it i.e.},
successive edges consisting of a vertical edge and a horizontal edge, 
the weighted sum of number of boxes by the diagonal insertion at vertex $(q,p)$, 
which is the vertex between the vertical edge  and the horizontal edge, is given by $|p-q|+1$ 
via Eqn. (\ref{wtdiag}).
We transform the local up-right configuration to a right-up configuration by moving the 
vertex $(q,p)$ to the vertex $(q+1,p+1)$.
The weighted sum of number of boxes by the insertion at vertex $(q+1,p+1)$ is again $|p-q|+1$.
By successive transformations, we arrive at the rectangular shape with $k-1$ rows and $l-1$ columns.
Let $B_{n}(k,l)$ be the contribution of the diagonal insertion to the weighted sum of 
non-crossing boxes.
The weighted sum of number of boxes at the vertex $(k-1,1)$ in $T'$ 
is given by $k-1$ and we have a diagram with $k$ rows and $l$ columns.
Thus, we have the recurrence relation (\ref{RRBn1}).
By a similar argument with respect to a column, we obtain the recurrence relation (\ref{RRBn2}).
Then, we have 
\begin{eqnarray*}
A_{n}(k-1,l-1)B_{n}(k-1,l-1)=A_{n}(k-1,l-1)\left\{\frac{1}{2}k(k-1)+\frac{1}{2}l(l-1)-1\right\}.
\end{eqnarray*}

Summing all over contributions, we obtain Eqn. (\ref{eqn:recNCr}).
\end{proof}

Note that $\mathrm{NCr}(n,k,l)=\mathrm{NCr}(n,l,k)$ from the symmetry of the recurrence relation.
First few values of $\mathrm{NCr}(n,k,l)$ are 
\begin{eqnarray*}
&&\mathrm{NCr}(1,1,1)=1, \\
&&\mathrm{NCr}(2,1,2)=2, \ \mathrm{NCr}(2,2,2)=2, \\
&&\mathrm{NCr}(3,1,3)=3,\ \mathrm{NCr}(3,2,2)=14, \ 
\mathrm{NCr}(3,2,3)=14, \  \mathrm{NCr}(3,3,3)=11, \\
&&\mathrm{NCr}(4,1,4)=4, \  \mathrm{NCr}(4,2,3)=46, \ \mathrm{NCr}(4,2,4)=46, \\
&&\mathrm{NCr}(4,3,3)=194, \ \mathrm{NCr}(4,3,4)=139, \ \mathrm{NCr}(4,4,4)=88.
\end{eqnarray*}

\subsection{\texorpdfstring{Enumerations of diagrams in $\mathcal{T}_{n}$}{Enumerations of 
diagrams in Tn}}
According to \cite{ABN11}, we introduce the polynomial $T_{n}(x,y)$ by
\begin{eqnarray*}
T_n(x,y):=\sum_{T\in\mathcal{T}_n}x^{\mathrm{left}(T)}y^{\mathrm{top}(T)},
\end{eqnarray*}
where $\mathrm{left}(T)$ and $\mathrm{top}(T)$ are the number of {\it left points} 
and {\it top points} in $T$.
Here, the top points (resp. left points) are the non-root points appearing 
in the first row (resp. first column) of its diagram \cite{ABN11}.
When a tableau $T$ is of size $n$, we have $2n+1$ ways to insert a point.
We have a unique way to put a point at the top row or at the left column, and 
$2n-1$ ways to hold two statistics $\mathrm{left}(T)$ and $\mathrm{top}(T)$ 
the same.
Thus, we have the recurrence relation 
\begin{eqnarray*}
T_{n+1}(x,y)=(x+y+2n-1)T_{n}(x,y),
\end{eqnarray*}
with the initial condition $T_{1}=1$.
This gives 
\begin{eqnarray*}
T_{n}(x,y)=(x+y+1)(x+y+3)\cdots (x+y+2n-3).
\end{eqnarray*}

Recall that we have two types of points, off-diagonal points and diagonal points.
An off-diagonal point $p_{0}$ is said to be attached to a diagonal point $p_{1}$ 
if $p_{1}$ is above $p_{0}$ in the same column or left to $p_{0}$ in the same 
row.
Let $X_{n}(h,p)$ be the number of tree-like tableaux such that 
it is size $n$, half-perimeter $h$ and the total number of off-diagonal points 
attached to a diagonal point is $p$.

\begin{prop}
\label{prop:recpatt}
The number $X_{n}(h,p)$ satisfies the following recurrence relation:
\begin{multline}
\label{eqn:Xhprec}
X_{n}(h,p)=(2n+1-h)(X_{n-1}(h-2,p)+X_{n-1}(h-1,p)) \\
+2(h-1-n)X_{n-1}(h-1,p-1).
\end{multline}
\end{prop}
\begin{proof}
We need a diagonal insertion to obtain a tree-like tableau $T$ of size $n$, half-perimeter $h$ and 
$p$ attached off-diagonal points from a tree-like tableau $T'$ of size $n-1$, half-perimeter $h-2$ and 
$p$ attached off-diagonal points.
There are $2n+1-h$ valid vertices in $T'$, which gives a contribution $(2n+1-n)X_{n-1}(h-2,p)$.

There are two cases to obtain $T$ by a row or column insertion: the first case is the one where 
the added off-diagonal point is not attached and the second case is the one where 
the added off-diagonal point is attached to a diagonal point.
In a tree-like tableau of size $n-1$ and half-perimeter $h-1$, we have $h-1-n$ diagonal points.
In the first case, we have $h-1-2(h-1-n)=2n+1-h$ ways to perform a row or column insertion since 
the added off-diagonal point is not attached to a diagonal point.
In the second case, we have $2(h-1-n)$ ways to perform a row or column insertion.

From these observations, we obtain Eqn. (\ref{eqn:Xhprec}).
\end{proof}

We change the variable from $h$ to $h'$ by $h=n+1+h'$.
We define $\widetilde{X}_{n}(h'):=X_{n}(h,0)$.
Then, from Proposition \ref{prop:recpatt}, $\widetilde{X}_{n}(h')$ satisfies the following recurrence relation:
\begin{eqnarray}
\label{eqn:recXtilde}
\widetilde{X}_{n}(h')=(n-h')\{\widetilde{X}_{n-1}(h'-1)+\widetilde{X}_{n-1}(h')\}
\end{eqnarray}
with $0\le h'\le n-1$ and the initial condition $\widetilde{X}_{1}(0)=1$.

\begin{prop}
The number $\widetilde{X}_{n}(h')$ is expressed as 
\begin{eqnarray}
\label{eqn:Xtildefac}
\widetilde{X}_{n}(h')=
\begin{cases}
n! & h'=0, \\
(n-h')\cdot n!\cdot f_{h'}(n), & h'\ge1,
\end{cases}
\end{eqnarray}
where $f_{h'}(n)$ is a polynomial of $n$ in degree $h'-1$ and its expansion is 
\begin{eqnarray}
\label{eqn:ltfh}
f_{h'}(n)=\frac{1}{(2h')!!}n^{h'-1}+\cdots.
\end{eqnarray}
\end{prop}
\begin{proof}
When $h'=0$, we have $\widetilde{X}_{n}(0)=n\cdot\widetilde{X}_{n-1}(0)=n!$
from the recurrence relation (\ref{eqn:recXtilde}).

For $h'\ge1$, we prove Proposition by induction on $h'$.
For $h'=1$, we have 
\begin{eqnarray*} 
\widetilde{X}_{n}(1)&=&(n-1)\{\widetilde{X}_{n-1}(0)+\widetilde{X}_{n-1}(1)\}, \\
&=&(n-1)\cdot(n-1)!+(n-1)\widetilde{X}_{n-1}(1), \\
&=&\sum_{s=1}^{n-1}(n-1)\cdots(n-s)\cdot(n-s)!, \\
&=&\sum_{s=1}^{n-1}(n-s)\cdot(n-1)! \\
&=&\frac{1}{2}(n-1)\cdot n!.
\end{eqnarray*}
For $h'\ge2$, We assume that $\widetilde{X}_{n}(h')$ can be factorized as 
Eqn. (\ref{eqn:Xtildefac}).
Then, $f_{h'}(n)$ satisfies the recurrence relation
\begin{eqnarray*}
f_{h'}(n)=\frac{n-h'}{n}f_{h'-1}(n-1)+\frac{n-1-h'}{n}f_{h'}(n-1).
\end{eqnarray*}
We multiply the both sides by $n$. Then, it is obvious that when $f_{h'-1}(n)$ 
is a polynomial of $n$ of degree $h'-2$, we have a unique polynomial $f_{h'}(n)$ 
of degree $h'-1$.
We denote by $a_{h'}$ the leading coefficient of $f_{h'}(n)$.
Then, we have $2h'\cdot a_{h'}=a_{h'-1}$ from the above recurrence relation, 
which implies Eqn. (\ref{eqn:ltfh}).
\end{proof}

The first few polynomials $f_{h'}(n)$ are 
\begin{eqnarray*}
f_{1}(n)&=&\frac{1}{2}, \\
f_{2}(n)&=&\frac{1}{24}(3n-5), \\
f_{3}(n)&=&\frac{1}{48}(n-2)(n-3), \\
f_{4}(n)&=&\frac{1}{5760}(15n^3-150n^2+485n-502), \\
f_{5}(n)&=&\frac{1}{11520}(n-4)(n-5)(3n^2-23n+38).
\end{eqnarray*}

Let $\widetilde{X}_{n}(h',p):=X_{n}(h,p)$ with $h=n+1+h'$.
The first few expressions with $p\ge1$ are 
\begin{eqnarray*}
\widetilde{X}_{n}(1,1)&=&\frac{1}{2}(n-1)(n-2)\cdot(n-1)!, \\
\widetilde{X}_{n}(2,1)&=&\frac{1}{36}(n-2)(n-3)(9n^2-19n+4)\cdot(n-2)!, \\
\widetilde{X}_{n}(3,1)&=&\frac{1}{144}(n-2)(n-3)(n-4)(9n^3-44n^2+49n-6)\cdot(n-3)!, \\
\widetilde{X}_{n}(1,2)&=&\frac{1}{2}\left((n-1)(n-6)+4\sum_{k=1}^{n-1}\frac{1}{k}\right)\cdot(n-1)!, \\
\widetilde{X}_{n}(2,2)&=&\frac{1}{216}\left\{
(n-2)(81n^3-728n^2+1487n-240)+72(3n^2-9n+2)\sum_{k=1}^{n-2}\frac{1}{k}
\right\}\cdot(n-2)!
\end{eqnarray*}

Let $A_n(k,l)$ with $1\le k,l\le n$ and $n+1\le k+l\le 2n$ 
be the Eulerian numbers of the second order.
Namely, $A_{n}(k,l)$ satisfies
\begin{eqnarray}
\label{RREuler}
A_{n+1}(k,l)=kA_{n}(k,l-1)+lA_{n}(k-1,l)+(2n+3-k-l)A_{n}(k-1,l-1),
\end{eqnarray}
with the initial condition $A_{1}(1,1)=1$.
These numbers correspond to the integer sequence A321591 in \cite{Slo}.

\begin{prop}
The number of tree-like tableaux of size $n$ with $k$ rows and 
$l$ columns is given by $A_{n}(k,l)$.
\end{prop}
\begin{proof}
Suppose $T\in\mathcal{T}_{n}$ has $k$ rows and $l$ columns.
Then, by the insertion procedure, we have $k$ tableaux 
with $k$ rows and $l+1$ columns, $l$ tableaux 
with $k+1$ rows and $l$ columns, and 
$2n+1-k-l$ tableaux with $k+1$ rows and $l+1$ columns 
in $\mathcal{T}_{n+1}$.
From this, we obtain the recurrence relation (\ref{RREuler})
\end{proof}

We introduce the Eulerian polynomial
$A_n(t,s):=\sum_{1\le k,l\le n}A_n(k,l)t^{k}s^{l}$. 
Then, we have the recurrence relation for $A_n(t,s)$:
\begin{eqnarray*}
A_{n+1}(t,s)=(2n+1)tsA_{n}(t,s)+t(1-t)s\cdot\partial_{t}A_{n}(t,s)
+ts(1-s)\cdot\partial_{s}A_{n}(t,s),
\end{eqnarray*}
with the initial condition $A_{1}(t,s)=ts$, where 
we denote a partial derivative by $\partial_{x}:=\partial/\partial x$.
If we differentiate the recurrence relation 
once, and 
substitute $t=s=1$ in this case, we obtain equations:
\begin{eqnarray*}
\partial_{x}A_{n+1}(1,1)&=&(2n+1)A_{n}(1,1)+2n\cdot \partial_{x}A_{n}(1,1), 
\end{eqnarray*}
where $x=t$ or $s$. 
Since we have $A_{n}(1,1)=(2n-1)!!$, we get 
\begin{eqnarray*}
\partial_{t}A_{n}(1,1)=\partial_{s}A_{n}(1,1)
&=&\frac{1}{3}(2n+1)!!, 
\end{eqnarray*}

\begin{prop}
The average half-perimeter of a tree-like tableau of size $n$ in $\mathcal{T}_{n}$
is given by 
\begin{eqnarray}
\label{eqn:avtt1}
\frac{2}{3}(2n+1),
\end{eqnarray}
or equivalently the average number of diagonal dots is 
\begin{eqnarray}
\label{eqn:avttdd}
\frac{1}{3}(n-1).
\end{eqnarray}
\end{prop}
\begin{proof}
The total number of half-perimeters of a tree-like tableau of size $n$ is given 
by  
\begin{eqnarray*}
\sum_{(k,l)\in\Delta_{n}}(k+l)A_{n}(k,l),
\end{eqnarray*}
which is equal to 
\begin{eqnarray}
\label{eqn:avtt2}
(\partial_{t}+\partial_{s})A_{n}(t,s)\big|_{(t,s)=(1,1)}=\frac{2}{3}(2n+1)!!.
\end{eqnarray}
Dividing Eqn. (\ref{eqn:avtt2}) by $A_{n}(1,1)=(2n-1)!!$, we obtain Eqn. (\ref{eqn:avtt1}).

The number of diagonal dots is the half-perimeter minus $n+1$, which gives Eqn. (\ref{eqn:avttdd})
\end{proof}

\subsection{Crossings in a tree-like tableau}
\label{sec:crossttab}
Let $T$ be a natural label of the tree $\mathrm{Tree}(\lambda)$, 
$B$ be the tree-like tableau for $T$ and $\mathbf{h}:=(h_1,\ldots,h_n)$ be its insertion history.
From the definition of the insertion procedure of tree-like tableaux, 
we add a ribbon to a tree-like tableau when $h_{n-1}>h_{n}$.
Let $e_i$, $1\le i\le n$, be the edge with label $i$ in $T$.

\begin{prop}
Suppose $h_{n}>h_{n+1}$. Then, we have
\begin{eqnarray}
\label{eqn:hh}
\begin{aligned}
h_{n}-h_{n+1}&=\#\{k| e_{n}\uparrow e_{k}\}+\#\{k| e_{n+1}\uparrow e_{k}\} \\
&+2\cdot\#\{k|e_{n+1}\rightarrow e_{k}\rightarrow e_{n}\text{ and } k<n\}
-2\cdot\#\{k|e_{n}\uparrow e_{k} \text{ and }e_{n+1}\uparrow e_{k}\}.
\end{aligned}
\end{eqnarray}
\end{prop}
\begin{proof}
Let $\lambda$ be a Dyck path corresponding to $\mathbf{h}$.
It is clear from the definition of $\mathbf{e}$ that 
$h_{n}-h_{n+1}$ is equal to the number of edges, plus the number of valid vertices, minus the 
number of a local path $DU$ between the $h_{n+1}$-th and $h_{n}$-th position in $\lambda$.
Note that the number of valid vertices is equal to the number of a local path $DU$.
So, to show Proposition, it is enough to count the number of edges between $h_{n+1}$-th and $h_{n}$-th position.
Actually, the right hand side of Eqn. (\ref{eqn:hh}) counts the number of such edges.
\end{proof}

Let $R_{n}$ be the number of boxes in a ribbon added in a tree-like tableau by the $n$-th step 
of the insertion procedure. 
\begin{prop}
\label{prop:TTR}
Suppose $h_{n}>h_{n+1}$. 
Then, we have 
\begin{eqnarray*}
R_{n+1}=h_{n}-h_{n+1}-\#\{k<n | e_{n+1}\rightarrow e_{k}\rightarrow e_{n} 
\text{ and $e_{k}$ is connected to a leaf}\}.
\end{eqnarray*}
\end{prop}
\begin{proof}
If we add a ribbon in a Dyck tableau by the insertion procedure, the number of added box 
is equal to $h_{n}-h_{n+1}-m$ where $m$ is the number of valid vertices between the 
$h_{n+1}$-th and $h_{n}$-th position in $\lambda$.
From the definition of $\mathbf{e}$, we have a valid vertex if we have a $UD$ path 
in $\lambda$.
In the language of a tree, a local path $UD$ corresponds to an edge connected to 
a leaf.
Thus, $m$ is equal to the valid vertices and we have 
\begin{eqnarray*}
m=\#\{k<n | e_{n+1}\rightarrow e_{k}\rightarrow e_{n} 
\text{ and $e_{k}$ is connected to a leaf}\}.
\end{eqnarray*}
\end{proof}

\subsection{Dyck tableaux and tree-like tableaux}
In this section, we consider two combinatorial  objects: 
Dyck tableaux (see Section \ref{sec:DTab}) and tree-like 
tableaux (see Section \ref{sec:TTab}).
As pointed in \cite{ABDH11,ABN11}, Dyck tableaux and tree-like tableaux 
have similar recursive structure based on their insertion 
procedures.

Let $R_{n}$ be the number defined in Section \ref{sec:crossttab}.
\begin{theorem}
We have a bijection between Dyck tableaux $D$ and tree-like 
tableaux $T$ satisfying the following properties.
\begin{enumerate}
\item The number of labeled box in $T$ is the number of labeled box in $D$.
\item There is a ribbon between the labeled box $n$ and the labeled box $n+1$
in $D$ if and only if there is a ribbon between the labeled box $n$ and the labeled 
box $n+1$ in $D$ or if the labeled box $n+1$ is below the labeled box $n$ in $T$.
\item When $m_{i}:=h_{i-1}-h_{i}-R_{i}>0$ for $1\le i\le n-1$ in $T$, 
we have $M:=\sum_{i}m_{i}$ proper shadow boxes of $D$.
\end{enumerate}
\end{theorem}
\begin{proof}

Both Dyck tableaux and tree-like tableaux are characterized by a natural label of 
size $n$ and bijective to the natural label.
Thus, we have a natural bijection between Dyck tableaux and tree-like tableaux 
through natural labels.
We show that the bijection satisfies the three properties.
\begin{enumerate}
\item Obvious from the insertion procedures of Dyck tableaux and tree-like tableaux.
\item Let $\mathbf{h}$ be an insertion history.
We add a ribbon in the insertion procedure in $D$ if and only if $h_{n+1}\le h_{n}$.
Similarly, we add a ribbon in $T$ if and only if $h_{n+1}<h_{n}$.
The box with the label $n+1$ is below $n$ in $T$ if and only if $h_{n+1}=h_{n}$.
\item From Proposition \ref{prop:TTR}, $m_{i}$ is equal to the number of edges $e_{k}$
such that $e_{k}$ is right to $e_{i}$ and left to $e_{i-1}$ with $k<i-1$, and 
$e_{k}$ is connected to a leaf.
When we insert $i$ in $D$, $m_{i}$ is precisely the number of proper shadow boxes
added for the insertion of $i$.
Thus, we have $M:=\sum_{i}m_{i}$ proper shadow boxes in $D$.
\end{enumerate}
\end{proof}

\section{Relations among bijections on Dyck tilings}

When $\lambda$ is a Dyck path, we denote by 
$T(\lambda)$ a natural label of the tree $\mathrm{Tree}(\lambda)$.
Recall that the operations $\mathrm{ref}$ and $\alpha$ are involutions 
introduced in Section \ref{sec:alpha}.

Let $\lambda$ be a general Dyck path (not necessarily a zigzag path) and $U$ be a 
decreasing label of the tree 
$\mathrm{Tree}(\lambda)$.
We define a label $U^{\vee}$ from $U$ by 
\begin{eqnarray}
\label{eqn:decvee}
\mathrm{label}(e):=\#\{e'| e\rightarrow e', U(e)>U(e')\}
\end{eqnarray}
where $e$ is an edge in $U^{\vee}$.

\subsection{The DTS bijection}
Let $\lambda$ be a Dyck path, $T$ be a natural label of the tree $\mathrm{Tree}(\lambda)$, 
$S:=\alpha(T)$.  
We construct a decreasing tree $U$ from $T$ by the following operation.
Let $e_{i}$ for $i\in[1,n]$ be the edges of $T$ with the label $i$.
Take an edge $e_{n_0}$ in $T$. Suppose that $e_{n_1}$ is a child of $e_{n_0}$
and $n_{1}$ is the minimum satisfying $n_{0}<n_{1}$.
We denote this relation by $e_{n_{0}}\nearrow e_{n_1}$. 
Then, we have a unique chain of edges starting from $e_{n_0}$:
\begin{eqnarray}
\label{eqn:relne}
e_{n_0}\nearrow e_{n_1}\nearrow \cdots \nearrow e_{n_p}.
\end{eqnarray}
We first choose $n_{0}:=1$ and $n_{0}<n_1<\ldots<n_{p}$ with maximal $p$.
We change the label $n_{i}$ of the edge $e_{n_i}$ to $n_{i+1}$ 
for $i\in[0,p-1]$, and $n_{p}$ to $n_{0}$. 
The integer $1$ is on the edge connected to a leaf.
We choose $n_{0}=2$ and have the maximal chain (\ref{eqn:relne}).
Then, we change the labels as above.
The integer $2$ is on the edge connected to a leaf, or on the parent 
edge of the edge with the label $1$.
We continue this procedure until we obtain a decreasing tree.
We call this procedure {\it inverse cyclic operation} and 
denote by $U$ the new decreasing tree obtained from $T$.

The cyclic operation and inverse cyclic operation satisfy the following 
property.
Given an edge $e\in T$, we denote by $T(e)$ the label of the edge $e$ in $T$. 
\begin{prop}
\label{prop:US}
Let $T$ be a natural label, $S:=\alpha(T)$, and $U$ a decreasing tree 
obtained from $T$ by the inverse cyclic operation.
We have
\begin{eqnarray*}
U(e)+S(e)=n+1,
\end{eqnarray*}
where $e\in\mathrm{Tree}(\lambda)$.
\end{prop}
\begin{proof}
We compare the action of the cyclic operation on $\overline{T}$ with 
the one of the inverse cyclic operation on $T$.
Let $e$ be an edge of $\mathrm{Tree}(\lambda)$.
Suppose that $n_{0}<n_{1}<\cdots<n_{p}$ satisfies 
$e_{n_{0}}\nearrow e_{n_{1}}\nearrow \cdots \nearrow e_{n_{p}}$ 
in $T$ and $p$ is maximal.
By taking the bar involution on $T$, we have 
$e_{\overline{n_{0}}}\searrow e_{\overline{n_{1}}}\searrow \cdots \searrow e_{\overline{n_{p}}}$ 
in $\overline{T}$ and $p$ is maximal.
Recall that when $e_{n_{i+1}}$ is a child of $e_{n_{i}}$ in $T$, we take 
the minimum $n_{i+1}$ satisfying $n_{i}<n_{i+1}$.
This condition can be rephrased in $\overline{T}$ as taking the maximum $\overline{n_{i+1}}$
satisfying $\overline{n_{i}}>\overline{n_{i+1}}$.
We move all the labels except the top-most label upward by one edge in $T$ and $\overline{T}$, 
and we replace the label $n_{p}$ and $\overline{n_p}$ by $n_{0}$ and $\overline{n_{0}}$.
Note that the labels $n_{0}$ and $\overline{n_{0}}$ are on the same edge 
on $\mathrm{Tree}(\lambda)$.
The labels $\overline{n_{0}}$ and $n_{0}$ of this edge are not changed by successive cyclic 
and inverse cyclic operations respectively.
Further, we have $n_{0}+\overline{n_{0}}=n+1$.
By applying the cyclic and inverse cyclic operations on $\overline{T}$ and $T$
respectively, we obtain $S$ and $U$.
The above argument implies that the label $l$ and $\overline{l}$ are always 
on the same edge of $\mathrm{Tree}(\lambda)$ in $T$ and $\overline{T}$.
This means that $U(e)+S(e)=n+1$.
\end{proof}

\begin{prop}
\label{prop:ico}
Let $T, \overline{T}, U$ and $S$ be a label defined as above. 
We have 
\begin{eqnarray}
\label{eqn:TUS}
\begin{CD}
T                  @>\mathrm{bar}>>  \overline{T}  \\
@V\mathrm{ico}VV                      @VV\mathrm{co}V \\
U                  @>\mathrm{bar}>>   S  
\end{CD}
\end{eqnarray}
where $\mathrm{co}$ (resp. $\mathrm{ico}$) stands for 
the cyclic operation (resp. inverse cyclic operation).
\end{prop}
\begin{proof}
From Proposition \ref{prop:US}, we have $U(e)+S(e)=n+1$.
This implies that $S(e)=\overline{U(e)}$, that is, 
$\overline{U}=S$.
Then, we obtain the diagram (\ref{eqn:TUS}).
\end{proof}

\begin{cor}
We have 
\begin{eqnarray*}
T=S \Leftrightarrow \overline{T}=U.
\end{eqnarray*}
\end{cor}
\begin{proof}
From Proposition \ref{prop:ico}, if $T=S$, we have $\overline{T}=U$ 
since $\overline{\overline{T}}=T$.
Reversely, when $\overline{T}=U$, we have $T=S$.
\end{proof}

Recall that we have a description of DTS and DTR bijections in terms 
of insertion algorithm introduced in Section \ref{sec:DTSDTR}.
Similarly, the construction of a cover-inclusive Dyck tiling via
an Hermite history can be realized by an insertion algorithm.

Recall that an edge in $U^{\vee}$ corresponds to a pair of an up step
and a down step in $\mathrm{Tree}(\lambda)$.
To obtain a Dyck tiling over $\lambda$, we put Dyck tiles above $\lambda$
such that the statistics $\mathrm{art}$ for the trajectory of Dyck tiles 
starting from a down step is $\mathrm{label}(e)$ with $e\in U^{\vee}$. 
This defines an Hermite history and denote by $\mathrm{Hh}_{2}(U^{\vee})$
the Dyck tiling obtained by this Hermite history.
\begin{theorem}
\label{thrm:HhrlDTS}
Give a Dyck path $\lambda$, we have
\begin{eqnarray*}
\mathrm{Hh}_{2}(U^{\vee})=\mathrm{rlDTS}(U).
\end{eqnarray*}
\end{theorem}	
\begin{proof}
The label $\mathrm{label}(e)$ with $e\in U^{\vee}$ counts the number of 
edges which are strictly right to $e$ and whose labels are smaller than $\mathrm{label}(e)$.
Since $U$ is a decreasing label and we perform the reversed-order left 
DTS bijection on $U$, $\mathrm{label}(e)$ is counts the number of added 
boxes in the addition process of the DTS. 
By an insertion process of DTS, we may enlarge the size of a Dyck tile 
by one. 
It is obvious that the statistics $\mathrm{art}$ is increased by one.
Thus, the number $\mathrm{label}(e)$ is equal to the statistics $\mathrm{art}$
on the trajectory of the Hermite history in $\mathrm{Hh}_{2}(U^{\vee})$.
This implies $\mathrm{Hh}_{2}(U^{\vee})=\mathrm{rlDTS}(U)$.
\end{proof}

\begin{cor}
Let $U$ be a decreasing tree. Then, we have 
\begin{eqnarray*}
\mathrm{rDTS}(U)=\mathrm{ref}\circ\mathrm{Hh}_{2}\circ(\mathrm{ref}(U))^{\vee}.
\end{eqnarray*}
\end{cor}
\begin{proof}
Since $\mathrm{rDTS}$ is written as $\mathrm{ref}\circ\mathrm{rlDTS}\circ\mathrm{ref}$,
we have 
\begin{eqnarray*}
\mathrm{rDTS}(U)&=&\mathrm{ref}\circ\mathrm{rlDTS}\circ\mathrm{ref}(U), \\
&=&\mathrm{ref}\circ\mathrm{Hh_{2}}\circ(\mathrm{ref}(U))^{\vee},
\end{eqnarray*}
where we have used Theorem \ref{thrm:HhrlDTS} in the second equality.
\end{proof}

It is obvious that the bijections $\mathrm{rDTS}$ and $\mathrm{lDTS}$ are written by 
the DTS bijection, $\mathrm{bar}$ and $\mathrm{ref}$ as follows.
\begin{prop}
Let $T$ and $U$ be increasing and decreasing labels respectively.
We have 
\begin{enumerate}
\item $\mathrm{rDTS}(U)=\mathrm{DTS}\circ\mathrm{bar}(U)$, 
\item $\mathrm{lDTS}(T)=\mathrm{ref}\circ\mathrm{DTS}\circ\mathrm{ref}(T)$.
\end{enumerate}
\end{prop}

Theorem \ref{thrm:HhrlDTS} can be written in terms of the standard DTS bijection
and the inivolution $\alpha$ as follows.
\begin{prop}
Let $T$ be a natural label, $U$ be a decreasing tree obtained from $T$ as above, 
and $S=\alpha(T)$.
We define Dyck tilings $D_{1}$ and $D_{2}$ by
\begin{gather}
T\xrightarrow{\alpha}\xrightarrow{\mathrm{ref}}\xrightarrow{\mathrm{DTS}}D_{1}, \\
U^{\vee}\xrightarrow{\mathrm{Hh}_{2}}\xrightarrow{\mathrm{ref}}D_{2}.
\end{gather}
Then, we have $D_{1}=D_{2}$.
\end{prop}
\begin{proof}
Given a decreasing label $U$, the action of reverse-order left DTS bijection 
is equal to the actions of the bar operation, the reflection, the DTS bijection and 
the reflection:
\begin{eqnarray*}
\mathrm{rlDTS}(U)=\mathrm{ref}\circ\mathrm{DTS}\circ\mathrm{ref}(\overline{U}).
\end{eqnarray*}
Note that, in the right hand side of the above equation, the firs reflection acts 
on the label and the third reflection acts as the reflection of a Dyck tiling.
From Proposition \ref{prop:ico}, we have $\overline{U}=S=\alpha(T)$.
From Theorem \ref{thrm:HhrlDTS}, 
We have 
\begin{eqnarray*}
\mathrm{ref}\circ\mathrm{Hh}_{2}(U^{\vee})&=&\mathrm{ref}\circ\mathrm{rlDTS}(U), \\
&=&\mathrm{DTS}\circ\mathrm{ref}(\overline{U}), \\
&=&\mathrm{DTS}\circ\mathrm{ref}\circ\alpha(T),
\end{eqnarray*}
which implies $D_{1}=D_{2}$.
\end{proof}


The above theorems propositions are summarized as the following diagrams.
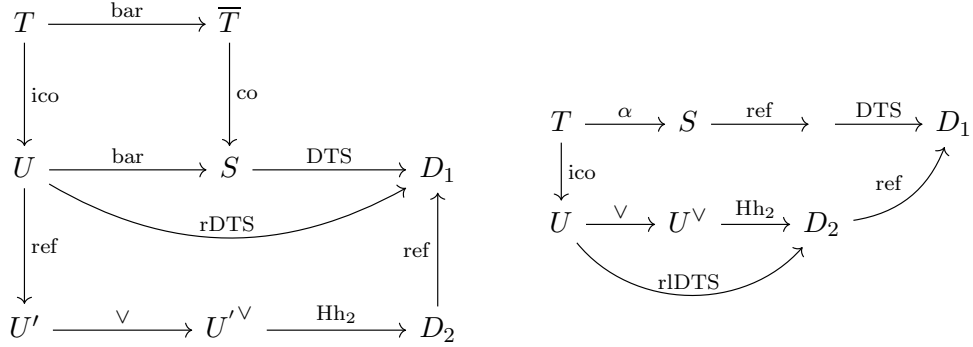
\begin{center}
\begin{tikzcd}[row sep=huge,column sep=huge]
T \arrow[r,"\mathrm{bar}"] \arrow[d,"\mathrm{ico}"] & \overline{T} \arrow[d,"\mathrm{co}"] & \\
U \arrow[r,"\mathrm{bar}"] \arrow[d,"\mathrm{ref}"] \arrow[rr,bend right,"\mathrm{rDTS}"] 
& S \arrow[r,"\mathrm{DTS}"] & D_{1} \\
U' \arrow[r,"\vee"] & {U^{'}}^{\vee} \arrow[r,"\mathrm{Hh}_{2}"] & D_{2} \arrow[u,"\mathrm{ref}"]
\end{tikzcd}
\qquad
\begin{tikzcd}
T \arrow[r,"\alpha"]\arrow[d,"\mathrm{ico}"] & S \arrow[r,"\mathrm{ref}"] 
& \arrow[r,"\mathrm{DTS}"] & D_{1} \\
U \arrow[r,"\vee"] \arrow[rr,"\mathrm{rlDTS}",bend right=50] & U^{\vee} \arrow[r,"\mathrm{Hh}_{2}"] 
& D_{2} \arrow[ru,bend right,"\mathrm{ref}"] 
\end{tikzcd}
\end{center}

\subsection{Involutions on Dyck tableaux and the DTR bijection}
In this subsection, we introduce two operations $\ast$-operation
and $\times$-operation on a natural label as follows.
The $\ast$-operation maps a natural label to a decreasing tree.
On the other hand, $\times$-operation maps a natural label 
to another natural label.
These two operations are dual to each other with respect to the bar operation.
In other words, we have $\times=\mathrm{bar}\circ\ast$ and $\ast=\mathrm{bar}\circ\times$: 
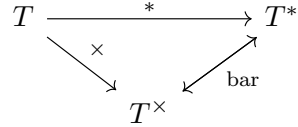
\begin{center}
\begin{tikzcd} 
T \arrow[rr,"\ast"] \arrow[rd,"\times"] & & T^{\ast} \arrow[ld,"\mathrm{bar}"] \\
& T^{\times}\arrow[ru] &  
\end{tikzcd} 
\end{center}

Let $T$ be a natural label and $f_{i}$ be the edge with the label $i$ in $T$.
A {\it decreasing sequence} in $T$ is a set of labels $\mathrm{Dec}(i,j):=[i,j]$ satisfying 
\begin{eqnarray*}
f_{j}\leftarrow f_{j-1}\leftarrow \cdots \leftarrow f_{i+1}\leftarrow f_{i},
\end{eqnarray*}
where $i\le j$ and $j-i$ is maximal.
A decreasing sequence is maximal if 
$f_{j+1}\uparrow f_{j}$ or $f_{j}\rightarrow f_{j+1}$, and 
$f_{i}\uparrow f_{i-1}$ or $f_{i-1}\rightarrow f_{i}$.
A given maximal decreasing sequence $\mathrm{Dec}(i,j)$, 
we define a decreasing sequence $\mathrm{Dec}(p,r)$ which is right to the maximal 
decreasing sequence $\mathrm{Dec}(i,j)$ as follows.
Suppose $q\in\mathrm{Dec}(p,r)$.
We say that $\mathrm{Dec}(p,r)$ is right to $\mathrm{Dec}(i,j)$
if and only if $f_{i}\rightarrow f_{q}$ for all $q\in[p,r]$.
We denote $\mathrm{Dec}(i,j)\rightarrow \mathrm{Dec}(p,r)$ 
when $\mathrm{Dec}(p,r)$ is right to $\mathrm{Dec}(i,j)$. 

We consider a chain of decreasing sequences
\begin{eqnarray}
\label{eqn:chain}
\mathrm{Dec}(i_1,j_1)\rightarrow \mathrm{Dec}(i_2,j_2)\rightarrow
\cdots\rightarrow\mathrm{Dec}(i_{s-1},j_{s-1})\rightarrow\mathrm{Dec}(i_s,j_s)
\end{eqnarray}
with $i_{d+1}=j_{d}+1$ for all $d\in[1,s-1]$.
The chain is maximal if $\mathrm{Dec}(i_{0},j_{0})\not\rightarrow\mathrm{Dec}(i_1,j_1)$ 
and $\mathrm{Dec}(i_{s},j_{s})\not\rightarrow\mathrm{Dec}(i_{s+1},j_{s+1})$
with $j_{0}=i_{1}-1$ and $i_{s+1}=j_{s}+1$.
Given the maximal chain with $i_1=1$, we call $\mathrm{Dec}(i_s,j_s)$ the 
{\it special decreasing sequence}.
By definition, we have a unique special decreasing sequence for a natural 
label.

The set $\mathrm{RDec}(i,j)$ is defined as the union of maximal 
decreasing sequences $\mathrm{Dec}(p,q)$ such that 
$\mathrm{Dec}(i,j)\rightarrow \mathrm{Dec}(p,q)$.
Similarly, we define the set $\mathrm{LDec}(i,j)$ as the union of
maximal decreasing sequences $\mathrm{Dec}(p,q)$ such that 
$\mathrm{Dec}(p,q)\rightarrow\mathrm{Dec}(i,j)$.

Let $\mathrm{Dec}(i_s,j_s)$ be the special decreasing sequence.
Then, $\mathrm{Dec}(i_s,j_s)$ is the right-most sequence of the chain
starting from $\mathrm{Dec}(1,j_1)$, but note that $\mathrm{RDec}(i_s,j_s)$ 
may not be empty.

We define the set $\mathrm{ChildDec}(i,j)$ as follows.
Let $k\in[i,j]$ and $q\in[p,r]$.
The set $\mathrm{Dec}(p,r)\in\mathrm{ChildDec}(i,j)$ 
if $e_{q}\uparrow e_{k}$ for some $k$ and $q$.
Note that we have 
$\mathrm{RDec}(i,j)\cap\mathrm{ChildDec}(i,j)=\mathrm{LDec}(i,j)\cap\mathrm{ChildDec}(i,j)=\emptyset$.

We define the standardization of partial tree in $T$. 
Take a set of labeled edges of $T$ and let $n$ be the number of edges.
By standardization, we replace the labels by $1,2,\ldots, n$ according 
to their order.
By de-standardization with respect to the set $S$ with $|S|=n$, 
we replace the labels of a natural label by the elements of $S$ 
according to their order.

Let $\lambda$ be a Dyck path and $T$ be a natural label of the tree $\mathrm{Tree}(\lambda)$.
The tree $\mathrm{Tree}(\lambda)$ is decomposed into a concatenation of trees at their roots.
We write $\mathrm{Tree}(\lambda):=S_{1}\circ\cdots\circ S_{r}$ where a tree $S_{i}$ for 
$1\le i\le r$ is a tree. Here, a tree $S_{i}$ can not be decomposed into a concatenation 
of trees of smaller size.
This means that a tree $S_{i}$ has exactly one edge connected to its root.

Since the natural label $T$ has also a tree structure, we will decompose $T$ into 
a concatenation of natural labels $T_{i}$ for $1\le i\le p$ with the following condition.
Let $|T_{i}|$ be the number of edges in $T_{i}$ and $\max(T_{i})$ (resp. $\min(T_{i})$)
be the largest (resp. smallest) label in $T_{i}$.
As a tree structure, a tree $T_{i}$ is given by a concatenation of trees 
$S_{a}\circ S_{a+1}\circ\cdots\circ S_{b}$ for some $a$ and $b$.
We consider the following conditions:
\begin{eqnarray*}
\min(T_{i})&=&\sum_{k=1}^{i-1}|T_{k}|+1, \\
\max(T_{i})&=&\sum_{k=1}^{i}|T_{k}|.
\end{eqnarray*}
We say that $T_{i}$ is a {\it minimal natural label} if $T_{i}$ satisfies the 
above conditions and $b-a$ is minimal.
When all $T_{i}$'s are minimal natural labels, 
we denote  
\begin{eqnarray*}
T:=T_{1}\circ T_{2}\circ\cdots\circ T_{p}. 
\end{eqnarray*}

We define the action of $\ast$-operation on a natural label $T$ by
\begin{eqnarray*}
T^{\ast}:=T_{1}^{\ast}\circ T_{2}^{\ast}\circ\cdots\circ T_{p}^{\ast}.
\end{eqnarray*}
Here, the action of $\ast$-operation on $T_{i}^{\ast}$ is given by 
the following two steps.
\begin{enumerate}
\item Standardize the $T_{i}$ and act the $\ast$-operation on the 
standardized natural label. Let $T_{i}'$ be the new label after the 
action of the $\ast$-operation.
\item De-standardize $T_{i}'$ with respect to $[\min(T_{i}),\max(T_{i})]$.
\end{enumerate}
Note that standaredization and de-standardization of $T$ are well-defined since 
the number of edges of $T$ is equal to $\max(T)-\min(T)+1$.

We define the action of $\times$-operation on $T$ in the similar way:
\begin{eqnarray*}
T^{\times}&:=&T_{1}^{\times}\circ T_{2}^{\times}\circ\cdots\circ T_{p}^{\times}, \\
&=&\overline{T_{1}^{\ast}}\circ\overline{T_{2}^{\ast}}\circ\cdots\circ \overline{T_{p}^{\ast}}.
\end{eqnarray*}

Below, we introduce the actions of the $\ast$-operation and the $\times$-operation
on a standardized minimal natural label.

\paragraph{\bf Algorithm for the $\ast$-operation on a standardized minimal natural label}
We recursively define the algorithm for the $\ast$-operation on a minimal natural label.
Let $n$ be the number of edges of the minimal natural label.
\begin{enumerate}
\item Find the special decreasing sequence $\mathrm{Dec}(p,q)$ 
and the set $\mathrm{RDec}(p,q)$. 
We define 
$M=|\mathrm{RDec}(p,q)|$.

\item We replace the label of $f_{r}$ by $n-M-q+r$ for all $r\in[p,q]$.

\item Let $T_{r}$ be the partial tree consisting of edges in $\mathrm{RDec}(p,q)$, 
and $T_{l}$ be the partial tree consisting of edges which are neither in 
$\mathrm{Dec}(p,q)$ nor $\mathrm{RDec}(p,q)$.
\begin{enumerate}
\item We standardize the partial tree $T_{l}$.
We apply the $\ast$-operation to this standard partial tree $T_{l}$.
We de-standardize $T_{l}$ with respect to $[1,n-M-q+p-1]$.

\item We standardize the partial tree $T_{r}$.
we apply the $\ast$-operation to this partial tree.
Then, we de-standardize $T_{r}$ with respect to $[n-M+1,n]$. 
\end{enumerate}
\end{enumerate}

\begin{remark}
In (3) of the algorithm for the $\ast$-operation, we consider the set which are neither 
$\mathrm{Dec}(p,q)$ nor $\mathrm{RDec}(p,q)$.
Note that this set is not $\mathrm{LDec}(p,q)\cup\mathrm{ChildDec}(p,q)$ in general.
\end{remark}

\paragraph{\bf Algorithm for the $\times$-operation on a standardized minimal natural label}
By definition, $\times$-operation is written as a composition of 
the $\ast$-operation and the bar operation. 
Thus, the algorithm for the $\times$-operation is mostly similar 
to the one for the $\ast$-operation.
The first step is the same as the $\ast$-operation.
We replace the second and third steps in the $\ast$-operation by the following:
\begin{enumerate}
\item[(2')] replace $n-M-q+r$ by $M+q-r+1$,
\item[(3a')] replace $[1,n-M-q+p-1]$ by $[M+q-p+2,n]$,
\item[(3b')] replace $[n-M+1,n]$ by $[1,M]$.
\end{enumerate}

An example of the actions of these operations is shown in Figure \ref{fig:asttimes}.

\begin{figure}[ht]
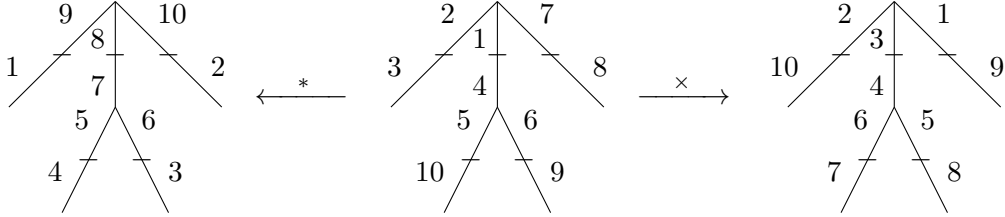

\tikzpic{-0.5}{[level distance=0.7cm, sibling distance=0.7cm]
\coordinate
	child{coordinate (c2)
        child{coordinate(c3)}
	child[missing]
	child[missing]
        }
	child{coordinate(c1)
	child{coordinate(c4)
	        child{coordinate(c5)
		child{coordinate(c10)}
		child[missing]
		}
		child{coordinate(c6)
		child[missing]
		child{coordinate(c9)}
		}
	}
	}
	child{coordinate(c7)
	child[missing]
	child[missing]
	child{coordinate(c8)}
	};
\draw(c2)node{$-$}(c1)node{$-$}(c5)node{$-$}(c6)node{$-$}(c7)node{$-$};
\node[anchor=south east] at ($(0,0)!.6!(c2)$){9};
\node[anchor=south east] at ($(c2)!.6!(c3)$){1};
\node[anchor=east] at ($(0,0)!.7!(c1)$){8};
\node[anchor=east] at ($(c1)!.6!(c4)$){7};
\node[anchor=south east] at ($(c4)!.6!(c5)$){5};
\node[anchor=south east] at ($(c5)!.6!(c10)$){4};
\node[anchor=south west] at ($(c4)!.6!(c6)$){6};
\node[anchor=south west] at ($(c6)!.6!(c9)$){3};
\node[anchor=south west] at ($(0,0)!.6!(c7)$){10};
\node[anchor=south west] at ($(c7)!.6!(c8)$){2};
}
$\xleftarrow[\hspace{1cm}]{\ast}$
\tikzpic{-0.5}{[level distance=0.7cm, sibling distance=0.7cm]
\coordinate
	child{coordinate (c2)
        child{coordinate(c3)}
	child[missing]
	child[missing]
        }
	child{coordinate(c1)
	child{coordinate(c4)
	        child{coordinate(c5)
		child{coordinate(c10)}
		child[missing]
		}
		child{coordinate(c6)
		child[missing]
		child{coordinate(c9)}
		}
	}
	}
	child{coordinate(c7)
	child[missing]
	child[missing]
	child{coordinate(c8)}
	};
\draw(c2)node{$-$}(c1)node{$-$}(c5)node{$-$}(c6)node{$-$}(c7)node{$-$};
\node[anchor=south east] at ($(0,0)!.6!(c2)$){2};
\node[anchor=south east] at ($(c2)!.6!(c3)$){3};
\node[anchor=east] at ($(0,0)!.7!(c1)$){1};
\node[anchor=east] at ($(c1)!.6!(c4)$){4};
\node[anchor=south east] at ($(c4)!.6!(c5)$){5};
\node[anchor=south east] at ($(c5)!.6!(c10)$){10};
\node[anchor=south west] at ($(c4)!.6!(c6)$){6};
\node[anchor=south west] at ($(c6)!.6!(c9)$){9};
\node[anchor=south west] at ($(0,0)!.6!(c7)$){7};
\node[anchor=south west] at ($(c7)!.6!(c8)$){8};
}
$\xrightarrow[\hspace{1cm}]{\times}$
\tikzpic{-0.5}{[level distance=0.7cm, sibling distance=0.7cm]
\coordinate
	child{coordinate (c2)
        child{coordinate(c3)}
	child[missing]
	child[missing]
        }
	child{coordinate(c1)
	child{coordinate(c4)
	        child{coordinate(c5)
		child{coordinate(c10)}
		child[missing]
		}
		child{coordinate(c6)
		child[missing]
		child{coordinate(c9)}
		}
	}
	}
	child{coordinate(c7)
	child[missing]
	child[missing]
	child{coordinate(c8)}
	};
\draw(c2)node{$-$}(c1)node{$-$}(c5)node{$-$}(c6)node{$-$}(c7)node{$-$};
\node[anchor=south east] at ($(0,0)!.6!(c2)$){2};
\node[anchor=south east] at ($(c2)!.6!(c3)$){10};
\node[anchor=east] at ($(0,0)!.7!(c1)$){3};
\node[anchor=east] at ($(c1)!.6!(c4)$){4};
\node[anchor=south east] at ($(c4)!.6!(c5)$){6};
\node[anchor=south east] at ($(c5)!.6!(c10)$){7};
\node[anchor=south west] at ($(c4)!.6!(c6)$){5};
\node[anchor=south west] at ($(c6)!.6!(c9)$){8};
\node[anchor=south west] at ($(0,0)!.6!(c7)$){1};
\node[anchor=south west] at ($(c7)!.6!(c8)$){9};
}
\caption{The actions of the $\ast$-operation and $\times$-operation on a natural 
label.}
\label{fig:asttimes}
\end{figure}

\begin{theorem}
\label{thrm:timesDTR}
We have the following diagram:
\begin{eqnarray}
\begin{tikzcd}[row sep=small]
T \arrow[r,"\mathrm{DTR}"] \arrow[dd,"\times"]  & D_{1} \arrow[rd,"\mathrm{ref}"] & \\
& & D_{2} \\
T^{\times} \arrow[r,"\mathrm{ref}"] & ^{\times}T \arrow[ru,"\mathrm{DTR}"] & 
\end{tikzcd}
\label{cd:timesref}
\end{eqnarray}
\label{thrm:ref}
\end{theorem}
\begin{proof}
We denote by $f^{\times}_{p}$ the edge labeled by $p$ in $T^{\times}$.
Then, a maximal increasing sequence is a set of labels $\mathrm{Inc}(i',j'):=[i',j']$ satisfying 
\begin{eqnarray*}
f^{\times}_{i'}\rightarrow f^{\times}_{i'+1}\rightarrow\cdots\rightarrow
f^{\times}_{j'}
\end{eqnarray*}
with $j'-i'$ is maximal.
By the definition of the $\times$-operation, 
a maximal decreasing sequence $\mathrm{Dec}(i,j)$ in $T$ 
is changed to a maximal increasing sequence $\mathrm{Inc}(i',j')$ in $T^{\times}$
with some $i'$ and $j'$ satisfying $j'-i'=j-i$.

When $\mathrm{Dec}(i,j)\rightarrow\mathrm{Dec}(k,l)$ in $T$ with $k=j+1$,   
we have $\mathrm{Inc}(i',j')\rightarrow\mathrm{Inc}(k',l')$ in $T^{\times}$.
Since we change the labels in $\mathrm{Dec}(k,l)$ before $\mathrm{Dec}(i,j)$,
we have $l'<i'$. Furthermore, the condition $k=j+1$ is translated to 
the condition $i'=l'+1$.

By the third step for the $\times$-operation, it is obvious that 
the new labels in $T^{\times}$ are increasing from the root to leaves.

The reflection of a natural tree reverses the order of labels
in $\mathrm{Inc}(i',j')$,  
that is, $\mathrm{Inc}(i',j')$ becomes a maximal decreasing 
sequence $\mathrm{Dec}(i',j')$ in $\mathrm{ref}(T^{\times})$.
Once given positions of maximal increasing sequence, we have 
a unique natural label $T^{\times}$.

Suppose we have $\mathrm{Dec}(i,j)$ in $T$. In the DTR bijection,
we have a ribbon between the box labeled by $k$ and the box labeled 
by $k+1$ if $k\in[i,j-1]$.
Once we fix the positions of maximal decreasing sequences and 
ribbons in the DTR bijection, we have a unique Dyck tiling via 
the DTR bijection.
This argument can also be applied to $\mathrm{ref}(T^{\times})$.
Therefore, we obtain the diagram (\ref{cd:timesref}).
\end{proof}

\begin{cor}
The composition of the $\times$-operation and the reflection on a natural label 
is an involution. 
We have $(\mathrm{ref}\circ\times)^{2}=\mathrm{id}$ where $\mathrm{id}$ is the 
identity.
\end{cor}
\begin{proof}
From Theorem \ref{thrm:ref}, we have 
\begin{eqnarray}
\label{eqn:refDTR}
\mathrm{ref}=\mathrm{DTR}\circ\mathrm{ref}\circ\times\circ\mathrm{DTR}^{-1},
\end{eqnarray}
where the $\mathrm{ref}$ in the left hand side acts on a Dyck tiling, and 
the $\mathrm{ref}$ in the right hand side acts on a natural label.
Since the refection on a Dyck tiling is an involution, 
we have $\mathrm{ref}^{2}=\mathrm{id}$.
From Eqn. (\ref{eqn:refDTR}), we have $(\mathrm{ref}\circ\times)^{2}=\mathrm{id}$.
\end{proof}

Let $\mathbf{m}=(m_{1},\ldots,m_{r})$ and $\lambda:=\wedge_{\mathbf{m}}$.
We define an involution $\heartsuit$ on Dyck tilings over $\lambda$.
Since $\lambda=\wedge_{\mathbf{m}}$, maximal decreasing sequences 
$\mathrm{Dec}(i_1,j_1),\ldots,\mathrm{Dec}(i_r,j_r)$ have 
a natural poset structure with respect to the order of the $(i_p,j_p)$'s with 
$1\le p\le r$.
More precisely, we write $\mathrm{Dec}(i,j)\prec\mathrm{Dec}(k,l)$ if a box 
in a Dyck tableau forming $\mathrm{Dec}(k,l)$ is above a box forming 
$\mathrm{Dec}(i,j)$.
Otherwise, two decreasing sequences are not comparable. 
Let $c(i_p)$ and $c(j_p)$ be the columns where the boxes labeled with $i_p$ 
and $j_p$ are placed. 
Let $\heartsuit$-operation be the operation such that 
it reverses the partial order of maximal decreasing sequences.
Further, if $\mathrm{Dec}(i_p,j_p)$ is mapped to $\mathrm{Dec}(i'_p,j'_p)$ by 
$\heartsuit$-operation, we have $c(i_p)=c(i'_{p})$ and $c(j_p)=c(j'_{p})$.
Figure \ref{fig:heart} is an example of the action of the $\heartsuit$-operation
on a Dyck tableau.
\begin{figure}[ht]
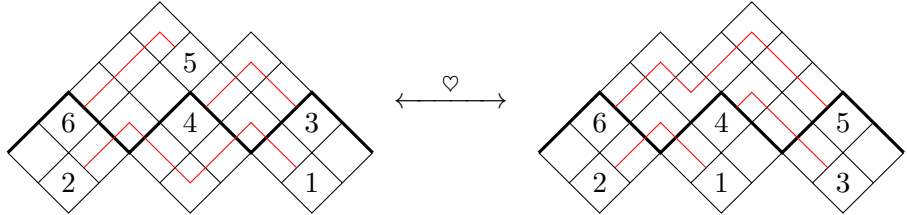

\tikzpic{-0.5}{[x=0.4cm,y=0.4cm]
\draw[very thick](0,0)--(2,2)--(4,0)--(6,2)--(8,0)--(10,2)--(12,0);
\draw(2,2)--(5,5)--(11,-1)(1,-1)--(6,4)(0,0)--(2,-2)--(4,0)
     (4,0)--(6,-2)--(8,0)--(10,-2)--(12,0)(1,1)--(3,-1)
     (3,3)--(7,-1)(4,4)--(6,2)(5,-1)--(9,3)(9,-1)--(11,1)
     (6,2)--(8,4)--(10,2);
\draw(2,-1)node{$2$}(2,1)node{$6$}(6,1)node{$4$}(6,3)node{$5$}(10,-1)node{$1$}(10,1)node{$3$};
\draw[red](2.5,1.5)--(5,4)--(5.5,3.5)(2.5,-0.5)--(4,1)--(6,-1)--(8,1)--(9.5,-0.5)
          (6.5,1.5)--(8,3)--(9.5,1.5);
}
$\overset{\heartsuit}{\overleftrightarrow{\hspace{1.5cm}}}$
\tikzpic{-0.5}{[x=0.4cm,y=0.4cm]
\draw[very thick](0,0)--(2,2)--(4,0)--(6,2)--(8,0)--(10,2)--(12,0);
\draw(2,2)--(4,4)--(6,2)--(8,4)--(10,2)(0,0)--(2,-2)--(4,0)--(6,-2)--(8,0)--(10,-2)--(12,0)
     (1,1)--(3,-1)(5,1)--(7,-1)(9,1)--(11,-1)
     (1,-1)--(7,5)--(8,4)(3,3)--(5,1)(6,4)--(9,1)(5,-1)--(9,3)(9,-1)--(11,1);
\draw(2,-1)node{$2$}(2,1)node{$6$}(6,-1)node{$1$}(6,1)node{$4$}(10,-1)node{$3$}(10,1)node{$5$};
\draw[red](2.5,1.5)--(4,3)--(5,2)--(7,4)--(9.5,1.5)(2.5,-0.5)--(4,1)--(5.5,-0.5)
          (6.5,1.5)--(7,2)--(9.5,-0.5);
}
\caption{An action of $\heartsuit$-operation on a Dyck tableau. In the left picture,
we have the poset $\mathrm{Dec}(1,2)\prec\mathrm{Dec}(3,4)\prec\mathrm{Dec}(5,6)$.}
\label{fig:heart}
\end{figure}

\begin{remark}
The involution $\heartsuit$ does not preserve the shape of a Dyck tiling in general.
\end{remark}

\begin{prop}
\label{prop:wedgeheart}
Let $\lambda:=\wedge_{\mathbf{m}}$ and $T$ be a natural label of $\mathrm{Tree}(\lambda)$. 
Then, the actions of the $\ast$-operation and the cyclic operation
give 
\begin{eqnarray}
\label{eqn:DTabheart}
\mathrm{DTab}(T)^{\heartsuit}=\mathrm{DTab}(\mathrm{co}(T^{\ast})).
\end{eqnarray} 
\end{prop}
\begin{proof}
By definition of the $\ast$-operation, a maximal decreasing sequence $\mathrm{Dec}(i,j)$ is 
mapped to a maximal decreasing sequence $\mathrm{Dec}(i',j')$ with $j-i=j'-i'$.
Further, we have $c(i)=c(i')$ and $c(j)=c(j')$.

When a chain Eqn. (\ref{eqn:chain}) exists, the order of maximal decreasing sequences  
in the chain is preserved.
Otherwise, we reverse the order of two maximal decreasing sequences.
The $\ast$-operation produces a decreasing tree from $T$.
By the cyclic operation, we reverse the order of labels in $\wedge_{m_i}$, $1\le i\le r$, 
in the decreasing tree.
This reverses the order of maximal decreasing sequences in the chain.
The cyclic operation does not reverse the order of two maximal decreasing sequences if 
they do not belong to the same chain.
From these observations, we have Eqn. (\ref{eqn:DTabheart}).
\end{proof}

We have a bijection between a Dyck tableau and a DTR bijection
of $T(\lambda)$.
We denote by $\mathrm{DTab}(T(\lambda))$ a Dyck tableau associated 
with a label $T(\lambda)$.
Below, we will define an involution on a Dyck tableau for a zigzag path $\lambda$:
\begin{eqnarray*}
\clubsuit: \mathrm{DTab}(T(\lambda))\rightarrow \mathrm{DTab}(T(\lambda)).
\end{eqnarray*}
Let $\mu$ be the top path of $\mathrm{DTab}(T(\lambda))$.
Suppose a dot $d$ in $\mathrm{DTab}(T(\lambda))$ is in the $i$-th floor, 
and the $p$-th floor is the highest floor below the path $\mu$.
Then, an involution $\clubsuit$-operation preserves its shape, {\it i.e.}, 
the top path $\mu$ is not changed by the action of the $\clubsuit$-operation, 
and the dot $d$ is moved from the $i$-th floor to the $(p-i)$-th floor.
Figure \ref{fig:club} is an example of the action of $\clubsuit$-operation.
\begin{figure}[ht]
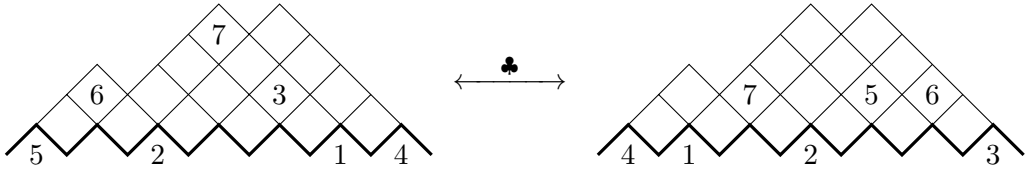

\tikzpic{-0.5}{[x=0.4cm,y=0.4cm]
\draw[very thick](0,0)--(1,1)--(2,0)--(3,1)--(4,0)--(5,1)--(6,0)--(7,1)
                  --(8,0)--(9,1)--(10,0)--(11,1)--(12,0)--(13,1)--(14,0);  
\draw(1,1)--(3,3)--(5,1)--(9,5)--(13,1)(2,2)--(3,1)--(7,5)--(11,1)--(12,2);
\draw(5,3)--(7,1)--(10,4)(6,4)--(9,1)--(11,3);
\draw(1,0)node{5}(3,2)node{6}(5,0)node{2}(7,4)node{7}(9,2)node{3}(11,0)node{1}(13,0)node{4};
}
$\overset{\clubsuit}{\overleftrightarrow{\hspace{1.5cm}}}$
\tikzpic{-0.5}{[x=0.4cm,y=0.4cm]
\draw[very thick](0,0)--(1,1)--(2,0)--(3,1)--(4,0)--(5,1)--(6,0)--(7,1)
                  --(8,0)--(9,1)--(10,0)--(11,1)--(12,0)--(13,1)--(14,0);  
\draw(1,1)--(3,3)--(5,1)--(9,5)--(13,1)(2,2)--(3,1)--(7,5)--(11,1)--(12,2);
\draw(5,3)--(7,1)--(10,4)(6,4)--(9,1)--(11,3);
\draw(1,0)node{4}(3,0)node{1}(5,2)node{7}(7,0)node{2}(9,2)node{5}(11,2)node{6}(13,0)node{3};
}
\caption{An action of the $\clubsuit$-operation on a Dyck tableau.}
\label{fig:club}
\end{figure}

Given a Dyck tableau $D$, we denote by $\mathrm{shadow}(D)$ the number of shadow 
boxes in $D$ and by $\mathrm{clear}(D)$ the number of clear boxes.
\begin{prop}
Let $D$ be a Dyck tableau for a permutation $\sigma$.
Then, we have 
\begin{eqnarray*}
\mathrm{shadow}(D^{\clubsuit})&=&\mathrm{clear}(D), \\
\mathrm{clear}(D^{\clubsuit})&=&\mathrm{shadow}(D).
\end{eqnarray*}
\end{prop}
\begin{proof}
By definition of the $\clubsuit$-operation, the number of boxes above a dotted 
box in $D$ is equal to the number of boxes below the dotted box in $D^{\clubsuit}$.
This implies $\mathrm{clear}(D^{\clubsuit})=\mathrm{shadow}(D)$.

By a similar argument, we have $\mathrm{shadow}(D^{\clubsuit})=\mathrm{clear}(D)$.
\end{proof}

\begin{theorem}
\label{thrm:astclub}
The action of $\ast$-operation on a zigzag path is equivalent 
to the action of $\clubsuit$. 
We have 
\begin{eqnarray}
\label{eqn:DTRclub}
\mathrm{DTab}(T)^{\clubsuit}=\mathrm{DTab}(T^{\ast}).
\end{eqnarray}
\end{theorem}
\begin{proof}
A zigzag path is a path $\wedge_{\mathbf{m}}$ with $\mathbf{m}=(1,\ldots,1)$.
One can apply Proposition \ref{prop:wedgeheart} to a natural label $T$.
Note that we have $\mathrm{co}(T^{\ast})=T^{\ast}$ in case of $\wedge_{\mathbf{m}}$
with $\mathbf{m}=(1,\ldots,1)$.

To show Eqn. (\ref{eqn:DTRclub}), it is enough to show that the shapes of Dyck tableaux
in Eqn. (\ref{eqn:DTRclub}) are the same.
When $\lambda$ is a zigzag path, we have no empty boxes in a Dyck tableau.
Since we reverse the order of maximal decreasing sequences, the shadow (resp. clear) boxes 
of a labeled box in $\mathrm{DTab}(T)$ are bijective to the clear (shadow) boxes of the 
labeled box in $\mathrm{DTab}(T^{\ast})$.
The number of boxes in a column where a labeled box is placed is the sum of 
the numbers of shadow and clear boxes plus one.
This is invariant under the action of $\ast$-operation. This implies that 
the shape of $\mathrm{DTab}(T)^{\clubsuit}$ is the same as that of $\mathrm{DTab}(T^{\ast})$.
\end{proof}

\begin{prop}
The $\clubsuit$-operation on a zigzag path is equivalent to the reverse order left DTR bijection.
When $T$ is a natural label for a zigzag path, we have
\begin{eqnarray}
\label{rel:DtabrlDTR}
\mathrm{DTab}(T)^{\clubsuit}=\mathrm{rlDTR}(T).
\end{eqnarray}
\end{prop}
\begin{proof}
Let $\mathrm{Dec}(i,j)$ be a maximal decreasing sequence.
By the standard DTR bijection, we connect the boxes with 
labels in $[i,j]$ by ribbons from left to right.
On the other hand, we connect the boxes with labels $[i,j]$ 
by ribbons from right to left by the reverse order left DTR bijection.
Suppose $k$ is in-between $i$ and $j$ in $T$.
Suppose $k>j$. 
Then, $k$ is above (resp. below) the ribbons associated with $\mathrm{Dec}(i,j)$ 
in case of the standard (resp. reverse order left) DTR bijection.
In case of $k<j$, we have a similar statement.
From these observations, we have 
\begin{eqnarray*}
\mathrm{shadow}(\mathrm{rlDTR}(T))&=&\mathrm{clear}(D), \\
\mathrm{clear}(\mathrm{rlDTR}(T))&=&\mathrm{shadow}(D).
\end{eqnarray*}
To show Eqn. (\ref{rel:DtabrlDTR}), it is enough to show that 
the top path of $\mathrm{DTab}(T)$ is the same as the top path 
of $\mathrm{rlDTR}(T)$.
Since we consider a zigzag path, we have no empty boxes 
in the Dyck tableaux.
The sum of the numbers of shadow and clear boxes in a column is preserved 
by the $\mathrm{rlDTR}$ bijection.
This implies that the top path is invariant under the action of 
the $\mathrm{rlDTR}$ bijection.
\end{proof}

\begin{theorem}
\label{thrm:alphaDTab}
We have the following diagram for a permutation $\sigma$:
\begin{center}
\begin{tikzcd}[row sep=normal,column sep=normal]
\sigma \arrow[r,"\mathrm{bar}"] \arrow[d,"\mathrm{DTab}"] 
& S \arrow[r,"\mathrm{ref}"] 
& S'\arrow[d,"\mathrm{DTab}"] 
\\
D_{1} \arrow[r,"\clubsuit"] 
& D_{1}' \arrow[r,"\mathrm{ref}"] 
& D_{2}
\end{tikzcd}
\end{center}
\end{theorem}
\begin{proof}
Since $\sigma$ is a permutation, the action of $\ast$-operation is equivalent to 
the $\clubsuit$-operation from Theorem \ref{thrm:astclub}.
The reflection of a Dyck tiling is equivalent to a composition $\mathrm{ref}\circ\times$
on $\sigma$ from Theorem \ref{thrm:timesDTR}.
Thus, the composition of $\mathrm{ref}\circ\clubsuit$ is written as 
\begin{eqnarray*}
\mathrm{DTab}^{-1}\circ\mathrm{ref}\circ\clubsuit\circ\mathrm{DTab}
&=&\mathrm{ref}\circ\times\circ\ast, \\
&=&\mathrm{ref}\circ\mathrm{bar}\circ\ast\circ\ast, \\
&=&\mathrm{ref}\circ\mathrm{bar},
\end{eqnarray*}
where we have used $\ast^{2}=\mathrm{id}$ on a permutation.
\end{proof}

\begin{cor}
Let $T$ be a natural label associated with a zigzag path.
We have the following diagram:
\begin{center}
\begin{tikzcd}
T \arrow[r,"\mathrm{bar}"]\arrow[d,"\mathrm{ref}"] & S\arrow[r,"\mathrm{DTab}"] & D_{1} \\
T' \arrow[r,"\mathrm{DTab}"] & D_{2}\arrow[r,"\mathrm{ref}"] & D_{2}' \arrow[u,"\clubsuit"]
\end{tikzcd}
\end{center}
\end{cor}
\begin{proof}
The involutions $\mathrm{bar}$ and $\mathrm{ref}$ commute with each other, and the involutions
$\mathrm{ref}$ and $\ast$ commute with each other.
From Theorem \ref{thrm:alphaDTab}, we obtain the diagram.
\end{proof}

\subsection{The maps from a Dyck tableau to the extreme Dyck tableaux for a zigzag path}
Given a Dyck tableau $D$ for a zigzag path, we have two extreme Dyck tableaux of the same 
shape as $D$.
The first one is a unique Dyck tableau with all dots in the highest floors, and the second one 
is with all dots in the lowest floors. 
We denote by $D_{\vee}$ (resp. $D_{\wedge}$) the first (resp. second) extreme Dyck tableau. 
Figure \ref{fig:exDT} is examples of extreme Dyck tableaux.
\begin{figure}[ht]
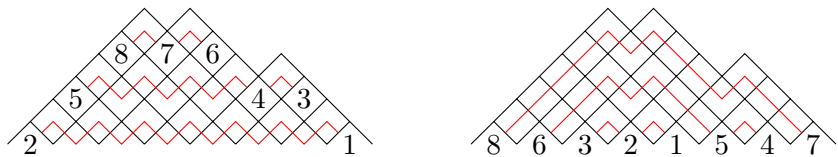

\tikzpic{-0.5}{[x=0.3cm,y=0.3cm]
\draw(0,0)--(6,6)--(12,0)--(14,2)--(16,0)(2,0)--(8,6)--(14,0)--(15,1)
     (4,0)--(9,5)(6,0)--(10,4)(8,0)--(12,4)--(14,2)(10,0)--(13,3)
     (1,1)--(2,0)(2,2)--(4,0)(3,3)--(6,0)(4,4)--(8,0)(5,5)--(10,0);
\draw(1,0)node{2}(3,2)node{5}(5,4)node{8}(7,4)node{7}(9,4)node{6}(11,2)node{4}
     (13,2)node{3}(15,0)node{1};
\draw[red](1.5,0.5)--(2,1)--(3,0)--(4,1)--(5,0)--(6,1)--(7,0)--(8,1)--
          (9,0)--(10,1)--(11,0)--(12,1)--(13,0)--(14,1)--(14.5,0.5)
          (3.5,2.5)--(4,3)--(5,2)--(6,3)--(7,2)--(8,3)--(9,2)--(10,3)--(10.5,2.5)
          (11.5,2.5)--(12,3)--(12.5,2.5)
          (5.5,4.5)--(6,5)--(6.5,4.5)(7.5,4.5)--(8,5)--(8.5,4.5);
}
\qquad
\tikzpic{-0.5}{[x=0.3cm,y=0.3cm]
\draw(0,0)--(6,6)--(12,0)--(14,2)--(16,0)(2,0)--(8,6)--(14,0)--(15,1)
     (4,0)--(9,5)(6,0)--(10,4)(8,0)--(12,4)--(14,2)(10,0)--(13,3)
     (1,1)--(2,0)(2,2)--(4,0)(3,3)--(6,0)(4,4)--(8,0)(5,5)--(10,0);
\draw(1,0)node{8}(3,0)node{6}(5,0)node{3}(7,0)node{2}(9,0)node{1}
     (11,0)node{5}(13,0)node{4}(15,0)node{7};
\draw[red](1.5,0.5)--(6,5)--(7,4)--(8,5)--(11,2)--(12,3)--(14.5,0.5)
          (3.5,0.5)--(6,3)--(7,2)--(8,3)--(10.5,0.5)(11.5,0.5)--(12,1)--(12.5,0.5)
          (5.5,0.5)--(6,1)--(6.5,0.5)(7.5,0.5)--(8,1)--(8.5,0.5);
}
\caption{Examples of extreme Dyck tableaux: $D_{\vee}$ (left picture) and 
        $D_{\wedge}$ (right picture). }
\label{fig:exDT}
\end{figure}

Let $\sigma$ be a permutation of length $n$.
When we have a maximal decreasing sequences $\mathrm{Dec}(i,j)$ in $\sigma$, 
we connect the positions $\sigma^{-1}(j), \sigma^{-1}(j-1),\ldots,\sigma^{-1}(i)$ 
by an arch.
We define the size of an arch corresponding to $\mathrm{Dec}(i,j)$ as $j-i+1$.
When $j=i$, we have an arch of size one. 
We do not depict anything for an arch of size one. 
The {\it skeleton} of $\sigma$ is a collection of arches associated with maximal 
decreasing sequences.
For example, when $\sigma=64835271$, the maximal decreasing sequences
are $\mathrm{Dec}(1,4)$, $\mathrm{Dec}(5,6)$ and $\mathrm{Dec}(7,8)$.
The skeleton of $\sigma$ is depicted as 
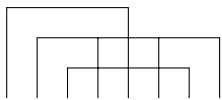
\begin{eqnarray*}
\tikzpic{-0.5}{[x=0.4cm,y=0.4cm]
\draw(3,0)--(3,1)--(7,1)--(7,0)
     (2,0)--(2,2)--(8,2)--(8,0)(4,0)--(4,2)(6,0)--(6,2)
     (1,0)--(1,3)--(5,3)--(5,0);
}
\end{eqnarray*}

Since a skeleton consists of arches and arches may cross, 
a skeleton has trivalent vertices and tetravalent vertices.
A tetravalent vertex corresponds to an intersection of two arches.
Given two maximal decreasing sequences $\mathrm{Dec}(i_1,j_1)$ and $\mathrm{Dec}(i_2,j_2)$,
we have three types of configurations as follows.
\begin{enumerate}
\item $\mathrm{Dec}(i_1,j_1)$ is strictly right to $\mathrm{Dec}(i_2,j_2)$, which means 
$\mathrm{Dec}(i_2,j_2)\rightarrow\mathrm{Dec}(i_1,j_1)$, or equivalently 
$\sigma^{-1}(j_2)<\sigma^{-1}(i_2)<\sigma^{-1}(j_1)<\sigma^{-1}(i_1)$. 
\item $\mathrm{Dec}(i_1,j_1)$ is weakly right to $\mathrm{Dec}(i_2,j_2)$, which means
$\sigma^{-1}(j_2)<\sigma^{-1}(j_1)<\sigma^{-1}(i_2)<\sigma^{-1}(i_1)$.
\item $\mathrm{Dec}(i_1,j_1)$ is inclusive to $\mathrm{Dec}(i_2,j_2)$, which means 
$\sigma^{-1}(j_2)<\sigma^{-1}(j_1)<\sigma^{-1}(i_1)<\sigma^{-1}(i_2)$.	
\end{enumerate}
When $\mathrm{Dec}(i_1,j_1)$ is strictly right to $\mathrm{Dec}(i_2,j_2)$, there is no
intersection between these two arches.
When $\mathrm{Dec}(i_1,j_1)$ is weakly right to $\mathrm{Dec}(i_2,j_2)$,
there is at least one intersection between these two arches.

In the picture of a skeleton, we depict an arch associated with $\mathrm{Dec}(i_1,j_1)$
inside of an arch associated with $\mathrm{Dec}(i_2,j_2)$ if $\mathrm{Dec}(i_1,j_1)$ 
is inclusive to $\mathrm{Dec}(i_2,j_2)$.
Similarly, if $\mathrm{Dec}(i_1,j_1)$ is weakly right to $\mathrm{Dec}(i_2.j_2)$, 
we depict the arch associated to $\mathrm{Dec}(i_2,j_2)$ is above the arch associated 
to $\mathrm{Dec}(i_1,j_1)$. 
Here, an arch $a$ is above an arch $b$ means that all the vertices of $a$ are above 
all the vertices of $b$. 

We transform a tetravalent vertex into two lines or a trivalent vertex by reconnecting 
the two arches.
Given two arches with tetravalent vertices, the right-most tetravalent vertex is said 
to be {\it special}.
When one of two arches is weakly right to another arch, 
we reconnect the two arches by changing the special tetravalent vertex into 
two lines. Pictorially, we have  
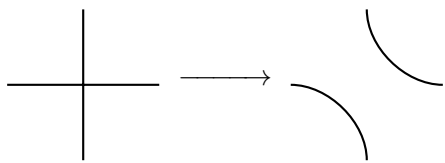
\begin{eqnarray*}
\tikzpic{-0.5}{
\draw[thick](0,0)--(2,0)(1,1)--(1,-1);
}\xrightarrow{\hspace{1cm}}
\tikzpic{-0.5}{
\draw[thick](0,0)..controls(0.5,0)and(1,-0.5)..(1,-1);
\draw[thick](1,1)..controls(1,0.5)and(1.5,0)..(2,0);
}
\end{eqnarray*}
When an arch is inclusive to another arch, we transform 
the special tetravalent vertex (if it exits) into a trivalent 
vertex. Pictorially, we have 
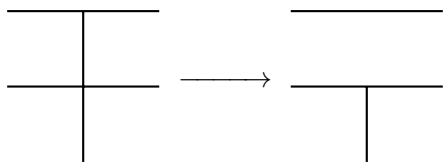
\begin{eqnarray*}
\tikzpic{-0.5}{
\draw[thick](0,0)--(2,0)(1,1)--(1,-1)(0,1)--(2,1);
}\xrightarrow{\hspace{1cm}}
\tikzpic{-0.5}{
\draw[thick](0,1)--(2,1);
\draw[thick](0,0)--(2,0)(1,0)--(1,-1);
}
\end{eqnarray*}
Note that we have a trivalent vertex above the special tetravalent vertex
since an arch below the special vertex is inclusive to an arch above the 
special vertex.
We call this operation on a tetravalent vertex a {\it resolution of a tetravalent vertex}.

We may have several special tetravalent vetices in a skeleton.
We perform the resolution on tetravalent vertices, and arrive at a skeleton 
without tetravalent vertices.
We call a skeleton without tetravalent vertices a {\it fundamental skeleton}.
The following proposition insures the existence of a unique fundamental skeleton
obtained from a skeleton by resolutions.	 
\begin{prop}
Suppose that a skeleton has more than one tetravalent vertices.
If we perform resolutions on the skeleton in any order of choices of special vertices, 
we arrive at the same fundamental skeleton.
\end{prop}
\begin{proof}
To show Proposition, it is enough to show that successive application of resolutions 
of two special vertices does not depend on the order of the choice of these two special vertices.
Let $s_{1}$ and $s_{2}$ be special vertices.
Suppose that $s_{1}$ (resp. $s_{2}$) is the intersection of two arched associated 
with $\mathrm{Dec}(i_1,j_1)$ and $\mathrm{Dec}(i_2,j_2)$ 
(resp. $\mathrm{Dec}(i_3,j_3)$ and $\mathrm{Dec}(i_4,j_4)$).
When all $i_1, i_2, i_3$ and $i_4$ are distinct, it is obvious that the resolutions 
does not depend on the order of choices of special vertices.
Below, we assume that $i_1=i_4$ and $j_1=j_4$ without loss of generality.
We have two special vertices on an arch associated with $\mathrm{Dec}(i_1,j_1)$.
The resolutions at $s_1$ and $s_2$ change a skeleton into another skeleton locally,
which means that the order of choices of special vertices $s_1$ and $s_2$ 
does not effect the new skeleton.
\end{proof}

Figure \ref{fig:res} is an example of resolutions on the skeleton for 
$\sigma=64837251$.
\begin{figure}[ht]
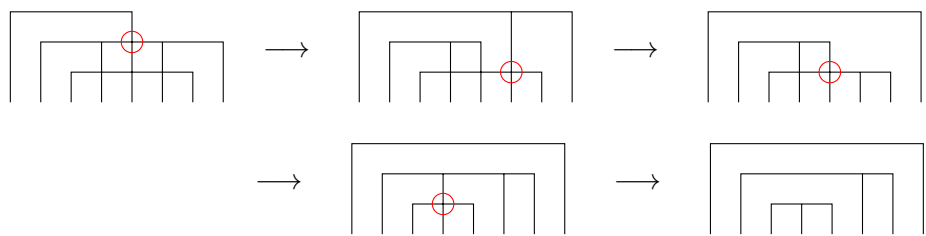

\tikzpic{-0.5}{[x=0.4cm,y=0.4cm]
\draw(0,0)--(0,3)--(4,3)--(4,0)(1,0)--(1,2)--(7,2)--(7,0)(3,2)--(3,0)(5,2)--(5,0)
     (2,0)--(2,1)--(6,1)--(6,0);
\draw(4,2)node[red,anchor=center,circle,inner sep=1mm,draw]{};
}\quad$\longrightarrow$\quad 
\tikzpic{-0.5}{[x=0.4cm,y=0.4cm]
\draw(0,0)--(0,3)--(7,3)--(7,0)(5,3)--(5,0)(1,0)--(1,2)--(4,2)--(4,0)(3,2)--(3,0)
     (2,0)--(2,1)--(6,1)--(6,0);
\draw(5,1)node[red,anchor=center,circle,inner sep=1mm,draw]{};
}\quad$\longrightarrow$\quad 
\tikzpic{-0.5}{[x=0.4cm,y=0.4cm]
\draw(0,0)--(0,3)--(7,3)--(7,0)(1,0)--(1,2)--(4,2)--(4,0)(3,2)--(3,0)
     (2,0)--(2,1)--(6,1)--(6,0)(5,1)--(5,0);
\draw(4,1)node[red,anchor=center,circle,inner sep=1mm,draw]{};
} \\[5mm]
\hspace*{3cm}\quad$\longrightarrow$\quad 
\tikzpic{-0.5}{[x=0.4cm,y=0.4cm]
\draw(0,0)--(0,3)--(7,3)--(7,0)(1,0)--(1,2)--(6,2)--(6,0)(3,2)--(3,0)(5,2)--(5,0)
     (2,0)--(2,1)--(4,1)--(4,0);
\draw(3,1)node[red,anchor=center,circle,inner sep=1mm,draw]{};
}
\quad$\longrightarrow$\quad 
\tikzpic{-0.5}{[x=0.4cm,y=0.4cm]
\draw(0,0)--(0,3)--(7,3)--(7,0)(1,0)--(1,2)--(6,2)--(6,0)(5,2)--(5,0)
     (2,0)--(2,1)--(4,1)--(4,0)(3,1)--(3,0);
}
\caption{An example of resolutions of the skeleton for $\sigma=64837251$.
A tetravalent vertex with a red circle is a special vertex.}
\label{fig:res}
\end{figure}

Once a skeleton $S$ given, one has a permutation $\sigma_{S}$ whose skeleton is $S$.
Note that we may have several such permutations, but at least one permutation for $S$.
\begin{prop}
\label{prop:sketop}
Let $S$ be a skeleton and $\sigma_{S}$ be a corresponding permutation.
We denote by $S'$ a skeleton obtained from $S$ by a resolution.
Then, the top path of $\mathrm{DTR}(\sigma_{S})$ is the same 
as the one of $\mathrm{DTR}(\sigma_{S'})$.
\end{prop}
\begin{proof}
A resolution of a special point means that we locally change the connectivity 
of arches at the special point.
Suppose that $\sigma_{S}(i)=j$.
From Proposition \ref{prop:lbrb}, the top path of $S$ and $S'$ depends only on
the relative positions of $j-1$, $j$ and $j+1$.
The new connectivity of arches after the resolution may change $\sigma_{S}(i)=j$
to $\sigma_{S'}(i)=j'$.
When we have an arch contains $j$, then $j+1$ is to the left of $j$, or $j-1$ is to 
the right of $j$.
Similarly, when an arch contains $j'$, then $j'+1$ is to the left of $j'$, or 
$j'-1$ is to the right of $j'$.
The resolution preserves the relative positions of $j$ and $j+1$, thus those of $j'$ and $j'+1$,
or those of $j$ and $j-1$, thus those of $j'$ and $j'-1$.
Therefore, the top paths of $\mathrm{DTR}(\sigma_{S})$ and $\mathrm{DTR}(\sigma_{S'})$ 
are the same. 
\end{proof}

We construct two permutations for the extreme Dyck tableaux from 
a fundamental skeleton as follows.
We denote by $\sigma_{\vee}$ (resp. $\sigma_{\wedge}$) a permutation 
corresponding to the extreme Dyck tableau $D_{\vee}$ (resp. $D_{\wedge}$).
By definition of fundamental skeletons, an arch is inclusive to 
or strictly right to another arch since we have no tetravalent 
vertices.
Suppose an arch $y$ is right to or inclusive to another arch $x$.
We denote by $x\twoheadrightarrow y$ this relation. 
Then, if a fundamental skeleton consists of arches $\{x_1,\ldots,x_{r}\}$, 
we have a unique chain 
\begin{eqnarray}
\label{eqn:chainskeleton}
x_r\twoheadrightarrow x_{r-1}\twoheadrightarrow \cdots\twoheadrightarrow x_{1}.
\end{eqnarray}
We denote by $|x_{i}|$ the size of the arch $x_{i}$.

\paragraph{\bf Construction of $\sigma_{\vee}$}
Given $x_{i}$ for $1\le i\le r$, we define 
\begin{eqnarray*}
p&:=&\sum_{k=1}^{i-1}|x_{k}|+1, \\
q&:=&\sum_{k=1}^{i}|x_{k}|.
\end{eqnarray*}
The arch $x_{i}$ corresponds to the maximal decreasing sequence 
$\mathrm{Dec}(p,q)$.
Once we have a correspondence between an arch and a maximal decreasing
sequence, we get a permutation $\sigma_{\vee}$ in this way.

\paragraph{\bf Construction of $\sigma_{\wedge}$}
Given $x_{i}$ for $1\le i\le r$, we define 
\begin{eqnarray*}
p'&:=&n+1-\sum_{k=1}^{i}|x_{k}|, \\
q'&:=&n-\sum_{k=1}^{i-1}|x_{k}|.
\end{eqnarray*}
The arch $x_{i}$ corresponds to the maximal decreasing sequence 
$\mathrm{Dec}(p',q')$.
We get a permutation $\sigma_{\wedge}$ in a similar way to $\sigma_{\vee}$.

For example, let $\sigma=64837251$. 
Then, We have $\sigma_{\vee}=25876431$ and $\sigma_{\wedge}=86321547$.

\begin{theorem}
Let $\sigma_{\vee}$ and $\sigma_{\wedge}$ be permutations constructed as above.
Then, we have 
\begin{eqnarray*}
D_{\vee}=\mathrm{DTR}(\sigma_{\vee}), \\
D_{\wedge}=\mathrm{DTR}(\sigma_{\wedge}).
\end{eqnarray*}
\end{theorem}
\begin{proof}
We first show that $D_{\vee}=\mathrm{DTR}(\sigma_{\vee})$.	
From Proposition \ref{prop:sketop}, permutations whose skeleton is a 
fundamental one have the same top path.
Since a fundamental skeleton has a unique chain of arches (see Eqn. (\ref{eqn:chainskeleton})), 
the construction of $\sigma_{\vee}$ ensures that dotted boxes are at 
the maximal floor under the top path.

The same argument for $D_{\wedge}=\mathrm{DTR}(\sigma_{\wedge})$.
\end{proof}

\subsection{The top path of an Hermite history}
Given a decreasing label $L_{\mathrm{dec}}$ of the tree $\mathrm{Tree}(\lambda)$,
one can obtain a Dyck tiling over $\lambda$ through its Hermite 
history.
More precisely, we have a collection of non-negative integers $\mathbf{H}$ (see Section \ref{sec:DTHh}) 
from $L_{\mathrm{dec}}^{\vee}$ (defined by Eqn. (\ref{eqn:decvee})) by reading the labels of $L_{\mathrm{dec}}^{\vee}$ using 
the post-order.
Here, the post-order means that we visit a node after both its left and right subtrees.
Let $\mathbf{h}:=(h_1,h_2,\ldots,h_n)$ be an insertion history for the decreasing 
label $L_{\mathrm{dec}}$, that is, $h_i\in[0,2(i-1)]$ and $h_{i}$ indicates 
the insertion point of the label $n+1-i$ in $L_{\mathrm{dec}}$.
We always have $h_{1}=0$ by its definition.

Let $(k_1,k_2,\ldots,k_{2n})$ be a bi-word consisting 
of $U$ and $D$ of length $2n$ and of an integer sequence in $[1,n]$, that is,
$k_{i}=\begin{bmatrix}X \\ p\end{bmatrix}$ with $X=U$ or $D$ and $p\in[1,n]$.
Then, a Dyck bi-word $\mathbf{K}_{n}:=(k_1,k_2,\ldots,k_{2n})$ is defined as follows.
The first word in the first row of $\mathbf{K}_{n}$ is a Dyck word of $U$'s and $D$'s. 
The second word in the second row of $\mathbf{K}_{n}$ is a parenthesis presentation of a Dyck word, {\it i.e.},
if we replace the first $i$ by $U$ and the second $i$ by $D$ for $i\in[1,n]$, we obtain 
a Dyck word, and a partial word between these $U$ and $D$ is again a Dyck word.

We define 
\begin{eqnarray*}
\mathbf{K}_{1}:=
\begin{bmatrix}
U & D \\
1 & 1
\end{bmatrix}.
\end{eqnarray*}
We recursively construct $\mathbf{K}_{n}$ from $\mathbf{K}_{n-1}$ by using 
the insertion history $\mathbf{h}$.

\begin{enumerate}
\item Increase all integers in the second word of $\mathbf{K}_{n-1}$ by $1$ 
and denote it by $\mathbf{K}^{'}_{n-1}$
\item Find the $h_{n}$-th position in $\mathbf{K}^{'}_{n-1}$ and insert $\mathbf{K}_{1}$ there.
\item If $\begin{matrix}D \\ 2\end{matrix}$ is left to $\begin{matrix}U \\ 1\end{matrix}$,
we move $\begin{matrix}U \\ 1\end{matrix}$ left to $\begin{matrix}D \\ 2\end{matrix}$.
As a sequence, we have 
\begin{eqnarray*}
\cdots
\begin{matrix}
U & D \\
2 & 2
\end{matrix}
\cdots
\begin{matrix}
U & D \\
1 & 1
\end{matrix}\cdots
\longrightarrow
\cdots
\begin{matrix}
U & U & D\\
2 & 1 & 2
\end{matrix}
\cdots
\begin{matrix}
 D \\
 1
\end{matrix}\cdots
\end{eqnarray*}
where the dotted parts are unchanged.
\item We hold the integer label for $U$, and change the integer labels for $D$ such
that a new integer sequence is a parenthesis presentation of a Dyck word.
\end{enumerate}

\begin{example}
We consider $\mathbf{h}=(0,2,3,1,5,2)$. Then, we have 
\begin{eqnarray*}
&&\mathbf{K}_{1}=
\begin{bmatrix}
U & D \\
1 & 1
\end{bmatrix},
\quad
\mathbf{K}_{2}=
\begin{bmatrix}
U & U & D & D \\
2 & 1 & 1 & 2
\end{bmatrix},
\quad
\mathbf{K}_{3}=
\begin{bmatrix}
U & U & U & D & D & D \\
3 & 2 & 1 & 1 & 2 & 3
\end{bmatrix},\\
&&\mathbf{K}_{4}=
\begin{bmatrix}
U & U & D & U & U & D & D & D\\
4 & 1 & 1 & 3 & 2 & 2 & 3 & 4
\end{bmatrix},
\quad
\mathbf{K}_{5}=
\begin{bmatrix}
U & U & U & D & U & U & D & D & D & D \\
5 & 2 & 1 & 1 & 4 & 3 & 3 & 4 & 2 & 5
\end{bmatrix},\\ 
\quad
&&\mathbf{K}_{6}=
\begin{bmatrix}
U & U & U & D & U & D & U & U & D & D & D & D\\
6 & 3 & 1 & 1 & 2 & 2 & 5 & 4 & 4 & 5 & 3 & 6
\end{bmatrix},
\end{eqnarray*}
The first word $\mathbf{K}_{6}$ is the top path of a Dyck 
tiling obtained from an Hermite history (see Figure \ref{fig:Hhtop}).
\begin{figure}[ht]
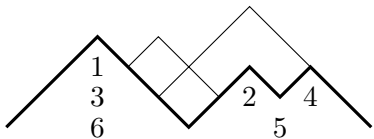

\tikzpic{-0.5}{[x=0.4cm,y=0.4cm]
\draw[very thick](0,0)--(3,3)--(6,0)--(8,2)--(9,1)--(10,2)--(12,0);
\draw(4,2)--(5,3)--(7,1)(5,1)--(8,4)--(10,2);
\draw(3,2)node{$1$}(3,1)node{$3$}(3,0)node{$6$};
\draw(8,1)node{$2$}(10,1)node{$4$}(9,0)node{$5$};
}
\caption{A Dyck tiling associated with an insertion history $\mathbf{h}=(0,2,3,1,5,2)$
for a decreasing label $L_{\mathrm{dec}}$ where the post-order word of $L_{\mathrm{dec}}$ 
is $136245$.}
\label{fig:Hhtop}
\end{figure}
\end{example}

\begin{prop}
Let $L_{\mathrm{dec}}$ be a decreasing label, $\mathbf{h}$ be its insertion history, and 
$\mathbf{K}_{n}$ be a Dyck bi-word associated with $\mathbf{h}$.
Then, the top path of the Hermite history of $L_{\mathrm{dec}}$ is the first Dyck word 
of $\mathbf{K}_{n}$.
\end{prop}
\begin{proof}
Let $p\in[1,n]$ be an integer such that the first words of $k_{1},\ldots,k_{p}$ are 
$U$ and the first word of $k_{p+1}$ is $D$ in $\mathbf{K}_{n-1}$. 
We first show that the second words of $k_{p}$ and $k_{p+1}$ are $1$ by induction.
It is obvious when $n=1$, and we assume that the statement holds up to $n-1$.
We construct $\mathbf{K}_{n}$ from $\mathbf{K}_{n-1}$ by the insertion history $h_{n}$. 
If $h_{n}\le p$, we insert a word $UD$ into the $p$-th position in $\mathbf{K}_{n-1}$, 
and statement is true.
If $h_{n}>p$, the step (3) ensures that the first words of $k_{1},\ldots,k_{p+1}$ are 
$U$, the first word of $k_{p+2}$ is $D$ and the second words of $k_{p+1}$ and $k_{p+2}$ 
are $1$ in $\mathbf{K}_{n}$.
Thus, the statement holds for $n$.

When $k=1$, the first word of $\mathbf{K}_{1}$ is the top path of the Hermite history 
$\mathbf{H}=(0)$ with the insertion history $\mathbf{h}=(0)$.
We prove Theorem by induction on $n$ and assume that Theorem holds up to $n-1$.
Let $\lambda$ be the Dyck path associated with the insertion history $\mathbf{h}$.
We insert the bi-word $\mathbf{K}_{1}$ at the $h_{n}$-th position in $\mathbf{K}_{n-1}$,
which means that we insert a $UD$ path in $\lambda$ and 
it corresponds to the edge labeled by $1$ in $L_{\mathrm{dec}}$.
Since we consider an Hermite history, the top path right to the $h_{n}+2$-th position 
in the Dyck tiling of size $n$ is the same as the top path right to the $h_{n}$-th 
position in the Dyck tiling of size $n-1$. 
In $L_{\mathrm{dec}}$, the label $1$ is on an edge connected to a leaf and the smallest integer.
We have to increase by one the statistics $\mathrm{art}$ for a trajectory starting 
from a $D$ step left to the $h_{n}$-th position.
This is realized by the step (3).
Thus, the top path of the Hermite history associated with $\mathbf{h}$ is 
the first word of $\mathbf{K}_{n}$.
\end{proof}

\bibliographystyle{amsplainhyper} 
\newcommand{\doi}[1]{https://doi.org/#1}
\bibliography{biblio}

\end{document}